\documentclass[11pt]{amsart}
\usepackage{amssymb}
\usepackage{hyperref}
\usepackage{amsmath,amsthm,amsxtra,mathtools}
\usepackage{eucal,mathrsfs}
\usepackage{enumitem}
\usepackage{a4wide}
\usepackage{xcolor}
\usepackage{soul}
\usepackage{comment}
\usepackage{bbm,stix}
\usepackage{subcaption}
\usepackage{wrapfig}
\usepackage{multicol}

\usepackage{tikz}
\usetikzlibrary{cd}

\newcommand{\N}{\mathbb{N}}

\newcommand{\R}{\mathbb{R}}
\newcommand{\Z}{\mathbb{Z}}
\newcommand{\Ind}{\mathbbm{1}}
\newcommand{\Id}{\text{Id}}

\newcommand{\calB}{\mathcal{B}}

\newcommand{\calF}{\mathcal{F}}
\newcommand{\calG}{\mathcal{G}}
\newcommand{\calH}{\mathcal{H}}

\newcommand{\calJ}{\mathcal{J}}

\newcommand{\calL}{\mathscr{L}}
\newcommand{\calM}{\mathcal{M}}

\newcommand{\calP}{\mathcal{P}}

\newcommand{\calS}{\mathcal{S}}
\newcommand{\calT}{\mathcal{T}}

\newcommand{\calW}{\mathcal{W}}

\newcommand{\spt}{\mathrm{spt}}
\newcommand{\Lip}{\mathrm{Lip}}

\newcommand{\Fcm}{\calF_{\text{CM}}}
\newcommand{\Scm}{\calS_{\text{CM}}}

\newcommand{\vh}{\mathbf{h}}

\newcommand{\vdh}{\mathbf{d}_h}

\newcommand{\dnabla}{\overline{\nabla}}

\newcommand{\dd}[1]{\mathop{}\!\mathrm{d} #1}

\newtheorem{theorem}{Theorem}[section]

\newtheorem{cor}[theorem]{Corollary}
\newtheorem{lemma}[theorem]{Lemma}
\newtheorem{proposition}[theorem]{Proposition}

\theoremstyle{definition}
\newtheorem{definition}[theorem]{Definition}
\newtheorem{remark}[theorem]{Remark}
\newtheorem{example}[theorem]{Example}

\numberwithin{equation}{section}

\title{Convergence of the fully discrete JKO scheme}

\begin{document}

\author{Anastasiia Hraivoronska}
\author{Filippo Santambrogio}

\begin{abstract}
    The JKO scheme provides the discrete-in-time approximation for the solutions of evolutionary equations with Wasserstein gradient structure. We study a natural space-discretization of this scheme by restricting the minimization to the measures supported on the nodes of a regular grid. The study of the fully discrete JKO scheme is motivated by the applications to developing numerical schemes for the nonlinear diffusion equation with drift and the crowd motion model. The main result of this paper is the convergence of the scheme as both the time and space discretization parameters tend to zero in a suitable regime. 
\end{abstract}

\maketitle
\tableofcontents

\section{Introduction}

The theory of gradient flows in the space of probability measures provides powerful tools for studying a wide class of evolutionary equations \cite{ambrosio2008gradient}. The JKO scheme (or the minimizing movement scheme) plays an essential role in this theory, acting as a counterpart for the implicit Euler scheme for classical gradient-flow ODEs. It first appeared in Jordan, Kinderleherer, and Otto's paper \cite{jordan1998variational}, where they recognized the Fokker-Planck dynamics as a gradient flow of the Boltzmann entropy in the space of probability measures with an appropriate metric. The JKO scheme provides as an output a sequence of probability measures (see below), each one optimizing a variational problem involving a transport cost to the previous one and an energy functional (the Boltzmann entropy in the case of Fokker-Planck), which converges to the solution of the desired PDE when a time step parameter $\tau$ tends to $0$. As such, it is a means to study well-posedness and to develop numerical schemes for the equations of interest.

In this paper, we consider the most natural space-discretization of this scheme: we consider a uniform and regular grid and restrict our analysis to measures supported on the nodes of this grid. While the optimal transport problem and the Wasserstein distance that it defines are well-posed for atomic measures, our approach requires adapting the definition of some energy functionals to the same atomic measures, but this is very natural when we consider measures on a regular and fixed grid. We need hence to introduce another approximation parameter $h$, corresponding to the space step of the grid, and we study the convergence of the scheme when $\tau,h\to 0$. We identify for this sake a sharp condition, namely $h/\tau\to 0$, which is both necessary and sufficient to prove this convergence. Note that this condition is the opposite of a very standard one in numerical analysis, called the CFL condition, where instead $\tau/h$ is required to be small (called the CFL condition). This is not surprising as the CFL condition is standard in explicit methods, while the JKO scheme is implicit.

\subsection{The JKO scheme and the evolution PDEs that it approximates}
The JKO scheme is a discrete-in-time iterative scheme defined in the space of probability measures $\calP(\Omega)$ (in this paper $\Omega\subset\R^d$ is a bounded domain):
\begin{equation}\label{eq:JKO-intro}\tag{JKO}
    \rho^\tau_{k+1} \in \arg \min_{\rho\in\calP(\Omega)} \Big\{ \calF(\rho) + \frac{1}{2\tau} W_2^2(\rho, \rho^\tau_k)\Big\}, \qquad \text{with given } \rho^\tau_0 = \overline{\rho} \in \calP(\Omega),
\end{equation}
where $\tau > 0$ is a time step parameter, $\calF: \calP(\Omega) \to \R\cup\{+\infty\}$ is an energy functional, and $W_2$ is the Wasserstein distance induced by the optimal transport problem with the quadratic cost
\begin{equation}\label{eq:W2}
    W_2^2(\mu, \nu) = \min_{\gamma\in\Gamma(\mu,\nu)} \iint |x-y|^2 \gamma (\dd x \dd y),
\end{equation}
where $\Gamma(\mu,\nu) = \{\gamma\in\calP(\Omega\times\Omega) : P^1_{\#}\gamma = \mu,  P^2_{\#}\gamma = \nu \}$ is the set of couplings between the marginals $\mu$ and $\nu$.
The sequence of minimizers $\{\rho^\tau_k\}$ of \eqref{eq:JKO-intro} can be used to build a continuous-in-time curve $\rho^\tau$, for instance, with piecewise constant interpolation
\begin{equation*}
    \rho^\tau(t) \coloneq \rho^\tau_{k+1} \qquad \text{for } t\in (k\tau, (k+1)\tau].
\end{equation*}
Under suitable assumption on the energy $\calF$, one can prove that $\rho^\tau$ converges to a weak solution of 
\begin{equation}\label{eq:GF-PDE}
    \partial_t \rho + \nabla \cdot \Big( \rho \nabla \frac{\delta \calF}{\delta \rho} \Big) = 0 \qquad \text{ in } (0,T) \times \Omega,
\end{equation}
with the initial datum $\rho_0 = \Bar{\rho}$ and complemented with the non-flux boundary conditions on $\partial \Omega$, where $\delta \calF / \delta \rho$ denotes the first variation of $\calF$.

After the seminal work \cite{jordan1998variational} for the linear Fokker-Planck equation, the JKO scheme has been proven to be a useful analytical tool for many classes of PDEs \cite{ambrosio2008gradient}, including general aggregation-diffusion equationsn the Keller-Segel model \cite{blanchet2008convergence}, cross-diffusion systems \cite{carrillo2018splitting, ducasse2023cross, kim2018nonlinear}, and higher-order equations \cite{carlier2019total, matthes2009family}. 

In this paper, we consider the energy functionals that are the sum of internal and potential energy
\begin{equation}\label{eq:energy-intro}
    \calF(\rho) = \int_\Omega f(\rho) + \int_\Omega \rho V,
\end{equation}
with a convex energy density $f: (0, \infty) \to \R \cup \{+\infty\}$ and an external potential $V: \Omega \to \R$. This energy gives rise to the drift-diffusion equations of the form
\begin{equation}\label{eq:drift-diffusion}
    \partial_t \rho = \Delta L(\rho) + \nabla \cdot (\rho \nabla V),
\end{equation}
where the diffusion term is defined by $L'(s) = s f''(s)$.

We provide a few examples that we have in mind as test cases for the results of this paper. The most well-known PDE of the type \eqref{eq:drift-diffusion} is the linear Fokker-Planck equation corresponding to $f(s) = s \log s$ and $L(s) = s$. The second example is the porous media equation with drift, where the energy density is the power law $f_m(s) = s^m / (m-1)$ and $L_m(s) = s^m$, $m > 1$. We use these two cases as running examples, but we prove the results for more general convex superlinear $f$ (the exact assumptions are specified in Section~\ref{sec:main-results}) and $V\in \Lip(\R^d)$.

The Hele-Shaw-type system is the third example that plays a special role in this paper:
\begin{equation}\label{eq:hele-shaw-intro}
    \begin{cases}
        \partial_t \rho = \nabla p + \nabla\cdot(\rho \nabla V), \\
        p \geq 0, \quad \rho \leq 1, \quad p (1 - \rho) = 0,
    \end{cases}
\end{equation}
where we introduced the pressure variable $p$. It emerges in the incompressible limit of the porous media equation, i.e., $m\to\infty$ \cite{gil2001convergence}. The gradient flow interpretation for \eqref{eq:hele-shaw-intro} appears in \cite{maury2010macroscopic, maury2011handling} with application to the crowd motion model. The idea is that the solutions of \eqref{eq:hele-shaw-intro} can be obtained by means of the JKO scheme \eqref{eq:JKO-intro} choosing the energy functional
\begin{equation*}
    \calF_\text{CM}(\rho) = \begin{cases}
        \int_\Omega \rho V, \quad \rho \leq 1, \\
        +\infty, \qquad \text{otherwise.}
    \end{cases}
\end{equation*}
One can see that $\Fcm$ has a similar structure as \eqref{eq:energy-intro} with the degenerate energy density $f(s) = 0$ for $s\in [0, 1]$ and $f(s) = + \infty$ for $s > 1$, which is exactly the pointwise limit of the porous-media energy density $f_m$ as $m \to \infty$. 

\subsection{Contribution}
In this paper, we explore a fully discrete JKO scheme which, to our knowledge, has not been studied before. We restrict the minimization at each step of the JKO scheme to the set of probability measures supported on the vertices of the grid $h\Z^d$ with a small space discretization parameter $h>0$. Precisely, we define the discrete domain $\calT^h \coloneq h\Z^d\cap\Omega$ and the set of corresponding probability measures
$$
    \calP(\calT^h) \coloneq \Big\{ \rho \in \calP(\Omega) : \rho = \sum_{z\in\calT^h} \rho_z \delta_z \Big\}.
$$
Then the fully discrete JKO scheme is defined as
\begin{equation}\label{eq:JKOd-intro}\tag{JKO$^{h,\tau}$}
    \rho^{h,\tau}_{k+1} \in \arg \min_{\rho\in\calP(\calT^h)} \Big\{ \calF^h(\rho) + \frac{1}{2\tau} W_2^2(\rho, \rho^{h,\tau}_k)\Big\}, \qquad \text{with given } \rho^{h,\tau}_0 \in \calP(\calT^h),
\end{equation}
where $\calF^h : \calP(\calT^h) \to \R\cup\{+\infty\}$ is the discrete counterpart of the energy $\calF$ from \eqref{eq:energy-intro} defined as 
\begin{equation}\label{eq:energy-disc-intro}
    \calF^h(\rho) = \sum_{z\in \calT^h} f \Big( \frac{\rho_z}{h^d} \Big) h^d + \sum_{z\in \calT^h} V(z) \rho_z.
\end{equation}
For the crowd motion case, we similarly define
\begin{equation}\label{eq:energy-CM-disc-intro}
    \Fcm^h (\rho^h) = \begin{cases}
        \displaystyle \sum_{z\in \calT^h} V(z) \rho_z, & \frac{\rho_z}{h^d} \leq 1 \text{ for every }z, \\
        +\infty, & \text{otherwise.}
    \end{cases} 
\end{equation}
We stress that we use the standard Wasserstein distance in the discrete scheme \eqref{eq:JKOd-intro} differently from other works where the optimal transport model itself is modified in order to take care of the discrete structure of the measures (see for instance \cite{Maas}, where the distance is defined in connection with a Markov chain on a graph).

As we will explain better below, we introduce the fully discrete JKO scheme having applications for numerical schemes in mind. In this sense, the crowd motion model is a particularly interesting case, because \eqref{eq:JKOd-intro} becomes a linear optimization problem:
\begin{equation*}
    \rho^{h,\tau}_{k+1} = P^1_{\#}\gamma^{h,\tau}_k, \quad \gamma^{h,\tau}_k \in \arg \min_{\gamma\in \calP(\calT^h\times\calT^h)} \bigg\{ \iint \Big\{ V(x) + \frac{|x-y|^2}{2\tau} \Big\} \gamma (\dd x \dd y), \quad P^1_{\#} \gamma \leq 1, ~~  P^2_{\#} \gamma = \rho^\tau_k \Big\}.
\end{equation*}
In the case of the drift-diffusion equation, \eqref{eq:JKOd-intro} is a nonlinear convex optimization problem. We will also mention thoughout the paper how to consider the case of interaction potentials (corresponding to a drift given by $-\nabla W*\rho$, or to an energy functional including a term of the form $\int\int W(x-y)\dd\rho(x)\dd\rho(y)$), which is a small variant in terms of PDEs, but in terms of optimization it is less interesting, as it could give raise to non-convex terms in the optimization problem. For this reason, the interaction case is only mentioned as a remark and is not the main scope of our contribution.

The goal of this paper is to establish the convergence result for \eqref{eq:JKOd-intro} in the joint limit $h,\tau \to 0$. The relation between $h$ and $\tau$ plays a crucial role in the convergence. To see how this relation affects the convergence, we consider a toy example with the movement driven only by an external potential. Specifically, let the energy functional with $V\in C^{1,1}(\R^d)$ be defined as:
\begin{equation}\label{eq:potential-energy-intro}
    \calF^h_V(\rho^h) \coloneq \sum_{z\in\calT^h} V(z) \rho^h_z.
\end{equation}
In this case, the Dirac masses in the initial datum move independently of each other. Consider the movement of $\delta_{x_0}$ with $x_0 \in \calT^h$. In the continuous setting, it simply means that particles move along the trajectories of the ODE $\Dot{x} = - (\nabla V)(x)$ with $x(0) = x_0$. In the discrete setting, we look for a position $x_1\in\calT^h$ minimizing $V(x) + |x - x_0|^2 / 2\tau$. Using $x_0$ as a competitor for the minimizer, we obtain
    $$
        V(x_0) \geq V(x_1) + \frac{|x_1 - x_0|^2}{2 \tau}.
    $$
    Since the minimal distance that each particle has to travel is $h$, a necessary condition in order to observe some movement is  $h/\tau \leq 2 \|\nabla V\|_{\text{Lip}}$. We see that if asymptotically $h/\tau > 2 \|\nabla V\|_{\text{Lip}}$, then we cannot expect convergence to the continuous solution, because every subsequent minimizer is equal to $x_0$ and the evolution is "frozen".

    We will further see that, for convergence, we need an even stronger assumption $h/\tau \to 0$, and this assumption is sharp. If $h \sim \tau$, we cannot guarantee the convergence to the solution because the discrete grid can introduce additional drift to the evolution. The example with the potential movement is detailed in Section~\ref{sec:potential-movement}.

Strikingly enough, the condition $h/\tau\to 0$ is indeed also sufficient to prove the convergence when the functional $\mathcal F$ also includes a term depending on the density, which means when the corresponding PDE has diffusion terms. Our main results are thus the convergence statements for \eqref{eq:JKOd-intro} to the drift-diffusion equation and the crowd motion model. Let $\{\rho^{h,\tau}_k\}$ be the family of minimizers for \eqref{eq:JKOd-intro} with $\calF^h$ defined in \eqref{eq:energy-disc-intro} and let $\rho^{h,\tau}$ be the piecewise constant-in-time interpolation of $\{\rho^{h,\tau}_k\}$. Then the curve $\rho^{h,\tau}$ converges (up to a subsequence) as $h,\tau, h/\tau \to 0$ to a continuous curve $\rho \in C([0,T], \calP(\Omega))$, which is a weak solution to \eqref{eq:drift-diffusion}. An analogous result holds for the crowd motion energy defined in \eqref{eq:energy-CM-disc-intro} with the convergence to \eqref{eq:hele-shaw-intro}. The precise statements of the main results are given in Section~\ref{sec:main-results}. The strategy of the proof is based on the so-called {\it Energy Dissipation Inequality}, which characterizes the solutions of the PDE as those curves satisfying a certain (sharp) inequality on the rate of dissipation of the energy $\mathcal F$. This inequality will be obtained by passing to the limit similar inequalities obtained on the discrete problems thanks to the optimality conditions, and carefully mixing different forms of interpolations. Indeed, a variational interpolation {\it \`a la} De Giorgi will be used to make a slope term appear (see Section~\ref{sec:nonlinear-diffusion}.), and this interpolation will be composed of measures also supported on the grid, but a gedoesic interpolation (in the sense of geodesic in the Wasserstein space on the whole space, not only on the grid) will be used to handle a velocity term.

\subsection{Related work}
The JKO scheme has been actively employed to develop numerical schemes for the evolution equations with the gradient structure. From the theoretical point of view, it has advantages such as a variational formulation, providing gradient-structure preserving and stability properties. The difficulty is finding an efficient approach to discretizing the Wasserstein distance. The existing methods rely on various reformulations of the optimal transport problem, such as dynamical formulation by Benamou-Brenier \cite{benamou2000computational}, entropic regularization \cite{cuturi2013sinkhorn}, and semi-discrete formulation \cite{kitagawa2019convergence, merigot2011multiscale}. We briefly review the results in the literature on numerical schemes based on the JKO scheme.
 
The idea of the Benamou-Brenier formulation is that the Wasserstein distance $W_2$ admits an alternative formulation
\begin{equation}\label{eq:BB-W2}
    W_2^2(\mu, \nu) = \min \bigg\{ \frac{1}{2} \int_0^1 \int_\Omega \bigg| \frac{\dd j_t}{\dd \rho_t} \bigg|^2 \dd \rho_t \dd t, \quad \partial_t \rho + \text{div} j = 0, \quad \rho_0 = \mu, \rho_1 = \nu \bigg\}.
\end{equation}
Here, to compute $W_2$, one considers all absolutely continuous curves $[0,1]\ni t \mapsto \rho_t$ connecting $\mu$ and $\nu$ and minimizes the corresponding integral of the velocity field $v_t = \dd j_t/\dd \rho_t$. Relying on this formulation, they implement the augmented Lagrangian method in \cite{benamou2016augmented} and the successive primal-dual method in \cite{carrillo2022primal} for the numerical approximation of \eqref{eq:JKO-intro}. A disadvantage of this approach is that solving the system of optimality conditions requires inner time stepping. As an alternative, which avoids the sub-discretization, several works 
use the linearization of the Wasserstein distance \cite{cances2020variational, li2020fisher, natale2020tpfa}. The main idea is that the subsequent minimizers $\rho^\tau_k$ and $\rho^\tau_{k+1}$ of \eqref{eq:JKO-intro} are "close" to each other, and it is reasonable to replace the $W_2$ distance with a suitable weighted $H^{-1}$ norm. 

Another approach to computing the Wasserstein distance relies on modifying it with the entropic regularization term because the regularized version can be computed efficiently with the Sinkhorn algorithm \cite{cuturi2013sinkhorn}. The entropic regularization is applied to the JKO scheme in \cite{carlier2017convergence}, where the authors prove the convergence of the regularized scheme in the joint limit $\tau, \varepsilon \to 0$ with $\varepsilon/\tau \to 0$, where $\varepsilon > 0$ is the regularization parameter. 

In \cite{benamou2016discretization}, the authors treat the JKO scheme as a variational problem on the space of convex functions with the Monge-Amp\`ere operator, for which they construct a discretization. 

The crowd motion model corresponding to the constraint $\rho\leq 1$ has also been treated with numerical methods inspired by the JKO scheme, often using variants of the same methods as for non-linear diffusion (see, in particular, \cite{benamou2016augmented} or \cite{carlier2017convergence}). In the original works on the gradient flow interpretation of this crowd motion model (\cite{maury2010macroscopic,maury2011handling}) a different specific approach was used in order to compute the projection of a measure onto the set of admissible measures satisfying $\rho\leq 1$ (and the drift was treated via a splitting method, i.e. first advecting the density ignoring the constraint, and then projecting). The projection was approximated via a stochastic procedure reminiscent of {\it Diffusion Limited Aggregation}, which was well efficient in practice but a convergence proof was out of reach.

Finally, we cite the works in \cite{leclerc2020lagrangian} and \cite{CarPatCra}, where both linear diffusion and crowd motion models are addressed. In spirit this is very different from what we do, since it is based on a discretization in ``space'' (replacing densities with moving particles) but not in time, while in our approach the starting point is the JKO scheme, which is a discretization in time, which we complete with a space discretization. Yet, a common point with our work is that they require a suitable extension of some energy functionals (such as the Boltzmann entropy or the constraint on the density) to atomic measures. Here we see how our work really chooses the simplest and most natural point of view: having a fixed grid allows us without ambiguity to spread the mass of an atom over the cell around it, while when particles move this is more delicate. This is done via a Moreau-Yosida approximation of the functional in \cite{leclerc2020lagrangian} and via convolution in \cite{CarPatCra}.

\subsection{Numerical challenges and applications}\label{sec:intro-numerics}
Even if this paper is not devoted to numerical computations, but only to the proof of convergence of the method, we would like to explain here the interest of our approach for numerics. If the original JKO scheme is a way to approximate the solutions of a parabolic PDEs via a sequence of variational problems, the fully discrete JKO scheme that we propose is an effective way to approximate it via a sequence of finite-dimensional variational problems. In the examples that we analyze in detail, these optimization problems are convex. The unknown in the problems can be the new measure $\rho^\tau_{k+1}$ or the transport plan $\gamma$ between it and the previous one. If the number of points in the space discretization is $N=O(h^{-d})$, then the problem is set on a space $\R^N$ (with positivity and mass constraints) in the first case, in a space $\R^{N\times N)}$ in the second. On the other hand, the function to be minimized is more explicit when expressed in terms of the transport plan as it includes a linear part and a non-linear one depending on its projection onto the space $\R^N$. In the particular case of the crowd motion model, the whole problem is linear in $\gamma$ and only subject to linear equality or inequality constraints, hence it is a typical linear programming problem (LP). This LP problem (or the one appearing as a part of the problem when $f$ is nonlinear) can be solved by the simplex method, but this can be very slow, or better by the network simplex, which has been implemented in a very efficient and fast way in \cite{NetworkBonneel,BPPH11}. However, since our main result states that one should choose $h/\tau$ very small (because of the reverse CFL condition), the very large number of points in the space discretization can be an obstacle in the practical implementation of the method. 

Luckily, there exist at least two ways to decrease the complexity of the problem by reducing the number of unknowns of the problem. The first one is a clever idea from \cite{AurBasGuaVen}, which exploits the separability of the quadratic cost in terms of costs on each coordinate direction and transform this into an LP problem with $N^{1+1/d}=O(h^{-(d+1)})$ unknowns instead of $N^2=O(h^{-2d})$. Note that this requires the use of a regular grid for the discretization, which we do anyway because it is also needed in some of our proofs. If some parts of the proofs could be adapted to other discretizations, some others are really based on the crucial role played by the coordinate axis. We also mention the fact that this requirement to choose a regular grid prevents us from easily considering local refinements of the mesh, which is instead a very common feature of many efficient numerical approaches to solve parabolic PDEs.

The second improvement is based on the idea that, when $\tau$ is small, the displacement between $\rho^\tau_{k}$ and $\rho^\tau_{k+1}$ will be small. If one can prove an $L^\infty$ bound on the displacement $|x-y|$ in the support of the optimal plan $\gamma$, then $\gamma$ will be concentrated on a neighborhood of the diagonal and some points in the product space can be removed from the unknown. If a bound of the form $|x-y|\leq C\tau^\alpha$ can be proven (ideally with $\alpha=1$), the number of unknowns obtained thanks to \cite{AurBasGuaVen} goes down to $\tau^\alpha h^{-(d+1)}$. A very general method to prove this kind of $L^\infty$ bounds is contained in \cite{BJR2007}, and the proof therein can be easily adapted to the case of a grid. This unfortunately provides an exponent $\alpha<1$ while, with very different methods (PDE-based and difficult to adapt to the case of a grid), at least in the case of linear diffusion, one can obtain $\alpha=1$ (see \cite{FerSan}). Notice however that in the crowd motion case the bound with $\alpha=1$ is false, as one can see from the sharp estimate obtained in \cite{DavPer} where the velocity is proven to belong to $L^4$ in space and time, and not better, while $|x-y|\leq C\tau$ would imply an $L^\infty$ bound on the velocity.

We also mention the fact that the discrete transport problem that we need to solve could also be solved via the Sinkhorn algorithm based on entropic regularization, which would insert an extra parameter $\varepsilon$ to be handled. The recent result in \cite{BarHraSan} explains how to use Sinkhorn without taking $\varepsilon/\tau \to 0$ if one wants to approximate an equation already including linear diffusion. However, we do not consider in this paper the combination of the two methods, i.e. Sinkhorn on a grid with a proof of convergence.

All in all, the effective implementation of the approximation suggested in the present paper opens a certain number of challenging questions both in terms of estimates and of complexity. The convergence rate in terms of $h$ and $\tau$ should also be considered for a rigorous comparison to other methods. In this comparison, this fully discrete JKO scheme is for sure a costly approach due to the smallness of the parameter $h$, but it has the advantage of not suffering from instability or slower convergence in the regions where the density is small, and it works in the same way for linear or nonlinear diffusion. It is, of course, a bad choice to approximate the linear Fokker-Planck equation, but it can be competitive for nonlinear equations.

\bigskip
\textbf{Outline.} The rest of the paper is organized as follows. In Section~\ref{sec:prelim-main}, we present the setting of the problem in full detail, discuss the preliminaries, and formulate the main results. Section~\ref{sec:potential-movement} is dedicated to the toy example with pure potential energy, where we motivate the necessity of the assumption $h/\tau \to 0$. The convergence result for the fully discrete JKO scheme for the nonlinear drift-diffusion equation is proven in Section~\ref{sec:nonlinear-diffusion}. Finally, we adapt the proof of the convergence of the scheme to the crowd motion case in Section~\ref{sec:crowd-motion}. Note that the precise assumptions on the integrand $f$ and on the potential $V$ defining the energy $\mathcal F$ will only be given at the beginning of Section ~\ref{sec:nonlinear-diffusion}.

\bigskip

\bigskip
\noindent {\bf Acknowledgments.} The authors acknowledge the support of the European Union via the ERC AdG 101054420 EYAWKAJKOS. 

\section{Setting of the problem and statement of the main results}\label{sec:prelim-main}

\subsection{Continuous models} In this section, we discuss two PDEs that we expect to obtain in the discrete-to-continuum limit: the nonlinear drift-diffusion equation and the crowd motion model. In both cases, we need the characterization of the weak solutions by the energy-dissipation inequality. 

Let $\Omega\subset\R^d$ be a bounded domain, either convex or with a smooth boundary.



\subsubsection{Nonlinear diffusion equation} Consider the nonlinear diffusion equation with drift
\begin{equation}\label{eq:nonlinear-diff-pde}
    \partial_t \rho + \mathrm{div} \big( \rho \nabla (f'(\rho) + V) \big) = 0, \qquad \text{in } (0,T)\times\Omega
\end{equation}
complemented with the non-flux boundary condition
$$
    \partial_\nu L(\rho) + \rho \partial_\nu V = 0 \qquad \text{on } \partial \Omega,
$$
where the function $L$ is defined by $L'(s) = s f''(s)$ and $\nu$ is the outer normal vector on $\partial \Omega$. The free energy functional for this equation is
\begin{equation}\label{eq:general-energy}
    \calF(\rho) = \begin{cases}
        \int_\Omega f(u) \dd \calL^d + \int_\Omega V \dd \rho \quad \text{if } \rho = u\calL^d\\
        +\infty \hspace{3.5cm} \text{otherwise.}
    \end{cases}
\end{equation}

An essential tool for the convergence results is the characterization of solutions by the energy-dissipation balance. We first define the Fisher information as follows.
\begin{definition}[Fisher information]\label{def:Fisher} Let the energy be defined as in \eqref{eq:general-energy}. We defined the corresponding Fisher information as the functional $\calS: \calP(\Omega) \to [0, +\infty]$:
    \begin{equation}
        \calS(\rho) = \begin{cases}
            \frac{1}{2} \int_\Omega |\nabla \ell(u) + \sqrt{u} \nabla V|^2 \dd \calL^d  & \text{if } \rho = u \calL^d \text{ and } \ell(u) \in H^1(\Omega), \\
            +\infty &  \text{otherwise},
        \end{cases}
    \end{equation}
    where $\ell: (0,\infty) \to \R$ is a function defined by $\ell'(s) = \sqrt{s} f''(s)$ for $s\in (0,\infty)$. Note that this definition is algebraically equivalent to $ \frac{1}{2} \int_\Omega |\nabla (f'(u)+ V)|^2 \dd \rho$ (which also equals $\frac{1}{2} \int_\Omega |\nabla \frac{\delta\mathcal F}{\delta\rho}|^2 \dd \rho$), but we prefer for the rest of our analysis to use quantities which are not weighted with $\rho$.
\end{definition}

\begin{example} As mentioned in the introduction, we use the Fokker-Planck and porous media equations as running examples. For the Fokker-Planck equation, a simple calculation gives $\ell'(s) = s^{-1/2}$ and $\ell(s) = 2 \sqrt{s}$. Thus, the Fisher information for $\rho = u \calL^d$ is
\begin{equation*}
    \calS(\rho) = \int_\Omega |2\nabla\sqrt{u} + \sqrt{u }\nabla V |^2 \dd \calL^d
\end{equation*}.

For the porous media equation, where $f(s) = s^m / (m - 1)$ with $m > 1$ and $V = 0$. In this case, $\ell'(s) = m s^{m-3/2}$ and $\ell(s) = m s^{m - 1/2} / (m - 1/2)$. The Fisher information for $\rho = u \calL^d$ becomes
\begin{equation*}
    \calS(\rho) = \frac{m^2}{(m - 1/2)^2} \int_\Omega \big| \nabla u^{m-1/2} \big|^2 \dd\calL^d.
\end{equation*}
\end{example}

\begin{proposition}[Characterization of solutions]\label{prop:charact}
    Let a pair $(\rho, v)$ satisfy the continuity equation
    \begin{equation}\label{eq:CE}
        \partial_t \rho + \mathrm{div} (\rho v) = 0 \qquad \text{in } (0,T)\times\Omega
    \end{equation}
    in the distributional sense with the no-flux boundary condition. Further, let $(\rho, v)$ satisfy the EDI:
    \begin{equation}\label{eq:EDI}
        \calF(\rho_T) - \calF(\rho_0) + \int_0^T \bigg\{ \frac{1}{2} \int_\Omega | v_t |^2 \dd \rho_t + \calS(\rho_t) \bigg\} \dd t \leq 0.
    \end{equation}

    If the chain rule holds, i.e., for any curve satisfying \eqref{eq:CE}
    \begin{equation}\label{eq:chain-rule}
        \frac{\dd}{\dd t} \calF(\rho_t) = \int \nabla (f'(u_t) +  V) \cdot v_t \dd \rho_t,
    \end{equation}
    then $\rho\in C((0,T);\calP(\Omega))$ is a distributional solution for \eqref{eq:nonlinear-diff-pde} with $v_t = -\nabla f'(\rho_t) -\nabla V $. 
\end{proposition}

The point of Proposition~\ref{prop:charact} is that the EDI characterizes the weak solution given that the chain rule \eqref{eq:chain-rule} holds. It can be quite technical to prove that the chain rule holds for an arbitrary $f$. The proof of the chain rule can be based on \cite[Theorem~10.4.8,~Theorem~10.4.13]{ambrosio2008gradient}, where $f$ is assumed to be a convex differentiable function with superlinear growth and to satisfy the McCann condition \cite{mccann1997convexity}, i.e., the map $s\mapsto f(s^{-d})s^d$ is convex and non increasing in $(0,+\infty)$, which yields the geodesic convexity of the internal energy. For example, the Fokker-Planck and porous media equations satisfy these assumptions.

Instead of the McCann condition, the chain rule can be proved assuming that the initial datum is bounded from above and below by a positive constant \cite{agueh2002existence}. A relaxation of the McCann condition appears in \cite{caillet2024doubly}, where $f$ is decomposed, after some approximation, into the difference of two functions satisfying the McCann condition, and a mild assumption on integrability of the initial datum is required. In this paper, we do not go into detail about proving the chain rule. This indeed only concerns the limit equation, and is actually quite standard for porous-medium-type equations. Hence, we will simply use the EDI condition \eqref{eq:EDI} to characterize the solution of \eqref{eq:nonlinear-diff-pde}.

\subsubsection{Crowd motion model}\label{sec:crowd-motion-model} Here we explain the macroscopic crowd motion model \cite{maury2010macroscopic, maury2011handling} and its formulation as a gradient flow in $(\calP(\Omega), W_2)$ formalized using the JKO scheme and the energy-dissipation balance. 

Consider a crowd described by an absolutely continuous measure $\rho = u \calL^d \in \calP(\Omega)$. To account for the hard constraint on the concentration of people, we introduce the congestion constraint requiring the density $u\in L^1(\Omega)$ to remain below 1. We assume the population wishes to move with the velocity $w = - \nabla V$ given by a potential $V\in C^1(\R^d)$. The gradient form of the velocity is the necessary assumption for the model to have the gradient structure. The spontaneous velocity $w$ is not always possible to realize due to the congestion constraint. Thus, the actual velocity is defined as a $L^2$ projection of $w$ on the cone of admissible velocities $C_\rho$. Intuitively, the cone $C_\rho$ is the set of velocities that do not increase the density $u$ in the saturated region $\{u = 1\}$. In the PDE terms, the time evolution of $\rho$ is defined by the advection equation with the projected velocity:
\begin{equation}\label{eq:crowd-motion-PDE}
    \begin{array}{cc}
         \partial_t \rho + \text{div} (\rho v) = 0  \\
         \qquad\qquad\qquad\quad v = \text{Proj}_{C_\rho} w.
    \end{array}
\end{equation}
The cone of admissible velocities is defined in duality with the pressure variable $p$, which plays the role of the Lagrangian multiplier for the constraint $u \leq 1$. Thus, the set of density-dependent pressure variables is
$$
    H^1_u = \{ p \in H^1 (\Omega) : p \geq 0 \text{ and } p (1 - u) = 0 \text{ a.e. in } \Omega \}
$$
and then 
$$
    C_{\rho} \coloneq \Big\{ v \in L^2 (\Omega; \mathbb{R}^d) : \int_{\Omega} \nabla p \cdot v \dd \rho \leq 0 \quad \forall p\in H^1_u \Big\}.
$$

The model has a gradient structure in $(\calP(\Omega), W_2)$ with the driving energy
\begin{equation}\label{eq:energy-infty}
    \Fcm (\rho) = \begin{cases}
        \int_\Omega V \, u \dd \calL^d \qquad \text{if } \rho = u \calL^d \text{ and } u \leq 1,\\
        +\infty \qquad \text{otherwise.}
    \end{cases}
\end{equation}
One can define this gradient flow with the JKO scheme. As before, given an initial datum $\rho_0\in \spt (\Fcm)$ and a time step $\tau>0$, we set up an iterative scheme
\begin{equation}\label{eq:JKO-crowd} \tag{JKO$_\text{cm}^\tau$}
    \rho^\tau_{k+1} \in \arg\min_{\rho\in\calP(\Omega)} \Big\{ \Fcm(\rho) + \frac{1}{2\tau} W_2^2(\rho, \rho_k^\tau) \Big\}.
\end{equation}

Again, we use the energy-dissipation balance to characterize solutions. The Fisher information in this case is defined as follows.
\begin{definition} For a density-pressure pair $(u, p)$, we define the Fisher information as
    \begin{equation}\label{eq:Fisher-CM}
    \Scm (u, p) \coloneq \begin{cases}
        \frac{1}{2} \int_\Omega |\nabla p + \nabla V|^2 u \dd \calL^d \quad \text{if } p \in H^1_u \\
        +\infty \hspace{3.2cm} \text{otherwise.}
    \end{cases} 
\end{equation}
\end{definition}

\begin{proposition}[Characterization of solutions for crowd motion]\label{prop:charact-crowd-motion}
    Let a pair $(\rho, v)$ satisfy the continuity equation \eqref{eq:CE} in the distributional sense with the no-flux boundary condition. Let the initial datum $\rho_0 = u_0 \calL^d$ be such that $\Fcm(\rho_0)<\infty$ and $p\in L^\infty([0,T]; H^1(\Omega))$ be the pressure variable such that $p (1- u) = 0$. 
    Further, let $(\rho, p, v)$ satisfy the EDI:
    \begin{equation}\label{eq:EDI-crowd-motion}
        \Fcm(\rho_T) - \Fcm(\rho_0) + \int_0^T \bigg\{ \frac{1}{2} \int_\Omega | v_t |^2 \dd \rho_t + \Scm (u_t, p_t) \bigg\} \dd t \leq 0.
    \end{equation}
    then $\rho\in C((0,T);\calP(\Omega))$ is a distributional solution for the crowd motion model with $v_t = -\nabla p_t -\nabla V $. 
\end{proposition}
\begin{proof} Using that $(\rho, v)$ satisfies the continuity equation, we directly get that
    $$
        \Fcm(\rho_t) - \Fcm(\rho_0) 
        = \int_\Omega V \dd \rho_t - \int_\Omega V \dd \rho_s = \int_0^t \int_\Omega \nabla V \cdot v_s \dd \rho_s.
    $$
    To prove the chain rule, we need to show that
    $$
        \Fcm(\rho_t) - \Fcm(\rho_0) 
        = \int_0^t \int_\Omega (\nabla V + \nabla p_s) \cdot v_s \dd \rho_s,
    $$
    which means that we need to prove $\int\nabla p\cdot v\dd\rho=0$.
    We proceed similarly as in \cite[Proposition~4.7]{di2016measure}. Fix a time $t_0>0$ and an admissible pressure $q\in H^1_{u_{t_0}}$. Consider the function $(t_0-\varepsilon, t_0 + \varepsilon) \ni t \mapsto \int_\Omega q \dd \rho_t$. We have
    $$
        \int_\Omega q \, \dd \rho_t \leq \int_\Omega q \dd \calL^d \qquad t \in (t_0 - \varepsilon, t_0 + \varepsilon)
    $$
    and the equality is achieved when $t = t_0$. If the map is differentiable at $t_0$, then
    $$
        \frac{\dd}{\dd t} \int_\Omega q \, \dd \rho_t \bigg|_{t=t_0} = 0,
    $$
    and it implies
    $$
        \int_0^t \int_\Omega \nabla p_s \cdot v_s \dd \rho_s = 0.
    $$
    The subtlety with the differentiability is that for any $q\in H^1(\Omega)$ we have for almost every $t$
    $$
        \frac{\dd}{\dd t} \int_\Omega q \, \dd \rho_t = \int_\Omega \nabla q \cdot v_t \dd \rho_t, 
    $$
    but it can happen that the derivative does not exist at $t_0$ for any $q$. Fortunately, we can show that the set of times, for which the derivative on the left-hand side exists for every $q\in H^1(\Omega)$ has full measure. By the continuity equation, this set of times can be chosen as all $t_0 \in [0,T]$ such that
    $$
        \intbar_{t_0}^{t_0+\varepsilon} v_t u_t \dd t \rightharpoonup v_{t_0} u_{t_0} \quad \text{weakly in } L^2(\Omega).
    $$
    The latter convergence holds by the density argument and the bounds $\int_0^T\int_\Omega |v_t|^2 \dd \rho_t \dd t < \infty$ and $\|u_t\|_\infty \leq 1$.
\end{proof}

\subsection{Discretization}

We define a discretization of $\Omega$ as $\calT^h = h\Z^d \cap \Omega$ for any $h>0$. 
An arbitrary probability measure on $\calT^h$ has the form
$$
    \rho^h = \sum_{z\in\calT^h} \rho^h_z \delta_z, \qquad \sum_{z\in\calT^h} \rho^h_z = 1.
$$

It will be convenient for some calculations to use a piecewise constant counterpart of $\rho^h$. Instead of assigning the mass $\rho^h_z$ to one point $z\in\calT^h$, we assign $\rho^h_z$ to the cell $Q_h(z) \coloneq z + (-h/2, h/2)^d$ around $z$  with the uniform density. Precisely, we define
\begin{equation}\label{eq:pwconst-measure}
    \hat{\rho}^h \coloneq \sum_{z\in\calT^h} \rho^h_z \calL^d|_{Q_h(z)}.
\end{equation}
The measure $\hat{\rho}^h$ is absolutely continuous with respect to the Lebesque measure and has the density:
$$
    \hat{u}^h \coloneq \frac{\dd \hat{\rho}^h}{\dd \calL^d} = \sum_{z\in\calT^h} \frac{\rho^h_z}{h^d} \Ind_{Q_h(z)} = \sum_{z\in\calT^h} u^h_z \Ind_{Q_h(z)}.
$$
In the same way, we define the piecewise constant reconstruction for any bounded discrete function $f^h \in \calB(\calT^h)$:
\begin{equation}\label{eq:pwconst-function}
    \hat{f}^h \coloneq \sum_{z\in\calT^h} f^h_z \Ind_{Q_h(z)}.
\end{equation}

For any $z\in\calT^h$ (away from a boundary), the set of neighbors has $2d$ points of the form $z + \vh$ with $\vh \in \{h e_1, - h e_1, \dots, h e_d, - h e_d\} \eqcolon \vdh$. We denote by $\Sigma^h\subset \calT^h\times\calT^h$ the set of pairs of neighboring points, i.e., $\Sigma^h \coloneq \{(z,\zeta) : |z - \zeta| = h\}$.
The interface between two neighboring cells $Q_h(z)$ and $Q_h(\zeta)$, $(z, \zeta)\in\Sigma^h$, is $I_h(z,\zeta) \coloneq \overline{Q_h(z)} \cap \overline{Q_h(\zeta)}$.

We define the fully discrete JKO scheme as a counterpart for \eqref{eq:JKO-intro} with the minimization restricted to the atomic measures on the regular lattice $\calT^h$. We use the discrete energy functional defined as
\begin{equation}\label{eq:disc-energy-nonlinear}
    \calF^h(\rho^h) = \sum_{z\in\calT^h} f \bigg( \frac{\rho^h_z}{h^d} \bigg) h^d + \sum_{z\in\calT^h} V(z) \rho^h_z.
\end{equation}
Note that the discrete version of the internal energy can be obtained by inserting the piecewise constant reconstruction $\hat\rho^h$ into the continuous internal energy 
$$
    \calF(\hat{\rho}^h) = \int_\Omega f \bigg( \frac{\dd\hat{\rho}^h}{\dd\calL^d} \bigg) \dd \calL^d = \sum_{z\in\calT^h} f \bigg( \frac{\rho^h_z}{h^d} \bigg) h^d = \calF^h(\rho^h).
$$
Instead, for the potential energy, we can use the atomic measure to obtain the discrete energy. The discrete energy for the crowd motion model takes the form
\begin{equation}\label{eq:disc-energy-crowd-motion}
    \Fcm(\rho^h) \coloneq \begin{cases} \displaystyle
        \sum_{z\in\calT^h} V(z) \rho^h_z \qquad \text{if } u^h = \frac{\rho^h}{h^d} \leq 1, \\
        +\infty \qquad \text{otherwise,}
    \end{cases}
\end{equation}
which is a natural counterpart for \eqref{eq:energy-infty}.

\begin{definition}\label{def:JKOh} For a given $h > 0$ and $\tau > 0$ of the form $\tau = T/N$ for $T > 0$, $N\in\N$, and given initial data $\rho^{h}_0 \in \calP(\calT^h)$, we define $\{\rho^{h,\tau}_k\}_{k=1,\dots,N}$ as 
    \begin{equation}\label{eq:JKOh}\tag{JKO$^{h,\tau}$}
        \rho^{h,\tau}_{k+1} \in \arg\min_{\rho^h\in\calP(\calT^h)} \Big\{ \calF^h(\rho^h) + \frac{1}{2\tau} W_2^2(\rho^h, \rho^h_k) \Big\}.
    \end{equation}
\end{definition}
It will be convenient to use the notation
\begin{equation*}
    \calJ^{h,\tau}_{\rho^h_0}(\rho^h) \coloneq \calF^h(\rho^h) + \frac{1}{2\tau} W_2^2(\rho^h, \rho^h_0).
\end{equation*}

\begin{remark}[Propagation of the $L^\infty$ bound] \label{remark:prop-Linfty}
    Throughout the paper, we do not require additional assumptions on the initial data. Yet we add the remarks indicating where the proofs can be simplified if we assume a uniform $L^\infty$ bound on $\{\rho^h_0\}_{h>0}$. In those remarks, we rely on the propagation of the $L^\infty$ bound through iterations of the fully discrete JKO scheme. In particular, if we assume an $L^\infty$ bound on the initial data, then we do not need  \eqref{ass:internal-energy-superlinear} for the nonlinear diffusion equation without drift. Here we provide an argument for the propagation of the $L^\infty$ bound. 

    Consider $\calF^h$ given in \eqref{eq:disc-energy-nonlinear} with $V=0$. Let $\displaystyle M \coloneq \max_{z\in\calT^h} \rho^h_0(z)$ and suppose that $S \coloneq \{ z\in\calT^h : \rho^{h,\tau}_1(z) > M \}$ is not empty. Then
    $$
        \sum_{z\in S} \rho^{h,\tau}_1 (z) > |S| M \geq \sum_{z\in S} \rho^{h}_0 (z).
    $$
    Therefore, the mass has to be transported from outside $S$ to inside $S$. More precisely, there exists $(x,y)\in \spt (\gamma)$, with $\gamma$ being the optimal coupling between $\rho^{h,\tau}_1$ and $\rho^h_0$, such that $x\in S$, $y \notin S$, and $\gamma_{xy} > 0$.

    We construct a competitor as
    $\Tilde{\rho}^{h,\tau}_1 \coloneq \rho^{h,\tau}_1 - \gamma_{xy} \delta_x + \gamma_{xy} \delta_y$, together with a new transport plan $\Tilde{\gamma} \in \calP (\calT^h\times\calT^h)$ defined as
    $\Tilde{\gamma}\coloneq\gamma+\gamma_{xy}\delta_{(y,y)}-\gamma_{xy}\delta_{(x,y)}$. This plan is a coupling between $\Tilde{\rho}^{h,\tau}_1$ and $\rho^h_0$, and using it we find
    $$
        W_2^2(\rho^h_0, \Tilde{\rho}^{h,\tau}_1) \leq W_2^2(\rho^h_0, \rho^{h,\tau}_1).
    $$
    Furthermore, by convexity of $f$, we get
    \begin{align*}
        \calF^h(\Tilde{\rho}^{h,\tau}_1)
        &= \calF^h(\rho^{h,\tau}_1) + f \bigg( \frac{\Tilde{\rho}^{h,\tau}_1(x)}{h^d} \bigg) h^d - f \bigg( \frac{\rho^{h,\tau}_1(x)}{h^d} \bigg) h^d
        + f \bigg( \frac{\Tilde{\rho}^{h,\tau}_1(y) }{h^d} \bigg) h^d - f \bigg( \frac{\rho^{h,\tau}_1(y)}{h^d} \bigg) h^d \\
        &\leq \calF^h(\rho^{h,\tau}_1) - f'\bigg( \frac{\rho^{h,\tau}_1(x)}{h^d}\bigg) \gamma_{xy} + f'\bigg( \frac{\rho^{h,\tau}_1(y)}{h^d}\bigg) \gamma_{xy},
    \end{align*} 
    where we recall that $\rho^{h,\tau}_1(x) > M$ and $\rho^{h,\tau}_1(y) \leq M$, thus, $\calF^h(\Tilde{\rho}^{h,\tau}_1) < \calF^h(\rho^{h,\tau}_1)$. Hence $\calJ^{h,\tau}_{\rho^h_0}(\Tilde{\rho}^{h,\tau}_1) < \calJ^{h,\tau}_{\rho^h_0}(\rho^{h,\tau}_1)$ we obtain a contradiction with $\rho^{h,\tau}_1$ being a minimizer.
\end{remark}

\subsection{Main results}\label{sec:main-results}
Before we get to the main results, we study a toy model in Section~\ref{sec:potential-movement}. We consider evolution driven purely by a potential $V$, i.e., the energy functional is
$$
    \calF_V(\rho) = \int_\Omega V \dd \rho.
$$
This toy model is instructive because it illustrates the necessary relation between the time and space discretization parameters $\tau$ and $h$ for passing to the joint limit. Already for the toy model as well as the main results we need $h/\tau \to 0$. 

The first main result is the convergence of \eqref{eq:JKOh} for the nonlinear diffusion. We require different assumptions on the energy density $f$ depending on the presence or absence of the external potential $V$. If $V=0$, then we assume the following.

\vspace{0.3cm}
    \framebox{ \centering \begin{minipage}{0.92\linewidth}
        \textbf{Assumptions on the energy density.} The energy density $f: [0,+\infty) \to \R$ satisfies the following:
        \begin{equation}\label{ass:internal-energy}\tag{f0}
            f \text{ is convex and } C^2((0,+\infty))
        \end{equation}
        \begin{equation}\label{ass:internal-energy-superlinear}\tag{f1}
            f \text{ has superlinear growth at infinity.} 
        \end{equation}
    \end{minipage}}
\vspace{0.3cm}

The superlinearity assumption is quite standard in this setting. Removing it would require to re-define the functional $\mathcal F$ using the recession function $f^\infty$ (see, for intance, Chapter 7 in \cite{santambrogio2015optimal}). Moreover, this assumption guarantees equiintegrability of the densities provided by the scheme, which is crucial in some arguments throughout the paper. This requirement has been weakened in related papers, such as in \cite{BarHraSan} (but the arguments there are quite technical) or in \cite{caillet2024doubly} (but in such a paper equiintegrability is obtained by propagating in time the integrability properties of the initial density; the method to prove this strongly relied on displacement convexity and cannot be reproduced in our setting).

If the potential $V$ is present, then, in addition to \eqref{ass:internal-energy} and \eqref{ass:internal-energy-superlinear}, we need $f$ to satisfy a sort of quantifiable convexity assumption, which is made precise below.

\vspace{0.3cm}
    \framebox{ \centering \begin{minipage}{0.92\linewidth}
        \textbf{Assumptions on the energy density.} The energy density $f: [0,+\infty) \to \R$ satisfies \eqref{ass:internal-energy}, \eqref{ass:internal-energy-superlinear}, and the following:
        \begin{equation} \label{ass:internal-energy-plus-potential} \tag{fV}
            \begin{array}{l}
                \text{(i) } f' \text{ is strictly increasing; } \\
                \text{(ii) } \text{there exist } C_f> 0, ~\theta < 1 + 1/d, \text{ and } s_0 \geq 1 \text{ such that for any } s \geq s_0: \\
                \qquad f''(s) \geq C_f s^{-\theta} \\
                \text{(iii) } V \in \Lip(\R^d) \\
            \end{array}
        \end{equation}
    \end{minipage}}
\vspace{0.3cm}

    \begin{remark}
        For the energy densities $f(s) = s^m / (m - 1)$, $m>1$ corresponding to the porous media equation, \eqref{ass:internal-energy-plus-potential} follows directly from \eqref{ass:internal-energy-superlinear} with $\theta=1$. Yet, in general, there are examples of $f$ such as $f(s) = s^2 / 2 + \sin(s)$, which is convex and superlinear at infinity, but does not satisfy \eqref{ass:internal-energy-plus-potential}.
    \end{remark}

    
    

\begin{theorem}[Main result for nonlinear diffusion]\label{th:main-convergence-diffusion}
    Assume \eqref{ass:internal-energy}, \eqref{ass:internal-energy-superlinear}, and either $V=0$ or \eqref{ass:internal-energy-plus-potential}. Let $\{\rho^{h,\tau}_k\}_{k=1,\dots,N}$ be the family of \eqref{eq:JKOh} minimizers as in Definition~\ref{def:JKOh} with $\calF^h$ defined in \eqref{eq:disc-energy-nonlinear}. Let the family of initial data $\{\rho^h_0\}_{h>0}$ be such that there exists $\rho_0\in\calP(\Omega)$ with finite energy $\calF(\rho_0) < \infty$ such that
    $$
        \rho^h_0 \rightharpoonup \rho_0 \quad \text{narrowly as } h\to 0 \quad \text{and} \quad \lim_{h\to 0} \calF^h(\rho^h_0) = \calF(\rho_0).
    $$
    
    Then a suitable interpolation in time of $\{\rho^{h,\tau}_k\}_{k=1,\dots,N}$ converges as $h, \tau, h/\tau \to 0$ uniformly in time for the $W_2$ distance to an absolutely continuous curve $[0,T] \ni t \mapsto \rho_t \in \calP(\Omega)$ such that $\rho_t$ is the distributional solution of \eqref{eq:nonlinear-diff-pde}.
\end{theorem}

\begin{remark}\label{remark:interaction-1}
    Another common example of driving energy is interaction energy that takes the form
    \begin{equation*}
        \calF_W(\rho) = \int_\Omega (W*\rho) (x) \rho(\dd x)
    \end{equation*}
    with the natural discrete counterpart
    $$
        \calF^h_W(\rho^h) = \sum_{x,y\in\calT^h\times\calT^h} W(x - y) \rho^h_x \rho^h_y.
    $$
    Assuming that the interaction potential $W: \R^d \to \R$ is regular enough, i.e. symmetric, non-negative, and $W\in \Lip(\R^d) \cap C^1(\R^d \backslash \{0\})$, the proof of Theorem~\ref{th:main-convergence-diffusion} can be adapted to deal with the sum of the internal and interaction energy instead of the internal and potential energy. This means that \eqref{eq:JKOd-intro} approximates the solutions of the aggregation-diffusion equation.

    We choose not to focus on the interaction energy in this paper, because the optimization problem in \eqref{eq:JKOd-intro} becomes non-convex, which is less attractive from the numerical point of view presented in Section~\ref{sec:intro-numerics}. Nevertheless, the adaptation of the proofs for $W*\rho$ instead of $V$ seems not to be too difficult, and we comment on them in Remarks~\ref{remark:interaction-2} and \ref{remark:interaction-3}.
\end{remark}

For the crowd motion model, we need a smoother potential.

\vspace{0.3cm}
    \framebox{ \centering \begin{minipage}{0.92\linewidth}
        \textbf{Assumptions on the potential for the crowd motion model.}
        \begin{equation}\label{ass:crwod-motion-energy} \tag{pV}
            V\in C^1(\R^d)
        \end{equation}
    \end{minipage}}
\vspace{0.3cm}

The main convergence result in this case is as follows.

\begin{theorem}[Main result for crowd motion] \label{th:main-convergence-crowd-motion}
    Assume \eqref{ass:crwod-motion-energy}. Let $\{\rho^{h,\tau}_k\}_{k=1,\dots,N}$ be the family of \eqref{eq:JKOh} minimizers as in Definition~\ref{def:JKOh} with $\calF^h$ defined in \eqref{eq:disc-energy-crowd-motion}.  Let the family of initial data $\{\rho^h_0\}_{h>0}$ be such that there exists $\rho_0 = u_0 \calL^d \in\calP(\Omega)$ with $\|u_0\|_{L^\infty} \leq 1$ such that
    $$
        \rho^h_0 \rightharpoonup \rho_0 \quad \text{narrowly as } h\to 0 \quad \text{and} \quad \lim_{h\to 0} \Fcm(\rho^h_0) = \calF_\text{CM}(\rho_0).
    $$

     Then a suitable interpolation in time of $\{\rho^{h,\tau}_k\}_{k=1,\dots,N}$ converges as $h, \tau, h/\tau \to 0$ uniformly in time for the $W_2$ distance to an absolutely continuous curve $[0,T] \ni t \mapsto \rho_t \in \calP(\Omega)$ such that $\rho_t$ is the distributional solution of \eqref{eq:crowd-motion-PDE}.
\end{theorem}

\section{Evolution driven by a potential}\label{sec:potential-movement}

In this section, we consider the fully discrete JKO scheme \eqref{eq:JKOh} with the potential energy
$$
    \calF^h(\rho^h) = \sum_{x\in\calT^h} V(x) \rho^h_x.
$$
We study this case as a toy example to understand the relation between discretization-in-space parameter $h$ and time step parameter $\tau$ that are needed to pass to the joint limit $h,\tau\to 0$. We see in Section~\ref{sec:rel-h-tau} that the condition $h/\tau \to 0$ is a necessary condition for the convergence of \eqref{eq:JKOh} to hold. In Section~\ref{sec:conv-JKOh-V}, we show the convergence of \eqref{eq:JKOh} for the potential energy with $V\in C^{1,1}(\R^d)$.

\subsection{Relation between $h$ and $\tau$}\label{sec:rel-h-tau}
We observe that the one step of \eqref{eq:JKOh} can be rewritten as a linear programming problem:
\begin{align*}
    \min_{\rho^h} \calJ^{h,\tau}_{\rho^h_0} (\rho^h)
    &= \min_{\rho^h} \bigg\{ \sum_{x\in\calT^h} V(x) \rho^h_x + \frac{1}{2\tau} \min_{\gamma\in\Gamma(\rho^h_0,\rho^h)} \Big\{ \sum_{x,y\in\calT^h\times\calT^h} |x-y|^2 \gamma^h_{xy} \Big\} \bigg\} \\
    &= \min_{\gamma^h, {P^1}_{\#} \gamma^h =\rho^h_0} \bigg\{ \sum_{x,y\in\calT^h\times\calT^h} \Big( V(x) + \frac{1}{2\tau} |x-y|^2  \Big) \gamma^h_{xy} \bigg\}.
\end{align*}
With this simple energy, one can consider the movement of each Dirac mass $\rho^h_0(x) \delta_{x(t)}$ independently of the others. Thus, to find the minimizer of $\calJ^{h,\tau}_{\rho^h_0} $ it is enough to solve for all $x_0\in\spt(\rho^h_0)$:
\begin{equation}\label{eq_minimization}\tag{Min-$\calT_h$}
    \min_{x\in\calT^h} \Big\{ V(x) + \frac{1}{2\tau} |x - x_0|^2 \Big\}.
\end{equation}

Here, we can make the first observation about the relation between $h$ and $\tau$. If $x_1$ is the minimizer of \eqref{eq_minimization}, then
$$
    V(x_1) + \frac{1}{2\tau} |x_1 - x_0|^2 \leq V(x_0).
$$
A simple rewriting gives
$$
    \frac{1}{2\tau} |x_1 - x_0| \leq \frac{V(x_0) - V(x_1)}{|x_1 - x_0|} \leq \Lip(V).
$$
Since the distance between two different points in $\calT^h$ is at least $h$, then a necessary condition for $x_1$ to be different from $x_0$ is
\begin{equation}\label{eq:h-tau-relation}
    \frac{h}{\tau} \leq 2 \Lip(V).
\end{equation}
Consequently, if \eqref{eq:h-tau-relation} does not hold, then the evolution is frozen.

To get error estimates, we compare the minimization problem \eqref{eq_minimization} constrained to the discrete set $\calT^h$ with its counterpart without the constraint:
\begin{equation}\label{eq_minimization_cont}\tag{Min-$\R^d$}
    \min_{x\in\\R^d} \Big\{ V(x) + \frac{1}{2\tau} |x - x_0|^2 \Big\}.
\end{equation}
We begin with even simpler example of the linear potential.
\begin{example}[Linear potential]\label{ex:linear-potential} Let $V(x) = -v\cdot x$ with some constant velocity $v\in\R^d$. The minimizer of \eqref{eq_minimization_cont} is clearly $x_1 = x_0 + \tau v$. Let $x_1^h$ be the minimizer of \eqref{eq_minimization}. Note that 
$$
    -v\cdot x + \frac{1}{2\tau} |x - x_0|^2 
    = \frac{1}{2\tau} \Big( \big|x - (x_0 + \tau v) \big|^2 - |x_0 + \tau v|^2 + x_0^2 \Big)
    = \frac{1}{2\tau} \big|x - (x_0 + \tau v) \big|^2 - v\cdot x_0 - \frac{\tau}{2} |v|^2,
$$
therefore
$$
    x_1^h \in \arg\min_{x\in\calT^h} \Big\{ -v\cdot x + \frac{1}{2\tau} |x - x_0|^2 \Big\} = \arg\min_{x\in\calT^h} \Big\{ \big|x - (x_0 + \tau v) \big|^2 \Big\}.
$$
This implies that $x_1^h = \mathrm{Proj}_{\calT^h} (x_1)$ and $|x_1 - x_1^h| \leq \sqrt{d} h / 2$ for all $x\in \R^d$. Furthermore, since the velocity does not depend on $x$, exactly the same error occurs at every iteration
\begin{align*}
    |x_k - x_k^h| = k |x_1 - x_1^h| \qquad k\in\N.
\end{align*}
We then obtain $|x_1 - x_1^h| \sim h$ (except for some particular combinations of $v$ and $h$). Thus, for $t = k\tau$, we have
$$
    |x_k - x_k^h| = t \frac{|x_1 - x_1^h|}{\tau} \approx t \frac{h}{\tau}.
$$
Therefore, $h/\tau \to 0$ is a necessary condition for the error to vanish as $h,\tau\to 0$.
\end{example}
Example~\ref{ex:linear-potential} shows that the assumption $h/\tau \to 0$ is necessary for the convergence of \eqref{eq:JKOh} to the expected limiting minimizer. We assume $h/\tau \to 0$ in all the subsequent convergence results.

\subsection{Convergence of the fully discrete JKO with potential}\label{sec:conv-JKOh-V}
Let $V$ be an arbitrary $C^{1,1}(\R^d)$ potential. Denote by $x_1$ the minimizer of \eqref{eq_minimization_cont}. One can rewrite \eqref{eq_minimization} in the following way using the optimality condition $x_1 = x_0 - \tau \nabla V(x_1)$:
\begin{align*}
    V(x) + \frac{1}{2\tau} (x - x_0)^2
    &= \frac{1}{2\tau} |x - x_1|^2 + V(x) + \frac{1}{2\tau} |x - x_0|^2 - \frac{1}{2\tau} |x - x_0 + \tau \nabla V(x_1) |^2 \\
    &= \frac{1}{2\tau} |x - x_1|^2 + V(x) - x \cdot  \nabla V(x_1) + (\text{terms independent of } x) \\
    &= \frac{1}{2\tau} |x - x_1|^2 + V(x) - V(x_1) - (x - x_1) \cdot  \nabla V(x_1) + (\text{independent of } x).
\end{align*}
Therefore, \eqref{eq_minimization} is equivalent to 
\begin{equation*}
    \min_{x\in\calT^h} \Big\{ \frac{1}{2\tau} |x - x_1|^2 + V(x) - V(x_1) - (x - x_1) \cdot  \nabla V(x_1) \Big\}.
\end{equation*}
We see the later formulation to have two parts: the distance to the minimizer of \eqref{eq_minimization_cont} $x_1$ and the part depending on the potential $V$. Since $V\in C^{1,1}(\R^d)$, it holds that
\begin{align*}
    |V(x) - V(x_1) - (x - x_1) \cdot  \nabla V(x_1)|
    &= \Big| \int_0^1 \big( \nabla V(x_1 + \lambda (x - x_1) ) - \nabla V(x_1) \big) \dd \lambda \cdot (x - x_1) \Big| \\
    &\leq \frac{|\nabla V|_\Lip}{2} |x - x_1|^2,
\end{align*}
where $|\cdot|_\Lip$ is the Lipshitz constant of $\nabla V$. 

We use two proximal operators defined as
\begin{equation*}
    \mathrm{Prox}_V^{h,\tau} (z) \coloneq \arg\min_{x\in\calT^h} \Big\{ V(x) + \frac{|x - z|^2}{2\tau} \Big\} \quad \text{ and } \quad
    \mathrm{Prox}_V^{\tau} (z) \coloneq \arg\min_{x\in\R^d} \Big\{ V(x) + \frac{|x - z|^2}{2\tau} \Big\},
\end{equation*}
and consider the corresponding sequences of minimizers. Given an arbitrary $x_0\in\R^d$, we denote by $\{x_k^h\}_{k=1,\dots,N} \subset \calT^h$ the sequence of the minimizers defined as $x_k^h \coloneq \mathrm{Prox}_V^{h,\tau} (x_{k-1}^h)$ and by $\{x_k\}_{k=1,\dots,N} \subset \R^d$ the sequence of the minimizers defined as $x_k \coloneq \mathrm{Prox}_V^{\tau} (x_{k-1})$.

\begin{lemma} Consider the time interval $[0,T]$ with $T = N\tau$ for some $N\in\N$. Then the error estimate between the two sequences reads as
$$
    |x^h_N - x_N| \leq C(T, |\nabla V|_{\Lip}) \,\frac{h}{\tau}.
$$
\end{lemma}

\begin{proof} 
The difference between $x^h_{k+1}$ and $x_{k+1}$ can be decomposed into two parts:
$$
    |x^h_{k+1} - x_{k+1}| \leq |\mathrm{Prox}_V^{h,\tau}(x^h_k) - \mathrm{Prox}_V^{\tau}(x^h_k)| + |\mathrm{Prox}_V^{\tau}(x^h_k) - \mathrm{Prox}_V^{\tau}(x_k)|.
$$
The first term is the local error between minimization over $\R^d$ and minimizing over $\calT^h$ with the same initial point $x^h_k$. The second term is the difference of minimizing over $\R^d$ for two different initial points $x^h_k$ and $x_k$. 

To estimate the first type of error $|\mathrm{Prox}_V^{h,\tau}(x_0) - \mathrm{Prox}_V^{\tau}(x_0)|$ for an arbitrary $x_0\in\R^d$, we use as a competitor a projection of $x_1$ on $\calT^h$ denoted by $\Tilde{x}_1$ ($|x_1 - \Tilde{x}_1| \leq h$):
\begin{align*}
    \frac{1}{2\tau} (x^h_1 - x_1)^2 &+ V(x^h_1) -V(x_1) - (x^h_1 - x_1) \cdot  \nabla V(x_1) \\
    &\leq \frac{1}{2\tau} (\Tilde{x}_1 - x_1)^2 + V(\Tilde{x}_1) - V(x_1) - (\Tilde{x}_1 - x_1) \cdot  \nabla V(x_1).
\end{align*}
Rearranging the terms gives
\begin{align*}
    (x^h_1 - x_1)^2 &\leq h^2 + 2\tau \big( V(\Tilde{x}_1) - V(x^h_1) - (\Tilde{x}_1 - x^h_1) \cdot \nabla V(x_1) \big) \\
    &= h^2 + 2\tau \int_0^1 \big( \nabla V(x^h_1 + \lambda (\Tilde{x}_1 - x^h_1)) - \nabla V(x_1) \big) \dd\lambda \cdot (\Tilde{x}_1 - x^h_1) \\
    &\leq h^2 + 2\tau \int_0^1 |\nabla V|_\Lip |(\lambda - 1) (x_1 - x^h_1) + \lambda (\Tilde{x}_1 - x_1) | \dd \lambda \, |\Tilde{x}_1 - x^h_1| \\
    &\leq h^2 + \tau |\nabla V|_\Lip \big( |x_1 - x^h_1| + h \big)^2
\end{align*}
Expanding the square, we obtain
\begin{align*}
    (1 - |\nabla V|_{\Lip} \tau) |x^h_1 - x_1|^2 - 2 |\nabla V|_{\Lip} \tau h |x^h_1 - x_1| \leq (1 + |\nabla V|_{\Lip} \tau) h^2 
\end{align*}
Let $\tau > 0$ be such that $\tau \leq (2|\nabla V|_{\Lip})^{-1}$, then
\begin{align*}
    \bigg( |x^h_1 - x_1| - \frac{|\nabla V|_{\Lip} \tau h}{1 - |\nabla V|_{\Lip} \tau} \bigg)^2
    \leq \frac{h^2}{(1 - |\nabla V|_{\Lip} \tau)^2},
\end{align*}
and we get
\begin{align*}
    |x^h_1 - x_1| \leq \frac{1 + |\nabla V|_{\Lip} \tau}{1 - |\nabla V|_{\Lip} \tau} h \leq (1 + 4 |\nabla V|_{\Lip} \tau) h,
\end{align*}
where we use the simple inequality $(1 + x) / (1 - x) \leq 1 + 4x$ that holds for any $x \in [0, 1/2]$.
Thus, for small enough $\tau$, the error introduced by the grid is bounded as
\begin{align*}
    |\mathrm{Prox}_V^{h,\tau}(x_0) - \mathrm{Prox}_V^{\tau}(x_0)| \leq  (1 + 4 |\nabla V|_{\Lip} \tau) h \eqcolon e_{h,\tau} \qquad \forall x_0 \in \R^d.
\end{align*}

To estimate the second part of the error, precisely, $|\mathrm{Prox}_V^{\tau}(x^h_1) - \mathrm{Prox}_V^{\tau}(x_1)|$, we write the optimality conditions:
\begin{align*}
    \nabla V(x_2) + \frac{x_2 - x_1}{\tau} &= 0 \qquad \text{for } x_2 = \mathrm{Prox}_V^{\tau}(x_1) \\
    \nabla V(y_2) + \frac{y_2 - x^h_1}{\tau} &= 0 \qquad \text{for } y_2 = \mathrm{Prox}_V^{\tau}(x^h_1).
\end{align*}
Subtracting the second equation from the first one and multiplying the difference with $x^h_2 - y_2$, we get
\begin{align*}
    (x_2 - y_2) \cdot \big( \nabla V(x_2) - \nabla V(y_2) \big) + (x_2 - y_2) \cdot \frac{x_2 - y_2 + x^h_1 - x_1}{\tau} = 0.
\end{align*}
We rearrange the terms  as
\begin{align*}
    (1 - |\nabla V|_\Lip \tau) |x_2 - y_2| \leq |x_1 - x^h_1|.
\end{align*}
A similar estimate can be obtained for any time step $k+1$:
$$
    (1 - |\nabla V|_\Lip \tau) |x_{k+1} - y_{k+1}| \leq |x_k - x^h_k|,
$$
where $y_{k+1} = \mathrm{Prox}_V^{\tau}(x^h_k)$.
Then the total error at time step $k+1$ is bounded as
\begin{align*}
    |x_{k+1} - x_{k+1}^h|
    &\leq |x_{k+1} - y_{k+1}| + |y_{k+1} - x_{k+1}^h|
    \leq \frac{1}{1 - |\nabla V|_\Lip \tau} |x_k - x^h_k| + e_{h,\tau} \\
    &\leq e_{h,\tau} \sum_{j=0}^k \frac{1}{(1 - |\nabla V|_\Lip \tau)^j}  
    \leq e_{h,\tau} \sum_{j=0}^k (1 + 2 |\nabla V|_\Lip \tau)^j
    \leq e_{h,\tau} \frac{(1 + 2 |\nabla V|_\Lip \tau)^{k+1} - 1}{2|\nabla V|_\Lip \tau} \\
    &\leq e_{h,\tau} \frac{1}{|\nabla V|_\Lip \tau} e^{2|\nabla V|_\Lip T},
\end{align*}
where we use that $T = k \tau$. Recalling the estimate for $e_{h,\tau}$, we finally obtain
$$
    |x_{k+1} - x_{k+1}^h| \leq \frac{3}{|\nabla V|_\Lip} e^{2|\nabla V|_\Lip T} \frac{h}{\tau}
$$
for $\tau \leq (2|\nabla V|_{\Lip})^{-1}$.
\end{proof}

\section{Nonlinear diffusion}\label{sec:nonlinear-diffusion}

In this section, we consider the fully discrete JKO scheme \eqref{eq:JKOh} for the energy functional of the form
\begin{equation}\label{eq:energy-f-V-discrete}
    \calF^h(\rho^h) = \sum_{x\in\calT^h} f \bigg( \frac{\rho^h_x}{h^d} \bigg) h^d + \sum_{x\in\calT^h} V(x) \rho^h_x
\end{equation}
with $f$ satisfying \eqref{ass:internal-energy}, \eqref{ass:internal-energy-superlinear}, and either $V=0$ or \eqref{ass:internal-energy-plus-potential}. The goal is to prove the convergence of \eqref{eq:JKOh} stated in Theorem~\ref{th:main-convergence-diffusion}.

To explain the strategy that we will use to prove Theorem~\ref{th:main-convergence-diffusion}, we first recall the idea of the convergence proof for the standard JKO scheme and then indicate the adaptations required for the fully discrete case. The approach we have in mind relies on the variational (De Giorgi) interpolation, explained, for instance in \cite[Chapter~3]{ambrosio2008gradient}. Given the minimizers $\{\rho^\tau_k\}$ of \eqref{eq:JKO-intro}, the idea is to prove the inequality 
\begin{equation}\label{eq:var-ineq}
    \calF(\rho^\tau_{k+1}) - \calF(\rho^\tau_k) + \frac{W_2^2(\rho^\tau_k, \rho^\tau_{k+1})}{2\tau} + \int_0^1 \frac{W_2^2(\rho^\tau_k, \rho^\tau_r)}{2\tau r^2} dr \leq 0,
\end{equation}
where $\rho^\tau_r$ is the variational interpolant between $\rho^\tau_k$ and $\rho^\tau_{k+1}$. Combining \eqref{eq:var-ineq} with the lower bound on the Wasserstein distance by the slope of the energy \cite[Lemma~3.1.1]{ambrosio2008gradient}
\begin{equation}\label{eq:slope-bound}
    \frac{1}{2\tau^2 r^2} W_2^2(\rho^\tau_k, \rho^\tau_r) \geq |\partial \calF'|^2 (\rho^\tau_r),
\end{equation}
one gets 
\begin{equation}\label{eq:disc-EDI-one-step}
    \calF(\rho^\tau_{k+1}) - \calF(\rho^\tau_k) + \frac{W_2^2(\rho^\tau_k, \rho^\tau_{k+1})}{2\tau} + \frac{\tau}{2} \int_0^1 |\partial \calF'|^2(\rho^\tau_r) \, dr \leq 0.
\end{equation}
Summing up \eqref{eq:disc-EDI-one-step} for all discrete time steps gives
\begin{equation}\label{eq:disc-time-EDI}
    \calF(\rho^\tau_N) - \calF(\rho^\tau_0) + \sum_{k=0}^{N-1} \tau \frac{W_2^2(\rho^\tau_k, \rho^\tau_{k+1})}{2\tau^2} + \frac{1}{2} \int_0^T \sum_{k=0}^{N-1} |\partial \calF'|^2(\rho^\tau_{r-k}) \, dr \leq 0,
\end{equation}
which is a discrete analog of EDI characterizing the solution of \eqref{eq:nonlinear-diff-pde} (see Proposition~\ref{prop:charact} for the characterization by EDI). It remains to use an appropriate interpolation in time to obtain a sharp inequality which is convenient to pass to the limit to recover \eqref{eq:EDI} (for more detail, we refer to \cite{ambrosio2008gradient}).

A challenging step in the fully discrete settings is to obtain a suitable replacement for \eqref{eq:slope-bound}. The definition of the local slope used in \cite[Lemma~3.1.1]{ambrosio2008gradient} is unsuitable for passing to the discrete-to-continuum limit. Furthermore, we know from the consideration in Section~\ref{sec:potential-movement} that the relation between the discretization parameters $h$ and $\tau$ plays a significant role in the convergence. Therefore, we expect the discrete counterpart for \eqref{eq:slope-bound} to depend on the ratio $h/\tau$ explicitly. 

We realize the following strategy to prove Theorem~\ref{th:main-convergence-diffusion}:
\begin{enumerate}
    \item We begin Section~\ref{sec:interpolations} by introducing the variational interpolation for the fully discrete JKO and obtaining an inequality resembling \eqref{eq:var-ineq}.
    \item We establish the crucial lower bound on the Wasserstein distance by the discrete Fisher information in Lemma~\ref{lemma:lower-bound-by-slope} (for nonlinear diffusion without an external potential) and in Corollary~\ref{cor:slope-bound-with-V} (with the potential). These inequalities provide the replacement for \eqref{eq:slope-bound}.
    \item We conclude Section~\ref{sec:interpolations} by showing the discrete-in-time-and-space version of EDI analogous to \eqref{eq:disc-time-EDI} (Lemma~\ref{lem:var-ineq}). This formulation of EDI provides the base to pass to the limit $h,\tau\to 0$.
    \item Section~\ref{sec:compactness} is devoted to the required compactness results and the convergence result is wrapped up in Section~\ref{sec:convergence}.
\end{enumerate}

\subsection{Interpolations and corresponding inequalities}\label{sec:interpolations}

\begin{definition}[Variational interpolation] \label{def:var-interpolation} \emph{The variational interpolation} $\Tilde{\rho}^{h,\tau}$ for a sequence $\{\rho^{h,\tau}_k\}_{k\in\N}$ is defined for $t = (k + s)\tau$ with any $k\in 1, \dots, N$ and any $s\in (0, 1]$ as
\begin{equation}\label{eq:min-prob-var-inter}
    \Tilde{\rho}^{h,\tau}_{(k + s)\tau} = \arg\min_{\rho^h\in\calP(\calT^h)} \Big\{ \calF^h(\rho^h) + \frac{1}{2\tau s} W_2^2(\rho^h, \rho^h_k) \Big\}
\end{equation}
\end{definition}

The standard way to get the inequality of the type \eqref{eq:var-ineq} using the variational interpolation translates into the discrete setting without changes. We briefly present here the argument for completeness. Consider a function $\calG_k: [0, 1] \to \R$ defined as
    \begin{equation*}
        \calG_k(s) \coloneq \min_{\rho^h\in\calP(\calT^h)} \Big\{ \calF^h (\rho^h) + \frac{1}{2\tau s} W_2^2(\rho^h, \rho^{h,\tau}_k) \Big\}.
    \end{equation*}
The function $\calG_k$ has the following properties:
\begin{enumerate}[label=(\roman*)]
    \item $\calG_k(0^+) = \calF(\rho^{h,\tau}_k)$ and $\displaystyle \calG_k(1) = \calF(\rho^{h,\tau}_{k+1}) + \frac{1}{2\tau} W_2^2(\rho^{h,\tau}_{k+1}, \rho^{h,\tau}_k)$, where $\rho^{h,\tau}_{k+1}$ is the minimizer for \eqref{eq:JKO-intro};
    \item $\calG_k$ is decreasing and, thus, differentiable almost everywhere;
    \item $\displaystyle \calG_k'(s) = -\frac{1}{2\tau s^2} W_2^2(\Tilde{\rho}^{h,\tau}_{(k+s)\tau}, \rho^{h,\tau}_k)$ if $s\in (0, 1]$ is a point of differentiablility.
\end{enumerate}
Since $\calG_k$ is decreasing, we have
$$
    \int_0^1 \calG_k'(s) \dd s \geq \calG_k(1) - \calG_k(0^+),
$$
therefore,
\begin{equation}\label{eq:disc-step-var-ineq}
    \calF^h(\rho^{h,\tau}_{k+1}) - \calF^h(\rho^{h,\tau}_k) + \frac{1}{2\tau} W_2^2(\rho^{h,\tau}_{k+1}, \rho^{h,\tau}_k) + \int_0^1 \frac{1}{2\tau s^2} W_2^2(\Tilde{\rho}^{h,\tau}_{(k+s)\tau}, \rho^{h,\tau}_k) \dd s \leq 0.
\end{equation}

\bigskip

\begin{lemma}[Optimality condition] \label{lemma:optimality-cond} 
    An optimal measure $\Tilde{\rho}^{h,\tau}_{(k + s)\tau}$ in \eqref{eq:min-prob-var-inter} satisfies
    \begin{equation*}
        f' \bigg( \frac{\Tilde{\rho}^{h,\tau}_{(k + s)\tau}(x)}{h^d} \bigg) + V(x) + \frac{\varphi^{h,\tau}(x)}{\tau s} = 0 \qquad \text{for all } x\in\calT^h,
    \end{equation*}
    where $\varphi^{h,\tau}$ is a Kantorovich potential corresponding to the optimal transport plan from $\Tilde{\rho}^{h,\tau}_{(k + s)\tau}$ to $\rho^{h,\tau}_k$.
\end{lemma}
\begin{proof}
    We use the notation
    $$
        \calJ^{h,\tau}_{\rho^h_0}(\rho^h) = \calF^h(\rho^h) + \calW^{h,\tau}_{\rho^h_0}(\rho^h), \qquad \calW^{h,\tau}_{\rho^h_0}(\rho^h) = \frac{1}{2\tau}  W_2^2(\rho^h, \rho^h_0)
    $$
    and the proof relies on the following dual formulation of the Wasserstein distance
    \begin{equation}\label{eq:W-dual}
        W_2^2(\rho^h_1, \rho^h_0) = \sup_{\varphi^h\in L^1(\rho^h_1), \,\psi^h \in L^1(\rho^h_0)} \Big\{ \sum_{x\in\calT^h} \varphi^h(x) \rho^h_1(x) + \sum_{y\in\calT^h} \psi^h(y) \rho^h_0(y): \, \varphi^h\oplus\,\psi^h \leq c \Big\},
    \end{equation}
    with the cost function $c(x,y) = \frac{1}{2} |x-y|^2$.
    
    We consider $\calJ^{h,\tau}_{\rho^h_0}$ as a functional on $L^\infty (\calT^h)$. Given $\rho^h_0\in\calP(\calT^h)$, the functional $\calJ^{h,\tau}_{\rho^h_0}$ equals $+\infty$ on $L^\infty(\calT^h) \backslash \calP(\calT^h)$, which one can show using \eqref{eq:W-dual} (or see \cite[Proposition~7.17]{santambrogio2015optimal} for details). 
    
    Let $\rho^h_1$ be a minimizer of $\calJ^{h,\tau}_{\rho^h_0}$ that is $0\in\partial \calJ^{h,\tau}_{\rho^h_0}(\rho^h_1)$. We claim that
    $$
        \partial \calJ^{h,\tau}_{\rho^h_0}(\rho^h_1) = \partial \calF^h(\rho^h_1) + \partial \calW^{h,\tau}_{\rho^h_0}(\rho^h_1).
    $$
    To show this we use, for example, \cite[Theorem~3.11]{carlier2022classical}, where the requirement is that there exists a point $\nu^h\in\text{dom}(\calF^h)\cap\text{dom}(\calW^{h,\tau}_{\rho^h_0})$ such that $\calF^h$ is continuous in $\nu^h$. We choose $\nu^h$ to be the uniform normalized vector, i.e.\ $\nu^h(x) = 1 / |\calT^h|$ for all $x\in\calT^h$. Indeed both $\calF^h$ and $\calW^{h,\tau}_{\rho^h_0}$ are finite in $\nu^h$ and $\calF^h$ is continuous on $L^\infty(\calT^h)$.

    By \cite[Proposition~7.17]{santambrogio2015optimal}, the subdifferential of $W_2^2(\cdot, \rho^h_0)$ coincides with the set of Kantorovich potentials:
    $$
        \partial \calW^{h,\tau}_{\rho^h_0}(\rho^h_1) = \Big\{ \frac{\varphi^h}{\tau}, \text{ for the Kantorovich potentials } \varphi^h \text{ from } \rho^h_1 \text{ to } \rho^h_0 \Big\}.
    $$

    For $\rho^h_1\in\calP(\calT^h)$, it holds that
    $$
        \Big\{ f'\Big( \frac{\rho^h_1}{h^d} \Big) + V \Big\} = \partial \calF^h(\rho^h_1).
    $$
    Thus, if $0 \in \partial \calJ^{h,\tau}_{\rho^h_0}(\rho^h_1)$, then we obtain the asserted optimality condition $\frac{\varphi^h}{\tau} + f'\Big( \frac{\rho^h}{h^d} \Big) + V =0$ for some Kantorovich potentials $\varphi^h$ from $\rho^h_1$ to $\rho^h_0$.
\end{proof}

We first consider the case where the driving energy is given by the internal energy without the external potential
\begin{equation}\label{eq:energy-internal-disc}
    \calF^h(\rho^h) = \sum_{x\in\calT^h} f\Big( \frac{\rho^h_x}{h^d} \Big) h^d.
\end{equation}

We define a discrete counterpart for the Fisher information from Definition~\ref{def:Fisher}.
\begin{definition}\label{def:discret-Fisher}
    Let  $\calF^h$ be given as in \eqref{eq:energy-internal-disc}. We define \emph{the discrete Fisher information} $\calS^h : \calP(\calT^h) \to [0, +\infty)$ corresponding to $\calF^h$ as 
    $$
        \calS^h(\rho^h) = \frac{1}{4} \sum_{x\in\calT^h} \sum_{\vh\in \vdh} \frac{ | \ell(u^h_{x+\vh}) - \ell(u^h_x) |^2}{h^2} h^d, \qquad u^h = \frac{\rho^h}{h^d},
    $$
    where $\ell$ is exactly the same as in Definition~\ref{def:Fisher}.
\end{definition}

\begin{lemma}\label{lemma:lower-bound-by-slope}
    Consider $\calF^h$ defined in \eqref{eq:energy-internal-disc} and assume \eqref{ass:internal-energy}. Let $\rho^h_0 \in \calP(\calT^h)$ be given and $\rho^{h,\tau}$ be a minimizer of one step of \eqref{eq:JKOh}. Then the following lower bound on the Wasserstein distance with the discrete Fisher information holds
    \begin{equation}
        \frac{1}{2\tau^2} W_2^2(\rho^{h,\tau}, \rho^h_0) \geq \bigg( 1 - \frac{h}{2\tau} \bigg) \calS^h(\rho^{h,\tau}) - \frac{dh}{4\tau}.
    \end{equation}
\end{lemma}
\begin{proof} Let $\gamma^{h,\tau}$ be an optimal coupling between the measures $\rho^{h,\tau}$ and $\rho^h_0$ and $(\varphi^{h,\tau}, \psi^{h,\tau})$ be the corresponding Kantorovich potentials. Let $(x, y) \in \spt(\gamma^{h,\tau})$, then
    \begin{equation}\label{eq:cond-x-y}
        \varphi^{h,\tau}(x) + \psi^{h,\tau}(y) = \frac{|x - y|^2}{2}.
    \end{equation}
    For any $\vh \in \vdh$, we have
    \begin{equation}\label{eq:cond-x+h-y}
        \varphi^{h,\tau}(x + \vh) + \psi^{h,\tau}(y) \leq \frac{|x + \vh - y|^2}{2}.
    \end{equation}
    Subtracting \eqref{eq:cond-x-y} from \eqref{eq:cond-x+h-y} and factorizing the difference of squares we obtain
    $$
        \varphi^{h,\tau}(x + \vh) - \varphi^{h,\tau}(x) \leq \vh \cdot (x - y + \vh/2).
    $$
    For what comes later, it is convenient to rearrange the terms as
    \begin{equation}\label{eq:lower-bound-a>b-h}
        (x - y) \cdot \frac{\vh}{h} \geq \frac{\varphi^{h,\tau}(x + \vh) - \varphi^{h,\tau}(x) }{h} - \frac{h}{2}.
    \end{equation}
    If $(x-y)\cdot\vh \geq 0$, it follows that
    \begin{equation}\label{eq:lower-bound-abs}
         \Big| (x - y) \cdot \frac{\vh}{h} \Big| \geq \frac{ \big( \varphi^{h,\tau}(x + \vh) - \varphi^{h,\tau}(x)  \big)^+  }{h} - \frac{h}{2}.
    \end{equation}
    Choosing $\varphi^{h,\tau}$ so that we can apply the optimality condition from Lemma~\ref{lemma:optimality-cond}, we get
    $$
        \varphi^{h,\tau}(x + \vh) - \varphi^{h,\tau}(x) = \tau \Big( f'(u^{h,\tau}_x) - f'(u^{h,\tau}_{x + \vh}) \Big),
    $$
    and it follows that
    \begin{equation}\label{eq:lower-bound-abs-a>b-h}
        \Big| (x - y) \cdot \frac{\vh}{h\tau} \Big| \geq \frac{ \Big( f'(u^{h,\tau}_x) - f'(u^{h,\tau}_{x + \vh}) \Big)^+  }{h} - \frac{h}{2\tau}.
    \end{equation}
    To deduce an inequality for square values from \eqref{eq:lower-bound-abs-a>b-h}, we use that for any $a,b \geq 0$ such that $a\geq b - \varepsilon$ for some $\varepsilon >0$, it holds that $a^2 \geq (1 - \varepsilon)b^2 - \varepsilon$ (see Lemma~\ref{lemma:inequality} below for the proof). Therefore, we obtain
    \begin{equation}\label{eq:eq:lower-bound-squared}
        \frac{|(x - y) \cdot \vh|^2}{h^2 \tau^2} \geq \Big(1 - \frac{h}{2\tau} \Big) \Bigg( \frac{\big( f'(u^{h,\tau}_x) - f'(u^{h,\tau}_{x + \vh})\big)^{+}}{h} \Bigg)^2 - \frac{h}{2\tau},  \qquad\mbox{ if } (x-y)\cdot \vh > 0.
    \end{equation}

    In the other case, when $(x-y)\cdot\vh < 0$, it means that $(x-y)\cdot(-\vh) > 0$ and
    \begin{equation*}
        (x-y) \cdot \frac{-\vh}{h} \geq \frac{ \big( \varphi^{h,\tau}(x - \vh) - \varphi^{h,\tau}(x)  \big)^+  }{h} - \frac{h}{2}.
    \end{equation*}
    Following similar steps as for $(x-y)\cdot\vh > 0$, we obtain counterparts of \eqref{eq:lower-bound-abs}, \eqref{eq:lower-bound-abs-a>b-h}, and \eqref{eq:eq:lower-bound-squared} for $(x-y)\cdot\vh < 0$. Eventually, we get 
    \begin{equation}\label{eq:eq:lower-bound-squared-neg}
        \frac{|(x - y) \cdot \vh|^2}{h^2 \tau^2} \geq \Big(1 - \frac{h}{2\tau} \Big) \Bigg( \frac{\big( f'(u^{h,\tau}_{x}) - f'(u^{h,\tau}_{x-\vh})\big)^{+}}{h} \Bigg)^2 - \frac{h}{2\tau},   \qquad\mbox{ if } (x-y)\cdot \vh < 0.
    \end{equation}
    
    Summing \eqref{eq:eq:lower-bound-squared} and \eqref{eq:eq:lower-bound-squared-neg} over $\vh\in \vdh$, the left-hand side becomes
    \begin{equation}\label{eq:vector-from-components}
        \sum_{\vh\in \vdh} \frac{|(x - y) \cdot \vh|^2}{h^2} = 2 \sum_{i=1}^d |(x - y) \cdot e_i|^2 = 2 |x - y|^2
    \end{equation}
    and, consequently,
    \begin{align*}
        \frac{|x - y|^2}{2\tau^2} \geq 
        \frac{1}{4} \Big(1 - \frac{h}{2\tau} \Big) \sum_{\vh\in \vdh} &\Bigg( \frac{\big( f'(u^{h,\tau}_x) - f'(u^{h,\tau}_{x + \vh})\big)^{+}}{h} \Bigg)^2 - \frac{d h}{4\tau}.
    \end{align*}

    Now we aim to use the estimate from the previous line to estimate the Wasserstein distance:
    $$
        \frac{1}{2\tau^2} W_2^2(\rho^{h,\tau}, \rho^h_0) = \frac{1}{2\tau^2} \sum_{(x,y)\in\calT^h\times\calT^h} |x-y|^2 \gamma^{h,\tau}(x, y).
    $$
    Therefore,
    \begin{align*}
        \frac{1}{2\tau^2} W_2^2 (\rho^{h,\tau}, \rho^h_0) \geq 
        \frac{1}{4} \Big(1 - \frac{h}{2\tau} \Big) &\sum_{x\in\calT^h} \rho^{h,\tau}_x \sum_{\vh\in \vdh} \Bigg( \frac{\big( f'(u^{h,\tau}_x) - f'(u^{h,\tau}_{x + \vh})\big)^{+}}{h} \Bigg)^2 - \frac{d h}{4\tau}.
    \end{align*}
    One can see the double sum over $\calT^h$ and $\vdh$ as a sum over edges of the grid, where each edge is counted twice: as $(x, x+\vh)$ and $(x+\vh, x)$. Hence the contribution to the sum of an undirected edge $(x, x + \vh)$ is
    \begin{align*}
        \rho^{h,\tau}_x &\Bigg( \frac{\big( f'(u^{h,\tau}_x) - f'(u^{h,\tau}_{x + \vh})\big)^{+}}{h} \Bigg)^2
        + \rho^{h,\tau}_{x + \vh} \Bigg( \frac{\big( f'(u^{h,\tau}_{x + \vh}) - f'(u^{h,\tau}_x)\big)^{+}}{h} \Bigg)^2
    \end{align*}
    Since we assumed that $f'$ is a monotonically increasing function, it means that the contribution of one edge equals to
    \begin{equation}\label{eq:contribution-one-edge}
        \max\big( \rho^{h,\tau}_x, \rho^{h,\tau}_{x + \vh} \big) \frac{\big( f'(u^{h,\tau}_x) - f'(u^{h,\tau}_{x + \vh})\big)^2}{h^2}
    \end{equation}  
    and the lower bound on the sum can be rewritten as
    \begin{align*}
        \sum_{x\in\calT^h} \rho^{h,\tau}_x  \sum_{\vh\in\vdh}  \Bigg( \frac{\big( f'(u^{h,\tau}_x) - f'(u^{h,\tau}_{x + \vh}) \big)^{+}}{h} \Bigg)^2 
        &= \sum_{(x,y)\in\Sigma^h} \max\big( u^{h,\tau}_x, u^{h,\tau}_y \big)  \frac{\big| f'(u^{h,\tau}_y) - f'(u^{h,\tau}_x)\big|^2}{h^2} h^d.
    \end{align*}
    
    Here we want to switch from dependence on $f'$ to $\ell$, where $\ell$ is defined up to an additive constant by the relation $\ell'(s) = \sqrt{s} f''(s)$. This change will be a discrete analog of the equality
    $$s |\nabla f'(s)|^2 = |\sqrt{s} f''(s) \nabla s|^2 = |\nabla \ell(s)|^2,$$ which is commonly used for different representations of the Fisher information.
    We assert that the following inequality holds true
    \begin{equation}\label{eq:ineq-for-f-and-ell}
        \max(a,b) (f'(b) - f'(a))^2 \geq (\ell(b) - \ell(a))^2 \qquad \text{for all } a,b\geq 0.
    \end{equation}
    We can assume w.l.o.g. the inequaliy $b > a$: then, taking into account convexity of $f$, we get
    \begin{align*}
        \sqrt{\max(a,b) } (f'(b) - f'(a))
        = \sqrt{b} \int_a^b f''(s) \dd s 
        \geq \int_a^b \sqrt{s} f''(s) \dd s
        = \ell(b) - \ell(a).
    \end{align*}
    Since both sides of the inequality are positive, we conclude 
    $$
        \max(a,b) (f'(b) - f'(a))^2 \geq (\ell(b) - \ell(a))^2
    $$
    and we proved \eqref{eq:ineq-for-f-and-ell}.

    Applying \eqref{eq:ineq-for-f-and-ell} yields
    \begin{align*}
        \sum_{(x,y)\in\Sigma^h} \max\big( u^{h,\tau}_x, u^{h,\tau}_y \big)  &\frac{\big| f'(u^{h,\tau}_y) - f'(u^{h,\tau}_x)\big|^2}{h^2} h^d \geq \sum_{(x,y)\in\Sigma^h} \frac{\big| \ell(u^{h,\tau}_y) - \ell(u^{h,\tau}_x)\big|^2}{h^2} h^d.
    \end{align*}

    Summarizing the result, we obtain
    \begin{equation*}
        \frac{1}{2\tau^2} W_2^2 (\rho^{h,\tau}, \rho^h_0) \geq 
        \frac{1}{4} \Big(1 - \frac{h}{2\tau} \Big) \sum_{(x,y)\in\Sigma^h} \frac{\big| \ell(u^{h,\tau}_y) - \ell(u^{h,\tau}_x)\big|^2}{h^2} h^d - \frac{d h}{4\tau}.\qedhere
    \end{equation*}
\end{proof}

\begin{lemma}\label{lemma:inequality}
    Let $a,b\geq 0$ be such that $a \geq b - \varepsilon$ for some $\varepsilon > 0$, then $a^2 \geq (1 - \varepsilon) b^2 - \varepsilon$.
\end{lemma}
\begin{proof}
    If $b > \varepsilon$, then rearranging the terms gives
    $$
        a^2 \geq (b-\varepsilon)^2 
        = (1 - \varepsilon) b^2 + \varepsilon b^2 - 2 \varepsilon b + \varepsilon^2
        = (1 - \varepsilon) b^2 - \varepsilon + \varepsilon^2 + \varepsilon (b - 1)^2
    $$
    and the asserted inequality follows by dropping the positive term $\varepsilon^2 + \varepsilon (b - 1)^2$.
    
    In the other case, if $b \in [0, \varepsilon]$, then
    $$
        (1 - \varepsilon) b^2 - \varepsilon \leq 
        (1 - \varepsilon) \varepsilon^2 - \varepsilon
        = - \varepsilon \Big[ \big( \varepsilon - 1/2 \big)^2 + 3/4 \Big] < 0
    $$
    and, trivialy, $a^2 \geq 0 \geq 1 - \varepsilon$.
\end{proof} 

We now consider what changes if the driving energy is given by the internal energy plus an external potential energy as in \eqref{eq:energy-f-V-discrete}. In this case, we define the discrete Fisher information as follows.
\begin{definition}\label{def:discret-Fisher-potential}
    \emph{The discrete Fisher information} $\calS^h_V : \calP(\calT^h) \to [0, +\infty)$ corresponding to $\calF^h$ defined in \eqref{eq:energy-f-V-discrete} is given by
    $$
        \calS^h_V(\rho^h) = \frac{1}{4} \sum_{x\in\calT^h} \sum_{\vh\in \vdh} \frac{ \big| \ell(u^h_{x+\vh}) - \ell(u^h_x) + \sqrt{\max(u^h_x, u^h_{x+\vh})} (V(x+\vh) - V(x)) \big|^2}{h^2} h^d, \qquad u^h = \frac{\rho^h}{h^d},
    $$
    where $\ell$ is exactly the same as in Definition~\ref{def:Fisher}.
\end{definition}
 
\begin{cor}\label{cor:slope-bound-with-V}
    Let the energy $\calF^h$ be given by \eqref{eq:energy-f-V-discrete} with the potential $V\in \mathrm{Lip}(\Omega)$ and let $f$ satisfy \eqref{ass:internal-energy} and \eqref{ass:internal-energy-plus-potential}. Then, in the setting of Lemma~\ref{lemma:lower-bound-by-slope}, we have
    \begin{equation}
        \frac{1}{2\tau^2} W_2^2(\rho^{h,\tau}, \rho^h_0) \geq \bigg( 1 - \frac{h}{2\tau} \bigg) \calS^h_V(\rho^{h,\tau}) - e_h(\rho^{h,\tau}) - \frac{dh}{4\tau}.
    \end{equation}    
    with the Fisher information $S_h$ given in Definition~\ref{def:discret-Fisher-potential} and the additional error term $e_h$ is such that $0 < e_h(\rho^{h,\tau}) \xrightarrow{h\to 0} 0$.
\end{cor}
\begin{proof} Notationwise, we omit the superscript $h,\tau$ in the proof and we distinguish the measure $\rho$ and its density $u_x = \rho_x / h^d$ for $x\in\calT^h$.    

The proof follows the similar steps as in the proof of Lemma~\ref{lemma:lower-bound-by-slope}. The main modification comes from the point that \eqref{eq:contribution-one-edge} does not have to hold anymore. The reason is that we cannot use the monotonicity of $f'$ directly because the potential $V$ can destroy the monotonicity property. Indeed, if we have a potential $V$, the contribution from one undirected edge $(x,y)\in\Sigma^h$ becomes
    \begin{align*}
        u_x \frac{[(f'(u_x) - f'(u_y) + V(x) - V(y))^+]^2}{h^2} + u_y \frac{[(f'(u_y) - f'(u_x) + V(y) - V(x))^+]^2}{h^2}.
    \end{align*}
    Taking into account also the strict monotonicity of $f'$, meaning that $f'(u_y) - f'(u_x) = 0$ if and only if $u_x = u_y$, we can rewrite the contribution of an undirected edge as
    \begin{align*}
        & \max(u_x, u_y) \frac{ (f'(u_x) - f'(u_y) + V(x) - V(y))^2}{h^2} \Ind_{\big\{ (f'(u_x) - f'(u_y) + V(x) - V(y))  (f'(u_x) - f'(u_y)) \geq 0  \big\}} \\
        &+ \min(u_x, u_y) \frac{(f'(u_x) - f'(u_ y) + V(x) - V(y))^2}{h^2} \Ind_{\big\{ (f'(u_x) - f'(u_y) + V(x) - V(y)) ( f'(u_x) - f'(u_y) ) < 0  \big\}}.
    \end{align*}

    We define the set of edges that violate the monotonicity property as
    \begin{equation}\label{eq:edges-violet-monoton}
        \Sigma_{f,V}^h \coloneq \big\{ (x, y)\in\Sigma^h : \big(f'(u_x) - f'(u_y) + V(x) - V(y) \big) \big(f'(u_x) - f'(u_y) \big) < 0  \big\}.
    \end{equation}
    Summing up over all the edges and rearranging the terms give
    \begin{align*}
        \sum_{(x,y)\in\Sigma^h} &\max (u_x, u_y) \frac{(f'(u_x) - f'(u_{x+\vh}) + V(x) - V(y))^2}{h^2} h^d \\
        &+ \sum_{(x,y)\in\Sigma^h_{f,V}} \Big(\min(u_x, u_y) - \max (u_x, u_y) \Big) \frac{(f'(u_x) - f'(u_y) + V(x) - V(y))^2}{h^2} h^d \\
        &= \calS^h(\rho^h) - e_h(u).
    \end{align*}
    The error term is
    $$
        e_h(u) \coloneq \sum_{(x,y)\in\Sigma^h_{f,V}} |u_x - u_y| \frac{(f'(u_x) - f'(u_y) + V(x) - V(y))^2}{h^2} h^d.
    $$
    We take a closer look at the set $\Sigma^h_{f,V}$. If $
    f'(u_x) - f'(u_y) \geq 0$ and $f'(u_x) - f'(u_y) + V(x) - V(y) < 0$, then
    \begin{equation*}
        0 \leq f'(u_x) - f'(u_y) < V(y) - V(x) \leq \text{Lip}(V) h.
    \end{equation*}
    Consequently, for any $(x,y)\in \Sigma^h_{f,V}$, we can conclude that
    \begin{align} \label{eq:bound-der-f}
        |f'(u_x) - f'(u_y)| &\leq \text{Lip}(V) h \\
        |f'(u_x) - f'(u_y) + V(x) - V(y)| &\leq 2 \text{Lip}(V) h, \label{eq:bound-der-f-plus-V}
    \end{align}
    Therefore, the error term is bounded as
    $$
        e_h(u) \leq 2 \text{Lip}(V) \sum_{(x,y)\in\Sigma^h_{f,V}} |u_x - u_y| h^d.
    $$

    Since $f'$ is strictly monotone, it has a well-defined inverse $g\coloneq (f')^{-1}$ and we can write
    \begin{align*}
        |u_x - u_y| &= \big| g(f'(u_x)) -  g(f'(u_y)) \big| = \bigg| \int_{f'(u_x)}^{f'(u_y)} g'(s) \dd s \bigg|
    \end{align*}

    We note that $f$ has to be defined only on $[0,\infty)$ and we extend it by $+\infty$ on $(-\infty,0)$, then the Legendre dual $f^*$ satisfies the relation
    \begin{equation}\label{eq:conj-df}
        (f^*)'(s) = \begin{cases}
            g(s) & \text{for } s\geq f'(0), \\
            0 & \text{otherwise.}
        \end{cases}
    \end{equation}
    We note that $f'(0)$ can equal $-\infty$, which means that $(f^*)'$ and $g$ coincide on the whole real line.
    
    First we consider the case $\min(u_x, u_y) < s_0$. We claim that in this case $\max(u_x, u_y) < s_0 + 1$. Suppose the opposite, then by monotonicity of $f'$ we have
    $$
        \Lip(V) h \geq |f'(u_x) - f'(u_y)| \geq f'(s_0 + 1) - f'(s_0) = f''(s_0 + \lambda_{s_0}),
    $$
    with some $\lambda_{s_0} \in (0, 1)$. Employing \eqref{ass:internal-energy-plus-potential} again, we obtain  
    $$
        \Lip(V) h
        \geq f''(s_0 + \lambda_{s_0})
          \geq \frac{1}{C_f (s_0 + 1)^\theta },
    $$
    which cannot hold for small enough $h$, thus we have a contradiction and $|u_x - u_y| \leq 1$. Therefore, 
    \begin{align*}
        S_1^h \coloneq \sum_{(x,y)\in\Sigma^h_{f,V}} |u_x - u_y| &h^d \Ind_{\{\min(u_x, u_y) < s_0\}}(x,y)
        \leq \sum_{(x,y)\in\Sigma^h_{f,V}} |u_x - u_y| &h^d \Ind_{\{\max(u_x, u_y) \leq s_0 + 1\}}(x,y).
    \end{align*}
    
    Since $(f^*)'$ is continuously differentiable and monotone on $[f'(0), f'(s_0 + 1)]$, we set $L_0:=\mathrm{Lip}\big((f^*)'_{|[f'(0), f'(s_0 + 1)]}\big)$ and we obtain
    \begin{align*}
        S_1^h &= \sum_{(x,y)\in\Sigma^h_{f,V}} |(f^*)'(f'(u_x)) -  (f^*)'(f'(u_y))| h^d \Ind_{\{\max(u_x, u_y) < s_0 + 1\}} (x, y) \\
        &\leq \sum_{(x,y)\in\Sigma^h_{f,V}} L_0|f'(u_x) - f'(u_y)|) h^d \Ind_{\{\max(u_x, u_y) < s_0 + 1\}} (x, y) \\
        &\leq L_0 \text{Lip}(V) h 4d |\Omega| = o(1)|_{h\to 0}.
    \end{align*}
    
    If $\min(u_x, u_y) \geq s_0$, then $\min\big( f'(u_x), f'(u_y) \big) \geq f'(s_0) > -\infty$ and by the mean value theorem there exists $s_{xy} \in [f'(u_x), f'(u_y)]$ such that
    $$
        |u_x - u_y| = \bigg| \int_{f'(u_x)}^{f'(u_y)} (f^*)''(s) \dd s \bigg|
        = (f^*)''(s_{xy})  \,| f'(u_x) - f'(u_y) |.
    $$
    Since $f$ and $f^*$ are convex conjugates and twice continuously differentiable on the corresponding domains, they satisfy the relation 
    $$
        (f^*)'' (f'(\sigma)) = \frac{1}{f''(\sigma)} \qquad \text{for } \sigma \geq 0.
    $$
    For any $s_{xy}\in [f'(u_x), f'(u_y)]$ there exists $\sigma_{xy}\in [u_x, u_y]$ such that $f'(\sigma_{xy}) = s_{xy}$, thus
    $$
        |u_x - u_y|
        = \frac{1}{|f''(\sigma_{xy})|} \,| f'(u_x) - f'(u_y) |
        \leq C_f \max(u_x, u_y)^\theta \text{Lip}(V) h,
    $$
    where we used \eqref{ass:internal-energy-plus-potential} to estimate $f''$. We notice that $\displaystyle\max_{x\in\calT^h} u_x \leq 1 / h^d$, therefore, 
    $$
     \max(u_x, u_y)^\theta 
        \leq \max(u_x, u_y) \frac{1}{h^{d(\theta-1)}}  \leq\frac{u_x + u_y}{h^{d(\theta-1)}}
    $$
    Combining the estimates for $(x,y)\in\Sigma^h_{f,V}$ such that $\min(u_x, u_y) \geq s_0$, we get
    \begin{align*}
        S_2^h \coloneq \sum_{(x,y)\in\Sigma^h_{f,V}} |u_x - u_y| &h^d \Ind_{\{\min(u_x, u_y) \geq s_0\}}(x,y) \\
        &\leq C_f \text{Lip}(V) h \sum_{(x,y)\in\Sigma^h_{f,V}}  \frac{u_x + u_y}{h^{d(\theta-1)}} h^d \Ind_{\{\min(u_x, u_y) \geq p\}}(x,y) \\
        &\leq 2C_f \text{Lip}(V)   h^{1 - d(\theta-1)}  = o(1)|_{h\to 0},
    \end{align*}
    where we take into account the assumption $1 - d (\theta-1) >0$ coming from \eqref{ass:internal-energy-plus-potential}. This concludes the proof because we got
    $$e_h(u) \leq S_1^h + S_2^h = o(1)|_{h\to 0}.$$

\end{proof}

%
    


\begin{remark}
    The linear Fokker-Planck case is easier to handle than the approach in the proof of Corollary~\ref{cor:slope-bound-with-V}. Indeed, the main point in the proof is the following: for two nearby points $x$ and $y$ in the grid, we assume $f'(u_x)+V(x)\geq f'(u_y)+V(y)$ and we need to bound from below $u_x$ with $\max\{u_x,u_y\}$. If $u_x\geq u_y$ there is nothing to prove, otherwise the computations of Corollary~\ref{cor:slope-bound-with-V} exactly aim at obtaining a similar bound up to an error term. In the case where $f'(s)=\log s$, then we simply obtain $u_x\geq u_y e^{-\mathrm{Lip}(V)h}$, which allows to obtain the result up to a multiplicative factor which tends to $1$ as $h\to 0$. Unfortunately, the computations are less explicit when we lose the logarithmic structure of the Fokker-Planck equation.
\end{remark}

\begin{remark}
    It might be possible to modify Lemma~\ref{lemma:lower-bound-by-slope} for a more general finite-volume tessellation. Let $x$ and $z$ be generating points of the finite-volume cells $K_x$ and $K_z$ and $\vh_{zx} \coloneq z - x$. We also use a common assumption that $z - x$ is orthogonal to the interface $f_{zx} \coloneq \overline{K_x} \cap \overline{K_z}$. If we can define the transition kernel as $\kappa_{zx}$ in such a way that \eqref{eq:vector-from-components} becomes
    \begin{align*}
        \sum_{z\sim x} \kappa_{zx} \frac{|(x - y) \cdot \vh_{zx}|^2}{|\vh_{zx}|^2} 
        = |x - y|^2,
    \end{align*}
    then the conclusion should follow.
\end{remark}

\begin{remark}\label{remark:interaction-2}
    As mentioned in Remark~\ref{remark:interaction-1}, the proof of Corollary~\ref{cor:slope-bound-with-V} can be adapted to deal with the interaction potential $W$ instead of the external potential $V$. For this purpose, one needs to replace $V(x)$ with the discrete convolution $\sum_{z\in\calT^h} W(x-z) \rho^h_z$ in the definition of the Fisher information and, consequently, in the proof. The main point, why the proof works similarly for the interaction potential, is to notice that
    $$
        \bigg| \sum_{z\in\calT^h} \big( W(x + \vh -z) - W(x -z) \big) \rho^h_z \bigg|
        \leq \Lip(W) h \sum_{z\in\calT^h} \rho^h_z
        = \Lip(W) h.
    $$
\end{remark}

\bigskip

We are now in a good position to return to the inequality \eqref{eq:disc-step-var-ineq} and turn it into a form suitable for passing to the limit $h\to 0$. For this purpose, we need another type of interpolation in time, namely piecewise geodesic interpolation.

Let a sequence $\{\rho^{h,\tau}_k\}_{k=1,\dots,N}$ be the sequence of \eqref{eq:JKOh} minimizers as in Definition~\ref{def:JKOh}. Since $\calP_2(\Omega)$ is a geodesic space, all pairs $(\rho^{h,\tau}_k, \rho^{h,\tau}_{k+1})$, $k=0, \dots, N-1$ can be connected by a (possibly not unique) geodesic $[0,1] \ni t \mapsto \check\rho^{h,\tau}_k(t)$. Then there exists a Borel time-dependent vector field $(0,1) \ni t \mapsto \check{v}^{h,\tau}_k(t) \in L^2(\rho^{h,\tau}_k(t); \R^d)$ such that the continuity equation 
\begin{equation*}
    \partial_t \check\rho^{h,\tau}_{k} + \mathrm{div} (\check\rho^{h,\tau}_{k} \check{v}^{h,\tau}_{k}) = 0
\end{equation*}
holds in the sense of distributions. These considerations allow us to introduce piecewise geodesic interpolation.
\begin{definition}[Piecewise geodesic interpolation]\label{def:geodesic-interpolation}
    \emph{The piecewise geodesic interpolation} is a curve $[0,T] \ni t \mapsto \check{\rho}^{h,\tau}_t$ defined as 
    $$
        \check{\rho}^{h,\tau}_t \coloneq \rho^{h,\tau}_k \Big( \frac{t - k\tau}{\tau} \Big) \qquad t \in [k\tau, (k+1)\tau),
    $$
    where $\rho^{h,\tau}_{k}(t)$ is a geodesic connecting $\rho^{h,\tau}_k$ to $\rho^{h,\tau}_{k+1}$.
    Moreover, we denote by $\check{v}^{h,\tau}$ the velocity field defined as
    $$
        \check{v}^{h,\tau}_t \coloneq v^{h,\tau}_k\Big( \frac{t - k\tau}{\tau} \Big) \qquad t\in (k\tau, (k+1)\tau).
    $$
\end{definition}

To bring in the integration in time, we will also make use of the following measure
\begin{equation}\label{eq:time-measure}
    \lambda_{\varepsilon}^{h,\tau} (\dd t) \coloneq \sum_{k=0}^{N-1} \Ind_{[(k + \varepsilon)\tau, (k+1)\tau]}(t) \Big( 1 - \frac{h}{2(t - k\tau)} \Big) \calL(\dd t),
\end{equation}
defined for an arbitrary $\varepsilon \in (0, 1)$.

\begin{lemma}\label{lem:var-ineq}
    Let $h>0$, $\tau>0$, and $\rho^h_0\in\calP(\calT^h)$ be given. Let also $\{\rho^{h,\tau}_k\}_k$ be the sequence of \eqref{eq:JKOh} minimizers from Definition~\ref{def:JKOh}. Then for any $\varepsilon\in (0,1)$, the following variational inequality holds
    \begin{equation*}
        0 \geq \calF^h(\rho^{h,\tau}_N) - \calF^h(\rho^h_0) + \frac{1}{2}\int_0^T \| \check{v}^{h,\tau}_t \|^2_{L^2(\check{\rho}^{h,\tau}_t)} \dd t + \int_0^T \calS^h_V (\Tilde{\rho}^{h,\tau}_t) \lambda_\varepsilon^{h,\tau} (\dd t) + \frac{dT}{4} \frac{h}{\tau} \log \varepsilon + o(1)_{h\to 0}.
    \end{equation*}
\end{lemma}

\begin{proof}
A slight modification of the proof of Lemma~\ref{lemma:lower-bound-by-slope} with $\tau s$ instead of $\tau$ shows that for the variational interpolant $\Tilde{\rho}^{h,\tau}_{st}$ defined in Definition~\ref{def:var-interpolation} we have
\begin{equation}
    \frac{1}{2\tau^2 s^2} W_2^2(\Tilde{\rho}^{h,\tau}_{s\tau}, \rho^h_0) \geq \bigg( 1 - \frac{h}{2\tau} \bigg) \calS^h_V(\Tilde{\rho}^{h,\tau}_{s\tau}) - \frac{dh}{4\tau s} + o(1)|_{h\to 0}.
\end{equation}
Multiplying the latter inequality by $\tau$ and integrating over $s$, we get
\begin{align*}
    \frac{\tau}{2} \int_\varepsilon^1 \frac{1}{\tau^2 s^2} W_2^2 (\Tilde{\rho}^{h,\tau}_{(k+s)\tau}, \rho^{h,\tau}_k) \dd s
    &\geq \tau \int_\varepsilon^1 \Big( 1 - \frac{h}{2 s\tau} \Big) \calS^h_V(\Tilde{\rho}^{h,\tau}_{(k+s)\tau}) \dd s + \frac{dh}{4} \log \varepsilon + \tau (1 - \varepsilon) o(1)|_{h\to 0},
\end{align*}
for an arbitrary $\varepsilon\in (0, 1)$. Then we obtain from \eqref{eq:disc-step-var-ineq} 
\begin{align*}
    0 
    \geq \calF^h(\rho^{h,\tau}_{k+1}) - \calF^h(\rho^{h,\tau}_k) + \frac{1}{2\tau} W_2^2(\rho^{h,\tau}_{k+1}, \rho^{h,\tau}_k) + \tau \int_\varepsilon^1 \Big( 1 - \frac{h}{2s\tau} \Big) &\calS^h_V (\Tilde{\rho}^{h,\tau}_{(k+s)\tau}) \dd s \\
    &+ \frac{dh}{4} \log \varepsilon + \tau (1 - \varepsilon) o(1)|_{h\to 0}.
\end{align*} 
Rescaling the time as
\begin{align*}
    \tau \int_\varepsilon^1 \Big( 1 - \frac{h}{2s\tau} \Big) \calS^h_V (\Tilde{\rho}^{h,\tau}_{(k+s)\tau}) \dd s 
    = \int_{\varepsilon\tau}^{\tau} \Big( 1 - \frac{h}{2s} \Big) \calS^h_V (\Tilde{\rho}^{h,\tau}_{k\tau + s}) \dd s,
\end{align*}
and summing up from $k=0$ to $N-1$, we obtain that
\begin{align*}
    0 \geq \calF^h(\rho^{h,\tau}_N) - \calF^h(\rho^h_0) + \frac{\tau}{2} \sum_{k=0}^{N-1} \frac{1}{\tau^2} W_2^2(\rho^{h,\tau}_{k+1}, \rho^{h,\tau}_k) + \int_{\varepsilon\tau}^\tau \Big( 1 - \frac{h}{2s} \Big) &\sum_{k=0}^{N-1} \calS^h_V (\Tilde{\rho}^{h,\tau}_{k\tau + s}) \dd s\\
    &+ \frac{dT}{4} \frac{h}{\tau} \log \varepsilon + o(1)|_{h\to 0}.
\end{align*}

To switch to the time integral from $0$ to $T$ in the term with the Fisher information, we use the measure $\lambda_{\varepsilon}^{h,\tau}$ defined in \eqref{eq:time-measure}:
\begin{align*}
    \int_{\varepsilon\tau}^{\tau} \Big( 1 - \frac{h}{2s} \Big) \sum_{k=0}^{N-1} \calS^h_V (\Tilde{\rho}^{h,\tau}_{k\tau + s}) \dd s
    = \int_0^T \calS^h_V (\Tilde{\rho}^{h,\tau}_t) \lambda_\varepsilon^{h,\tau} (\dd t).
\end{align*}

Now we need to rewrite the term $\frac{\tau}{2} \sum_{k=0}^{N-1} \frac{1}{\tau^2} W_2^2(\rho^{h,\tau}_{k+1}, \rho^{h,\tau}_k)$ in an integral form. For this purpose, we will employ the piecewise geodesic interpolation. For $t \in (k\tau, (k+1)\tau)$, we have 
$$
    \|\check{v}^{h,\tau}_t\|_{L^2(\check{\rho}^{h,\tau}_t)} = \frac{1}{\tau} W_2(\rho^{h,\tau}_{k+1}, \rho^{h,\tau}_k),
$$
therefore, with the appropriate time rescaling, we obtain
\begin{align*}
    \frac{\tau}{2} \sum_{k=0}^{N-1} \frac{1}{\tau^2} W_2^2(\rho^{h,\tau}_{k+1}, \rho^{h,\tau}_k) 
    = \frac{1}{2} \sum_{k=0}^{N-1} \int_{k\tau}^{(k+1)\tau} \int_\Omega |\check{v}^{h,\tau}_t|^2 \dd \check{\rho}^{h,\tau}_t \dd t
    = \frac{1}{2} \int_0^T \int_\Omega |\check{v}^{h,\tau}_t|^2 \dd \check{\rho}^{h,\tau}_t \dd t.
\end{align*}
And the asserted variational inequality holds.
\end{proof}

\subsection{Compactness results}\label{sec:compactness}

We begin the lemma on the convergence of the discrete curves to the limit continuous curve. This step is standard for proving the convergence of various versions of the JKO schemes with the strategy explained, for instance, in
\cite[Section~8.3]{santambrogio2015optimal}.
\begin{lemma}\label{lem:compactness-curves}
    The curves $\Tilde{\rho}^{h,\tau}$ and $\check{\rho}^{h,\tau}$ defined in Definitions~\ref{def:var-interpolation} and \ref{def:geodesic-interpolation}, respectively, converge up to a subsequence to the same limit curve $\rho \in C([0,T];\calP(\Omega))$ as $h, \tau \to 0$ uniformly in time for the $W_2$ distance.
\end{lemma}
\begin{proof}
     Using $\rho^{h,\tau}_k$ as a competitor for one step of \eqref{eq:JKOh}, we have 
    $$
        \calF^h(\rho^{h,\tau}_{k+1}) + \frac{1}{\tau} W_2^2(\rho^{h,\tau}_k, \rho^{h,\tau}_{k+1}) \leq \calF^h(\rho^{h,\tau}_k).
    $$
    Summing up the inequalities for $k = 0, \dots, N - 1$ gives
    $$
        \sum_{k=0}^{N-1} \frac{W_2^2(\rho^{h,\tau}_k, \rho^{h,\tau}_{k+1})}{\tau} \leq \sum_{k=0}^{N-1} \big( \calF^h(\rho^{h,\tau}_k) - \calF^h(\rho^{h,\tau}_{k+1}) \big) 
        = \calF^h(\rho^h_0) - \calF(\hat{\rho}^h_N)
        \leq \sup_{h>0} \calF^h(\rho^h_0) - \inf_{\rho\in\calP(\Omega)} \calF(\rho) \eqcolon C_\calF.
    $$
    We can use the latter inequality to bound the norm of the velocity $\check{v}^{h,\tau}$:
    \begin{align*}
        \int_0^T \|\check{v}^{h,\tau}_t \|^2_{L^2(\check{\rho}^{h,\tau}_t)} \dd t
        = \sum_{k=0}^{N-1} \frac{1}{\tau} W_2^2(\rho^{h,\tau}_{k+1}, \rho^{h,\tau}_k) \leq C_\calF.
    \end{align*}
    Denoting by $\check{\jmath}^{h,\tau}_t$ the corresponding flux $\check{\jmath}^{h,\tau}_t = \check{\rho}^{h,\tau}_t \check{v}^{h,\tau}_t$, we also obtain
    \begin{align*}
        \int_0^T |\check{\jmath}^{h,\tau}_t| (\Omega) \dd t 
        &= \int_0^T \|\check{v}^{h,\tau}_t \|_{L^1(\check{\rho}^{h,\tau}_t)} \dd t
        \leq \sqrt{T} \bigg( \int_0^T \|\check{v}^{h,\tau}_t \|^2_{L^2(\check{\rho}^{h,\tau}_t)} \dd t \bigg)^{1/2} 
        \leq \sqrt{C_\calF T}.
    \end{align*}
    Thus, the family $\{\check{\jmath}^{h,\tau}\}_{h,\tau>0}$ is weakly-* compact in $\calM([0,T]\times\Omega;\R^d)$. Since $(\check{\rho}^{h,\tau}, \check{\jmath}^{h,\tau})$ satisfies the continuity equation, $t \mapsto \check{\rho}^{h,\tau}_t$ is an absolutely continuous curve satisfying for $s < t$
    \begin{align*}
        W_2(\check{\rho}^{h,\tau}_t, \check{\rho}^{h,\tau}_s)
        &\leq \int_s^t |(\check{\rho}^{h,\tau}_r)'| \dd r
        \leq \int_s^t \|\check{v}^{h,\tau}_r \|_{L^1(\check{\rho}^{h,\tau}_r)} \dd r \\
        &\leq \sqrt{t-s} \bigg( \int_s^t \|\check{v}^{h,\tau}_r \|^2_{L^2(\check{\rho}^{h,\tau}_r)} \dd r \bigg)^{1/2}
        \leq \sqrt{C_\calF (t-s)}.
    \end{align*}
    Notice that the estimates above hold for an arbitrary $h>0$, so we can choose $h,\tau \to 0$. Therefore, we can apply the Ascoli--Arzel\'a theorem, and there exists a (not relabeled) subsequence $\{\check{\rho}^{h(\tau),\tau}\}$ and a limit curve $\rho \in C([0, T]; \calP(\Omega ))$, such that
    $$
        \check{\rho}^{h,\tau}_t \to \rho_t \qquad \text{uniformly for the } W_2 \text{ distance.}  
    $$
    Furthermore, the curves $\Tilde{\rho}^{h,\tau}$ and $\check{\rho}^{h,\tau}$ convergence to the same curve $\rho_t$. We use that the curves $\Tilde{\rho}^{h,\tau}$ and $\check{\rho}^{h,\tau}$ coincide at $t = k\tau$ for all $k = 0, \dots, N-1$. For $t\in ((k-1) \tau, k \tau]$, we have
    \begin{align*}
        W_2(\check{\rho}^{h,\tau}_t, \Tilde{\rho}^{h,\tau}_t)
        &\leq W_2(\check{\rho}^{h,\tau}_t, \check{\rho}^{h,\tau}_{k\tau}) + W_2(\check{\rho}^{h,\tau}_{k\tau}, \Tilde{\rho}^{h,\tau}_t)
        \leq \sqrt{C_\calF (t - k\tau)} + W_2(\rho^{h,\tau}_k, \rho^{h,\tau}_{k-1}) + W_2(\rho^{h,\tau}_{k-1}, \Tilde{\rho}^{h,\tau}_t) \\
        &\leq 3 \sqrt{C_\calF \tau}.
    \end{align*}
    Therefore, 
    $$
        \quad\Tilde{\rho}^{h,\tau}_t \to \rho_t \qquad \text{uniformly for the } W_2 \text{ distance.}\qedhere  
    $$
\end{proof}
\begin{remark}
    There is no need to assume $h/\tau\to 0$ to obtain compactness. In fact, the compactness result holds for any relation between parameters $h$ and $\tau$. However, having $h/\tau \to 0$ will be crucial for proving that the limit curve $\rho_t$ is a solution of the corresponding PDE. If $h/\tau \to 0$ is not satisfied, it is possible to construct examples where $\rho^{h,\tau}_k = \rho^h_0$ for all $k=1,\dots,N-1$ (the evolution is frozen) or where the error terms result in additional drift added to the original equation.
\end{remark}

Now we aim to prove the strong compactness in space of the density. To simplify the notation, we define
\begin{equation}\label{eq:reconstruct-ell}
    \hat{\ell}^h \coloneq \sum_{x\in\calT^h} \ell(u^h_x) \Ind_{Q_h(x)}.
\end{equation}
To prove the crucial compactness result in Lemma~\ref{lem:strong-compactness}, we need a version of the Poincar\'e-Wirtinger inequality presented in the following lemma. Several versions of discrete Poincar\'e-type inequalities are known for finite volume schemes \cite{eymard2000finite}. This proof is a simple modification of the standard argument used for such inequalities (as in \cite[Lemma~3.7]{eymard2000finite}).
\begin{lemma}[Discrete Poincar\'e-Wirtinger inequality]\label{lemma:poincare}
    Given $f^h \in \calB(\calT^h)$, let $\hat{f}^h$ be the piecewise constant reconstruction defined in \eqref{eq:pwconst-function} and set 
    $$
        \overline{f}^h \coloneq \intbar_\Omega \hat{f}^h(x) \dd x.
    $$

    Then the following inequality holds
    $$
         \|\hat{f}^h - \overline{f}^h \|_2^2 \leq C_{\Omega} \sum_{(p,q)\in\Sigma^h} |f^h_p - f^h_q|^2 h^{d-2},
    $$
    where $C_\Omega >0$ is independent of $h$.
\end{lemma}
\begin{proof} Fix an arbitrary $x\in\Omega$. There exists $p\in\calT^h$ such that $x\in Q_h(p)$. For any $z\in\Omega$ such that $z \ni Q_h(p)$, there exists $q \in \calT^h$ such that $x\in Q_h(q)$ and we can find a path between $p$ and $q$ consisting of the directed edges $(p_i, p_{i+1})\in\Sigma^h$ such that $[x, z] \cap I_h(p_i, p_{i+1}) \neq \emptyset$, where $[x, z]$ is the line segment connecting $x$ and $z$. We denote by $\text{path}(x,z)$ this sequence of edges
    $$
        \text{path}(x,z) = \big( (p_0 = p, p_1), (p_1, p_2), \dots, (p_{K-1}, p_K = q) \big), \qquad p_i\in\calT^h, \quad (p_i,p_{i+1}) \in \Sigma^h. 
    $$
    To obtain a bound on the number of edges in the path, we use the following idea. A point $p'\in\calT^h$ cannot be in the path if the distance from $p'$ to the line segment $[x,z]$ is more than $2h$. Thus, we can get bound $|\text{path}(x,z)| = K$ by comparing the volume of the set $H([x,z],2h) \coloneq \{y\in\Omega: \text{dist}(y, [x,z]) < 2h \}$ with the volume of one cell $h^d$. 
    \begin{equation}\label{eq:bound-length-path}
        K \leq \frac{|H([x,z],2h)|}{h^d} \leq C_d \frac{(|x-z| + 2h) h^{d-1}}{h^d}
        \leq C_\text{path} \frac{|x-z|}{h},
    \end{equation}
    where the constant $C_d$ is the volume of $(d-1)$-dimensional ball with the radius 1.
    
    Using first the triangular inequality and then the H\"older inequality, we get
    \begin{align*}
        |\hat f^h(x) - \hat f^h(z)| \leq \sum_{(p,q)\in\text{path}(x,z)} |f^h_p - f^h_q|
        \leq \bigg( \sum_{(p,q)\in\text{path} (x,z)} |f^h_p - f^h_q|^2 \frac{1}{h} \bigg)^{1/2} \bigg( \sum_{(p,q)\in\text{path}(x,z)} h \bigg)^{1/2}.
    \end{align*}
    By \eqref{eq:bound-length-path}, we have
    $$
        |\hat f^h(x) - \hat f^h(z)|^2 \leq C |x-z| \sum_{(p,q)\in\Sigma^h} |f^h_p - f^h_q|^2 \frac{1}{h} \Ind {\{[x, z] \cap I_h(p,q) \neq \emptyset\}}.
    $$
    Fix an arbitrary $(p,q)\in\Sigma^h$. We denote by $R$ the diameter of $\Omega$ and by $B_R$ the ball with the radius $R$ centered at the origin. Integrating in $x$ and $z$ over $\Omega$ and using the change of variables $\eta = x - z$, we obtain:
    \begin{align*}
        \int_\Omega \int_\Omega |x-z| \Ind {\{[x, z] \cap I_h(p,q) \neq \emptyset\}} \dd x \dd z
        &\leq \int_{B_R} \int_\Omega |\eta| \Ind {\{[x, x + \eta] \cap I_h(p,q) \neq \emptyset\}} \dd x \dd \eta \\
        &\leq \int_{B_R} |\eta|^2 h^{d-1} \dd \eta
        \leq C_R h^{d-1},
    \end{align*}
    where $C_R < \infty$ is a constant depending only on $R$. Therefore,
    \begin{align*}
        \int_\Omega \int_\Omega |\hat f^h(x) - \hat f^h(z)|^2 \dd x \dd z \leq C_{d,R} \sum_{(p,q)\in\Sigma^h} |f^h_p - f^h_q|^2 h^{d-2}.
    \end{align*}

    To conclude the proof, it is just left to use the Jensen inequality:
    \begin{align*}
        \| \hat{f}^h - \overline{f}^h \|_2^2
        &= \int_\Omega \bigg|\hat{f}^h(x) - \intbar_\Omega \hat{f}^h(z) \dd z \bigg|^2 \dd x \leq \int_\Omega \intbar_\Omega | \hat{f}^h(x) - \hat{f}^h(z) |^2 \dd z \dd x \\
        &\leq C_\Omega \sum_{(p,q)\in\Sigma^h} |f^h_p - f^h_q|^2 h^{d-2}.
    \end{align*}
\end{proof} 

\begin{lemma}\label{lem:strong-compactness}
    Let a family $\{\rho^h\}_{h>0}$ with $\rho^h\in\calP(\calT^h)$ be such that $\hat{\rho}^h \rightharpoonup \rho$ narrowly in $\calP(\Omega)$ and
    $$
        \sup_{h>0} \calS^h(\rho^h) \eqcolon C_\calS < \infty.
    $$
    Then for any $\eta\in\R^d$ we have
    \begin{equation}\label{eq:shift-estimate}
        \sup_{h>0} \|\hat{\ell}^h\|_{L^2(\Omega)} \eqcolon C_2 < \infty \quad \text{and} \quad \int_{\Omega_{|\eta|}} |\hat{\ell}^h(z - \eta) - \hat{\ell}^h (z) |^2 \dd z \leq C |\eta| \max(|\eta|, h),
    \end{equation}
    where $\Omega_{\epsilon} = \{x \in \Omega : \mathrm{dist}(x, \partial \Omega) \geq \epsilon\}$.
    
    Moreover, $\{\hat{\ell}^h\}_{h>0}$ is relatively compact in $L^2(\Omega)$, any subsequential limit $g$ of $\{\hat{\ell}^h\}_{h>0}$ belongs to $H^1(\Omega)$, and there exists a subsequence for which the limit has the form $\displaystyle g = \ell \Big( \frac{\dd\rho}{\dd\calL^d} \Big)$.
\end{lemma}
\begin{proof} We first aim to establish the strong compactness in $L^2(\Omega)$ of the family $\{\hat{\ell}^h - \overline{\ell}^h \}_{h>0}$. By Lemma~\ref{lemma:poincare}, we get the uniform bound for $L^2$ norm:
    \begin{align*}
        \sup_{h>0} \|\hat{\ell}^h - \overline{\ell}^h \|_2^2 \leq \sup_{h>0} C_\Omega \calS^h(\rho^h) \leq C_\Omega C_\calS.
    \end{align*}

    To apply the Riesz-Fr\'echet-Kolmogorov theorem for the family $\{\hat{\ell}^h - \overline{\ell}^h\}_{h>0}$, we need to prove that for $\eta\in\R^d$ the value of $\|\hat{\ell}^h(\cdot - \eta) - \hat{\ell}^h\|_2$ is vanishing uniformly as $|\eta|\to 0$. The approach to this proof is similar to the one in Lemma~\ref{lemma:poincare}. Notice that
    \begin{align*}
        \int_{\Omega_{|\eta|}} |\hat{\ell}^h(z - \eta) - \hat{\ell}^h (z) |^2 \dd z
        &= \sum_{x\in\calT^h} \int_{Q_h(x) \cap \Omega_{|\eta|}} |\hat{\ell}^h(z - \eta) - \hat{\ell}^h (z) |^2 \dd z
    \end{align*}
    For a fixed $z\in\Omega$, there exists a unique pair $x, y \in \calT^h$ such that $z\in Q_h(x)$ and $(z - \eta) \in Q_h(y)$. Furthermore, we can find a path between $x$ and $y$ by following a chain of vertices $\{z_i\}_{i=0,\dots,K}$ such that $Q_h(z_i) \cap [z, z - \eta]$, where $[z, z - \eta]$ is the line segment between $z$ and $z - h$. By construction, $z_0 = x$ and $z_K = y$. Therefore, using Jensen's inequality, we get
    \begin{align*}
        \sum_{x\in\calT^h} \int_{Q_h(x) \cap \Omega_{|\eta|}} &|\hat{\ell}^h(z - \eta) - \hat{\ell}^h (z) |^2 \dd z
        = \sum_{x\in\calT^h} \int_{Q_h(x) \cap \Omega_{|\eta|}} \Big| \sum_{i=0}^{K-1} ( \ell^h(z_{i+1}) - \ell^h(z_i) ) \Big|^2 \dd x \\
        &\leq \sum_{x\in\calT^h} \int_{Q_h(x) \cap \Omega_{|\eta|}} K \sum_{i=0}^{K-1} | \ell^h(z_{i+1}) - \ell^h(z_i) |^2 \dd x \\
        &\leq \sum_{x\in\calT^h} \int_{Q_h(x) \cap \Omega_{|\eta|}} K \sum_{z\in\calT^h} \sum_{\vh \in \vdh} | \ell^h(z + \vh) - \ell^h(z) |^2 \Ind_{H([x, x - \eta], 2h)}(z) \dd x,
    \end{align*}
    where $H([x, x - \eta], 2h) = \{ p\in\R^d : \mathrm{dist} ([x, x - \eta], p) < 2h \}$. The bound on the number of cells in the path was found in \eqref{eq:bound-length-path}
    $$
        K \leq C_{path} \frac{|\eta| + h}{h}.
    $$
    Therefore,
    \begin{align*}
        \sum_{x\in\calT^h} \int_{Q_h(x)} K &\sum_{z\in\calT^h} \sum_{\vh \in \vdh} | \ell^h(z + \vh) - \ell^h(z) |^2 \Ind_{\mathrm{H}([x, x - \eta], 3h)}(z) \dd x \\
        &\leq C_{path} \frac{|\eta| + h}{h} \sum_{z\in\calT^h} \sum_{\vh \in \vdh} | \ell^h(z + \vh) - \ell^h(z) |^2 \int_{\Omega_{|\eta|}} \Ind_{\mathrm{H}([x, x - \eta], 3h)}(z) \dd x \\
        &\leq C_{path} \frac{|\eta| + h}{h} C_{d-1} (|\eta| + 2h) (2h)^{d-1} \sum_{z\in\calT^h} \sum_{\vh \in \vdh} | \ell^h(z + \vh) - \ell^h(z) |^2 \\
        &\leq C (|\eta| + h)^2 \calS^h(\rho^h).
    \end{align*}
    To sum up, we obtain
    $$
        \int_{\Omega_{|\eta|}} |\hat{\ell}^h(z - \eta) - \hat{\ell}^h (z) |^2 \dd z \leq C (|\eta| + h)^2 \calS^h(\rho^h).
    $$

     The Riesz-Fr\'echet-Kolmogorov theorem \cite[Theorem~4.26]{brezis2010functional} provides relative compactness of $\{\hat{\ell}^h - \overline{\ell}^h \}_{h>0}$ in $L^2(\Omega)$. Let $g$ be an arbitrary subsequential limit of $\{\hat{\ell}^h - \overline{\ell}^h \}_{h>0}$. An application of \eqref{eq:shift-estimate} gives
    $$
        \int_{\Omega_{|\eta|}} |g(x - \eta) - g(x)|^2 \dd x
        = \lim_{h\to 0} \int_{\Omega_{|\eta|}} |\hat{\ell}^h(x - \eta) - \hat{\ell}^h(x)|^2 \dd x
        \leq C C_\calS |\eta|^2,
    $$
    which implies that $g \in H^1(\Omega)$ by the characterization of $H^1$ in terms of difference quotients \cite[Theorem~11.75]{leoni2017first}.

    We claim that $\{\overline{\ell}^h\}_{h>0}$ is uniformly bounded, i.e.
    $ \inf_{h>0} \overline{\ell}^h > -\infty$ and  $\sup_{h>0} \overline{\ell}^h < \infty$. First, we prove the upper bound. Suppose the opposite that there exists a (not relabeled) sequence such that $\overline{\ell}^h \to \infty$. The $L^2$ convergence $\hat\ell^h - \overline{\ell}^h \to g$ implies that there exists a further subsequence such that $\hat\ell^h(x) -\overline{\ell}^h \to g(x)$ pointwise a.e. on $\Omega$. Thus, together we obtain that $\hat\ell^h(x) \to \infty$ for a.e. $x\in\Omega$. Since $\hat{\ell}^h = \ell(\hat{u}^h)$ with $\ell$ being continuous and monotonically non-decreasing function, $\hat\ell^h(x) \to \infty$ for a.e. $x\in\Omega$ implies $\hat{u}^h(x) \to \infty$ for a.e. $x\in\Omega$, which is a contradiction with $\int_\Omega \hat{u}^h(x) \dd x = 1$ for all $h>0$. 

    For the lower bound, suppose that there exists a (not relabeled) subsequence such that $\overline{\ell}^h \to -\infty$. Arguing similarly to the upper bound, we arrive at the claim that there exists a subsequence such that $\hat{\ell}^h(x) \to -\infty$ for a.e. $x\in\Omega$. 
    Then we would deduce $\hat{u}^h(x) \to 0$ for a.e. $x\in\Omega$. Together with the equi-integrability of $\hat{u}^h$ (guaranteed by \eqref{ass:internal-energy-superlinear}), this would be again acontradiction to the condition $\int_\Omega \hat{u}^h(x) \dd x = 1$.
     
    Since $\{\overline{\ell}^h\}_{h>0}$ is uniformly bounded, the family $\{\hat{\ell}^h \}_{h>0}$ is compact in $L^2(\Omega)$ with any subsequencial limit being in $H^1(\Omega)$. Again, one can get the pointwise a.e.\ convergence up to subsequence, and using that $\ell$ has a well-defined inverse, we conclude $\hat{u}^{h}(x) \to (\ell^{-1}\circ g)(x)$ a.e.\ on $\Omega$. On the other hand, we have $\hat{\rho}^h \rightharpoonup \rho$ narrowly in $\calP(\Omega)$. Since the weak limit is unique, we obtain $\rho = (\ell^{-1}\circ g) \calL^d$. Therefore, for a.e. $x\in \Omega$, it holds that
    $$
        u(x) \coloneq \frac{\dd \rho}{\dd \calL^d}(x) 
        = \limsup_{r\to 0} \intbar_{B_r(x)} (\ell^{-1} \circ g) (z) \dd z = \ell^{-1} (g(x)).
    $$
    Consequently, $g(x) = \ell(u(x))$ a.e. on $\Omega$.
\end{proof}


\begin{lemma}\label{lem:V-doesnt-matter}
    Let $V\in \Lip(\Omega)$ and $\{\rho^h\}_{h>0}$ be a family of probability measures on $\calP(\calT^h)$. The uniform bound on the Fisher information with the potential $V$ is equivalent to the uniform bound on the Fisher information without the potential, i.e.,
    $$
        \sup_{h>0} \calS^h_V(\rho^h) < \infty
        \qquad \Longleftrightarrow \qquad
        \sup_{h>0} \calS^h(\rho^h) < \infty.
    $$
\end{lemma} 
\begin{proof} Since the part of the Fisher information depending on the potential is uniformly bounded
$$
    \sum_{(x,y)\in\Sigma^h} (V(y) - V(x))^2 \max(u^h_x, u^h_y) h^{d-2} 
    \leq \Lip^2(V) \sum_{(x,y)\in\Sigma^h} (u^h_x + u^h_y) h^d \leq \Lip^2(V),
$$
the proof follows by the Young inequality.
\end{proof}

\subsection{Convergence results} \label{sec:convergence}

To complete the proof of Theorem~\ref{th:main-convergence-diffusion}, we need to consider the liminf inequality for the Fisher information. Various results on discrete-to-continuum $\Gamma$-convergence of quadratic functionals that apply to our setting are available in the literature, see for instance \cite{alicandro2004general}. We present a simple proof of the liminf inequality for completeness.  

We define an approximation for $\nabla \ell(u)$ in the following way
\begin{equation}\label{eq:disc-grad}
    \dnabla^h \hat{\ell}^h \coloneq \frac{1}{2} \sum_{x\in\calT^h} \sum_{\vh\in\vdh} \Ind_{Q_h(x+\vh/2)} \frac{\vh}{h} \frac{\hat{\ell}^h(x + \vh) - \hat{\ell}^h(x)}{h}.
\end{equation}
Note that we use the family of cells $\{Q_h(x+\vh/2) \}_{x\in\calT^h}$ which are dual to the family of cells $\{Q_h(x) \}_{x\in\calT^h}$ used for the reconstruction of $\hat{\ell}^h$.

\begin{lemma}[Convergence of discrete gradients]\label{lem:convergence-disc-grad}
    Let a family $\{\rho^h\}_{h>0}$ with $\rho^h\in\calP(\calT^h)$ be such that $\hat{\rho}^h \rightharpoonup \rho$ in narrowly $\calP(\Omega)$ and
    $$
        \sup_{h>0} \calS^h (\rho^h) \eqcolon C_\calS < \infty.
    $$
    Let $\{\hat{\ell}^h\}_{h>0}$ and $\{\dnabla^h \hat{\ell}^h\}_{h>0}$ be defined as in \eqref{eq:reconstruct-ell} and \eqref{eq:disc-grad} respectively. Then, up to a subsequence,
    \begin{align*}
        \hat{\ell}^h \to \ell(u) \qquad \text{strongly in } L^2(\Omega), \\
        \dnabla^h \hat{\ell}^h \rightharpoonup \nabla \ell(u) \qquad \text{weakly in } L^2(\Omega),
    \end{align*}
    where $u = \dd \rho / \dd\calL^d \in L^1(\Omega)$. Furthermore, we have
    $$
        \liminf_{h\to 0} \calS^h_V(\rho^h) \geq \calS_V(\rho).
    $$
\end{lemma}
\begin{proof}
    The asserted convergence $\hat{\ell}^h \to \ell(u)$ follows from Lemma~\ref{lem:strong-compactness}.

    For the convergence of the gradients $\dnabla^h \hat{\ell}^h \rightharpoonup \nabla \ell(u)$, we first show that $\dnabla^h \hat{\ell}^h\in L^2(\Omega)$ with
    the $L^2$-norm equal to the discrete Fisher information defined in Definition~\ref{def:discret-Fisher} up to a multiplicative constant. Indeed, we calculate
    \begin{align} \label{eq:L2-norm-eq-Fisher}
        \| \dnabla^h \hat{\ell}^h \|^2_{L^2}
        &= \frac{1}{4} \int_\Omega \sum_{x\in\calT^h} \sum_{\vh\in\vdh} \Ind_{Q_h(x+\vh/2)} (z) \frac{|\hat{\ell}^h(x + \vh) - \hat{\ell}^h(x)|^2}{h^2} \dd z \\
        &= \frac{1}{4} \sum_{x\in\calT^h} \sum_{\vh\in\vdh} h^d \frac{|\hat{\ell}^h(x + \vh) - \hat{\ell}^h(x)|^2}{h^2} = \calS^h(\rho^h). \notag
    \end{align}
    The assumed uniform bound on $\calS^h$ implies that there exists a (non-relabeled) subsequence and $\zeta\in L^2(\Omega)$ such that $\dnabla^h \hat{\ell}^h \rightharpoonup \zeta$ weakly in $L^2(\Omega)$.

    To identify the limit $\zeta$, we show an approximate integration-by-parts formula for $\hat\ell^h$ and $\dnabla^h \hat{\ell}^h$.   Let $\varphi\in C^\infty_c (\Omega; \R^d)$. On the one hand, we have
    \begin{align*}
        \int (\nabla\cdot\varphi) \, \hat\ell^h
        &= \sum_{x\in\calT^h} \ell(u^h_x) \int_{Q_h(x)} (\nabla\cdot\varphi)(z) \dd z
        = \sum_{x\in\calT^h} \ell(u^h_x) \sum_{\vh\in\vdh} \int_{I(x,\vh)} \varphi(z) \cdot \frac{\vh}{h} \,\calH^{d-1}(\dd z),
    \end{align*}
    where we used the divergence theorem. We denote $\varphi^h_{x,\vh} \coloneq \intbar_{I(x,\vh)} \varphi(z) \calH^{d-1}(\dd z)$ and note that $\varphi^h_{x,\vh} = \varphi^h_{x+\vh,-\vh}$ to obtain
    \begin{align*}
        \int (\nabla\cdot\varphi) \, \hat\ell^h &=
        \sum_{x\in\calT^h} h^{d-1} \sum_{\vh\in\vdh} \ell(u^h_x) \,\varphi^h_{x,\vh} \cdot \frac{\vh}{h}
        = \frac{1}{2} \sum_{x\in\calT^h} h^{d-1} \sum_{\vh\in\vdh} \Big( \ell(u^h_x) \varphi^h_{x,\vh} \cdot \frac{\vh}{h} - \ell(u^h_{x+\vh}) \varphi^h_{x,\vh} \cdot \frac{\vh}{h} \Big) \\
        &= - \frac{1}{2} \sum_{x\in\calT^h} h^{d-1} \sum_{\vh\in\vdh} \big( \ell(u^h_{x+\vh}) - \ell(u^h_x) \big) \intbar_{Q_h(x+\vh/2)} \varphi^h_{x,\vh} \cdot \frac{\vh}{h} \dd z.
    \end{align*} 
    On the other hand, by the definition of $\dnabla^h \hat{\ell}^h$, we have
    \begin{align*}
        \int \varphi \cdot \dnabla^h \hat{\ell}^h 
        &= \frac{1}{2} \sum_{x\in\calT^h} h^{d-1} \sum_{\vh\in\vdh} \big( \ell(u^h_{x+\vh}) - \ell(u^h_x) \big) \intbar_{Q_h(x+\vh/2)} \varphi (z) \cdot \frac{\vh}{h} \dd z.
    \end{align*}
    Since $\varphi$ is smooth, there exists $C_\varphi> 0$ such that $|\varphi^h_{x,\vh} - \varphi(z)| \leq C_\varphi h$ for all $z\in Q_h(x+\vh/2)$ for any $x\in\calT^h$ and $\vh\in\vdh$. Therefore,
    \begin{align*}
        \bigg| \int (\nabla\cdot\varphi) \, \hat\ell^h + \int \varphi \cdot \dnabla^h \hat{\ell}^h  \bigg| 
        &\leq 
        \sum_{x\in\calT^h} \frac{h^{d-1}}{2} \sum_{\vh\in\vdh} | \ell(u^h_{x+\vh}) - \ell(u^h_x) | \intbar_{Q_h(x+\vh/2)} |\varphi (z) - \varphi^h_{x,\vh}| \dd z \\
        &\leq \frac{C_\varphi}{2} \sum_{x\in\calT^h} \sum_{\vh\in\vdh} | \ell(u^h_{x+\vh}) - \ell(u^h_x) | h^d \\
        &\leq \frac{C_\varphi}{2} \bigg( \sum_{x\in\calT^h} \sum_{\vh\in\vdh} | \ell(u^h_{x+\vh}) - \ell(u^h_x) |^2 h^{d-2} \bigg)^{1/2} \bigg(  \sum_{x\in\calT^h} \sum_{\vh\in\vdh} h^{d+2} \bigg)^{1/2} \\
        &\leq \frac{C_\varphi}{2} \sqrt{\calS^h(\rho^h)} h \sqrt{2d |\Omega|}.
    \end{align*}
    Using again the uniform bound on $\calS^h(\rho^h)$, we obtain the approximate integration-by-parts formula:
    $$
        \int (\nabla\cdot\varphi) \, \hat\ell^h = - \int \varphi \cdot \dnabla^h \hat{\ell}^h + O(h).
    $$
    Passing $h\to 0$ yields
    $$
        \int (\nabla\cdot\varphi) \, \ell(u) = - \int \varphi \cdot \zeta \qquad \text{for any } \varphi \in C_c^\infty(\Omega;\R^d).
    $$
    Since $\varphi$ is arbitrary, we obtain $\zeta = \nabla \ell(u)$. 

    We conclude that the liminf inequality for $\calS^h(\rho^h)$ holds due to \eqref{eq:L2-norm-eq-Fisher} and the lower-semicontinuity of the $L^2$ norm:
    \begin{equation*}
        \liminf_{h\to 0} \calS^h(\rho^h) = \liminf_{h\to 0} \| \dnabla^h \hat{\ell}^h \|^2_{L^2} \geq \|\nabla \ell(u) \|^2_{L^2}.
    \end{equation*}

    In the case where the external potential $V$ is present, we have to consider the other term that we denote by $\hat{\xi}^h\in L^2(\Omega;\R^d)$:
    \begin{align*}
        \hat{\xi}^h_i \coloneq \hat{\mu}^h_i \dnabla^h \hat{V}^h \quad i=1,\dots,d; \qquad \hat{\mu}^h_i &\coloneq \sum_{x\in\calT^h}  \Ind_{Q_h(x + he_i/2)} \sqrt{\max \big(u^h_x, u^h_{x+he_i} \big)},  \\
        \dnabla^h \hat{V}^h &\coloneq \frac{1}{2} \sum_{x\in\calT^h} \sum_{\vh\in\vdh} \Ind_{Q_h(x + \vh/2)} \frac{\vh}{h}  \frac{V(x+\vh) - V(x)}{h},
    \end{align*}
    that gives us $\calS^h_V = \|\dnabla^h \hat{\ell}^h + \hat{\xi}^h \|^2_{L^2}$.

    The compactness result in Lemma~\ref{lem:strong-compactness} implies that up to a subsequence $\hat{u}^h$ converges a.e. to $u$. Together with the equi-integrability of $\hat{u}^h$ (which comes from the superlinearity assumption on $f$), this implies $\hat{u}^h\to u$  strongly in $L^1(\Omega)$. This in turns implies $\sqrt{\hat{u}^h}\to \sqrt{u}$ strongly in $L^2(\Omega)$. The, we denote $v^h_i \coloneq \hat{u}^h (\cdot + he_i) = \sum_{x\in\calT^h}  \Ind_{Q_h(x)} u^h_{x+he_i}$. We claim that $\max (\sqrt{\hat{u}^h}, \sqrt{v^h_i}) \to \sqrt{u}$ strongly in $L^2$ for all $i=1,\dots,d$. Indeed,
    \begin{align*}
        \big\| \max (\sqrt{\hat{u}^h}, \sqrt{v^h_i}) &- \sqrt{u} \big\|_{L^2}
        \leq \big\| \sqrt{v^h_i} - \sqrt{\hat{u}^h} \big\|_{L^2} +  \big\|  \sqrt{\hat{u}^h} - \sqrt{u} \big\|_{L^2} \\
        &\leq \big\| \sqrt{\hat{u}^h} (\cdot + he_i) - \sqrt{u} (\cdot + he_i) \big\|_{L^2} + \big\| \sqrt{u} (\cdot + he_i) - \sqrt{u} \big\|_{L^2}  +  \big\|  \sqrt{\hat{u}^h} - \sqrt{ué} \big\|_{L^2} \\
        &= \big\| \sqrt{u} (\cdot + he_i) - \sqrt{u} \big\|_{L^2} +  \big\|  \sqrt{\hat{u}^h} - \sqrt{u} \big\|_{L^2} 
    \end{align*}
    where $\big\| \sqrt{u} (\cdot + he_i) - \sqrt{u} \big\|_{L^2} \to 0$ as $h\to 0$ holds by the density argument (\cite[Lemma~4.3]{brezis2010functional}). Thus, $\max (\sqrt{\hat{u}^h}, \sqrt{v^h_i}) \to \sqrt{u}$ strongly in $L^2$ and, consequently, $\hat{\mu}^h_i \to \sqrt{u}$ strongly in $L^2$ for all $i=1,\dots,d$. Using a similar argument as in the first part of the proof, one can show that $\dnabla^h \hat{V}^h \rightharpoonup \nabla V$ weakly in $L^\infty$. Therefore, $\hat{\xi}^h \rightharpoonup \sqrt{u} \nabla V$ weakly in $L^2$ and we conclude that we have
    $$
        \liminf_{h\to 0} S^V_h \geq \| \nabla \ell(u) + \sqrt{u} \nabla V \|_{L^2} ^2 = S^V.\qedhere
    $$
\end{proof}

\begin{remark}\label{remark:interaction-3}
    In continuation of Remarks~\ref{remark:interaction-1} and \ref{remark:interaction-2} on including the interaction potential $W$ instead of the external potential $V$, we note that the proof of Lemma~\ref{lem:convergence-disc-grad} can be adapted in the following way. One can show that 
    $$
        \widehat{(W*\rho)}^h \coloneq \sum_{x\in\calT^h} \Ind_{Q_h(x)} \sum_{z\in\calT^h} W(x - z)\rho^h_z \xrightarrow{h \to 0} W*\rho \quad \text{strongly in } L^1(\Omega), 
    $$
    using the regularity of $W$ and strong convergence $\hat{u}^h \to u$ in $L^1(\Omega)$. Secondly, one can also prove that the approximation for the gradient of the convolution
    $$
        \dnabla^h \widehat{(W*\rho)}^h \coloneq
        \frac{1}{2} \sum_{x\in\calT^h} \sum_{\vh\in\vdh} \Ind_{Q_h(x + \vh/2)} \frac{\vh}{h}  \frac{\sum_{z\in\calT^h} \big( W(x+\vh - z) - W(x - z) \big) \rho^h_z}{h}
    $$
    converges weakly in $L^\infty$ to $\nabla (W*\rho)$. This can be done by means of the approximate integration-by-parts formula as in the first part of the proof, where the key estimate is
    \begin{align*}
        \bigg| \int (\nabla\cdot \varphi) \widehat{(W*\rho)}^h &+ \int \varphi \cdot \dnabla^h \widehat{(W*\rho)}^h \bigg|
        \leq \frac{C_\varphi}{2} \sum_{x\in\calT^h} \sum_{\vh\in\vdh} \bigg| \sum_{z\in\calT^h} \big( W(x + \vh - z) - W(x - z) \big) \rho^h_z \bigg| h^d \\
        &\leq C_\varphi \Lip(W) |\Omega| d h \xrightarrow{h\to 0} 0.
    \end{align*}
\end{remark}

\bigskip

Finally, we have all the ingredients to summarize the proof of Theorem~\ref{th:main-convergence-diffusion}.
\begin{proof}
    We start with the variational inequality derived in Lemma~\ref{lem:var-ineq} 
    \begin{equation}\label{eq:var-ineq-proof}
        0 \geq \calF^h(\rho^{h,\tau}_N) - \calF^h(\rho^h_0) + \frac{1}{2}\int_0^T \| \check{v}^{h,\tau}_t \|^2_{L^2(\check{\rho}^{h,\tau}_t)} \dd t + \int_0^T \calS^h_V (\Tilde{\rho}^{h,\tau}_t) \lambda_\varepsilon^{h,\tau} (\dd t) + \frac{dT}{4} \frac{h}{\tau} \log \varepsilon + o(1)|_{h\to 0}.
    \end{equation}
    We aim to pass $h, \tau, h/\tau \to 0$ (we will simply write $h,\tau\to 0$ implying that $h/\tau \to 0$ holds) in all components of \eqref{eq:var-ineq-proof}. 
    
    We recall that the discrete energy can be written in the integral form as
    \begin{align*}
        \calF^h(\rho^{h,\tau}_T) = \int_\Omega f\bigg( \frac{\dd \hat{\rho}^{h,\tau}_T}{\dd\calL^d} \bigg) \dd x + \int_\Omega \hat{V}^h(x) \hat{\rho}^{h,\tau}_T (\dd x).
    \end{align*}
    Since $\rho^{h,\tau}_T \rightharpoonup \rho_T$ narrowly as $h,\tau\to 0$, then
    $$
        \liminf_{h,\tau\to 0} \int_\Omega f\bigg( \frac{\dd \hat{\rho}^{h,\tau}_T}{\dd\calL^d} \bigg) \dd x \geq \int_\Omega f\bigg( \frac{\dd \rho_T}{\dd\calL^d} \bigg) \dd x
    $$
    and
    \begin{align*}
        \lim_{h,\tau\to 0} \int_\Omega \hat{V}^h(x) \hat{\rho}^{h,\tau}_T (\dd x)
        &= \lim_{h,\tau\to 0} \int_\Omega V(x) \hat{\rho}^{h,\tau}_T (\dd x)
        + \lim_{h,\tau\to 0} \int_\Omega (\hat{V}^h(x) -V(x)) \hat{\rho}^{h,\tau}_T (\dd x) \\
        &= \int_\Omega V(x) \rho_T (\dd x) + \lim_{h\to 0} \Lip(V) h = \int_\Omega V(x) \rho_T (\dd x).
    \end{align*}
    Using also the assumption on the initial data, we get
    $$
        \liminf_{h,\tau\to 0} \big( \calF^h(\rho^{h,\tau}_T) - \calF^h(\rho^{h,\tau}_0) \big) \geq \calF(\rho_T) - \calF(\rho_0).
    $$

    For the Fisher information, we claim that, for any fixed $\varepsilon>0$, we have
    \begin{align*}
        \liminf_{h,\tau\to 0} \int_0^T \calS_V^h(\Tilde{\rho}^{h,\tau}_t) \lambda^{h,\tau}_\varepsilon(\dd t) \geq (1-\varepsilon)\int_0^T \calS_V(\rho_t) \dd t.
    \end{align*}
   This can be obtained by combining the two following facts:
   \begin{itemize}
       \item first, by Lemma~\ref{lem:compactness-curves} and Lemma~\ref{lem:convergence-disc-grad}, we have, for any sequence $t_n\to t$
    \begin{align*}
        \liminf_{h,\tau\to 0,n\to\infty} \calS^h(\Tilde{\rho}^{h,\tau}_{t_n}) 
        \geq \calS(\rho_t);
    \end{align*}
    \item second, we have the narrow convergence $\lambda^{h,\tau}_\varepsilon\rightharpoonup (1-\varepsilon)\mathcal{L}$ as $h,\tau\to 0$.
   \end{itemize}
   these two facts together allow to conclude via a standard argument proven, for instance, in \cite[Proposition 5.5]{BraBut}.
   


    By the lower semicontinuity of the Benamou-Brenier functional, we have
    \begin{align*}
        \liminf_{h,\tau\to 0} \| \check{v}^{h,\tau}_t \|^2_{L^2(\check{\rho}^{h,\tau}_t)} 
        \geq \| v_t \|^2_{L^2(\rho_t)},
    \end{align*}
    where we applied the convergence established in Lemma~\ref{lem:compactness-curves}.
    
    To summarize, in the limit $h,\tau\to 0$, we recover 
    \begin{align*}
        0 \geq \calF(\rho_T) - \calF(\rho_0) + \int_0^T \Big\{ \frac{1}{2} \int_\Omega |v_t|^2 \dd \rho_t + (1-\varepsilon)\calS(\rho_t) \Big\} \dd t,
    \end{align*}
    and the EDI inequality that we want to obtain as in Proposition~\ref{prop:charact} can be recovered taking $\varepsilon\to 0$.
\end{proof}

\section{Crowd motion model}\label{sec:crowd-motion}

In this section, we study the convergence of the fully discrete JKO scheme for the crowd motion model presented in Section~\ref{sec:crowd-motion-model}. We recall that the driving energy in this case becomes
\begin{equation}\label{eq:energy-infinity-disc}
    \Fcm^h(\rho^h) = \begin{cases}
        \displaystyle\sum_{x\in\calT^h} V(x) \rho^h_x, \qquad \text{if } u^h \leq 1 \\
        +\infty, \qquad \text{otherwise.}
    \end{cases} 
\end{equation}
for a probability measure $\rho^h = u^h h^d$ supported on the grid $\calT^h$. The fully discrete JKO scheme reads
\begin{definition}\label{eq:dJKO-crowd-motion}
    For a given $h>0$ and $\tau>0$ of the from $\tau=T/N$ for a fixed $T>0$ and some $N\in\N$, we define
    \begin{equation}\label{eq:JKO-crowd-disc} \tag{JKO$_\text{cm}^{h,\tau}$}
        \rho^{h,\tau}_{k+1} \in \arg\min_{\rho^h\in\calP(\calT^h)} \Big\{ \Fcm^h(\rho^h) + \frac{1}{2\tau} W_2^2(\rho^h, \rho_k^{h,\tau}) \Big\}
    \end{equation}
with a given initial datum $\rho^h_0\in \calP(\calT^h)$.
\end{definition}

The main result of this section is the following convergence result for \eqref{eq:JKO-crowd-disc}.
\begin{theorem}\label{th:main-result-cm-sec}
    Assume $V\in C^1(\R^d)$. Let $\{\rho^{h,\tau}_k\}_{k=1,\dots,N}$ be the family of \eqref{eq:JKO-crowd-disc} minimizers as in Definition~\ref{eq:dJKO-crowd-motion}.  Let the family of initial data $\{\rho^h_0\}_{h>0}$ be such that there exists $\rho_0 = u_0 \calL^d \in\calP(\Omega)$ with $\|u_0\|_{L^\infty} \leq 1$ such that
    $$
        \rho^h_0 \rightharpoonup \rho_0 \quad \text{narrowly as } h\to 0 \quad \text{and} \quad \lim_{h\to 0} \Fcm(\rho^h_0) = \calF_\text{CM}(\rho_0).
    $$

     Then a suitable interpolation in time of $\{\rho^{h,\tau}_k\}_{k=1,\dots,N}$ converges as $h, \tau, h/\tau \to 0$ uniformly in time for the $W_2$ distance to an absolutely continuous curve $[0,T] \ni t \mapsto \rho_t \in \calP(\Omega)$ such that $\rho_t$ is a distributional solution of \eqref{eq:crowd-motion-PDE}.
\end{theorem}

For what follows, it is convenient to define upwind-type mean:
\begin{equation}\label{eq:average-p-V}
        \Lambda_{p+V}(u_x, u_y) \coloneq \begin{cases}
            u_x, \qquad p_x + V(x) \geq p_y + V(y) \\
            u_y, \qquad p_x + V(x) < p_y + V(y) 
        \end{cases} 
        \qquad (x, y) \in\Sigma^h.
\end{equation}
We use the notation $\Lambda_V(u_x, u_y) $ if the average depends only on the potential $V$ ($p = 0$ in \eqref{eq:average-p-V}).

\begin{definition} Let $(u^h, p^h)$ be a density-pressure pair satisfying the relation $p^h_x (1 - u^h_x)$ for all $x\in\calT^h$. Given a smooth potential $V$, we define \emph{the discrete crowd-motion Fisher information} as 
    \begin{equation*}
        \Scm^h(u^h, p^h) = \frac{1}{4} \sum_{(x,y)\in\Sigma^h} \Big( (p^h_x - p^h_y)^2 + 2 (p^h_x - p^h_y) (V(x) - V(y)) + (V(x) - V(y))^2 \Lambda_V(u^h_x, u^h_y) \Big) h^{d-2}.
    \end{equation*}
\end{definition}

For $\varepsilon > 0$, we also introduce the truncation function $g_\varepsilon(s) \coloneq (s - \varepsilon)^+$, where $s^+$ denotes the positive part of $s$, e.i.\ $s^+ = s\Ind_{s>0}$. 

\begin{lemma}\label{lem:lower-bound-by-slope-CM} Let $\Fcm^h$ be defined as in \eqref{eq:energy-infinity-disc} with $V\in C^1(\R^d)$. Let $\rho^h_0\in \calP(\calT^h)$ be given and $\rho^{h,\tau}$ be a minimizer of \eqref{eq:JKO-crowd-disc} for some $h,\tau > 0$. Then the following lower bound holds
    $$
        \frac{1}{2\tau^2} W_2^2(\rho^{h,\tau},\rho^h_0)
        \geq \Big( 1 - \frac{h}{2\tau} \Big) \Scm^h(u^{h,\tau}, g_{\varepsilon(h)}(p^{h,\tau}) ) - \frac{dh}{4\tau} + o(1)|_{h\to 0},
    $$
    with $\varepsilon(h) = (\Lip(V) + 1) h$.
\end{lemma}
\begin{proof}
    There exists a Kantorovich potential $\varphi$ corresponding to the optimal transport from $\rho^{h,\tau}$ to $\rho^h_0$ such that
    \begin{equation*}
        \sum_{x\in\calT^h} \bigg( V(x) + \frac{\varphi_x}{\tau} \bigg) \rho^h_x \geq  \sum_{x\in\calT^h} \bigg( V(x) + \frac{\varphi_x}{\tau} \bigg) \rho^{h,\tau}_x \qquad \text{for all } \rho^h \in \calP(\calT^h),
    \end{equation*}
    by an argument as in \cite[Lemma~3.1]{maury2010macroscopic}. Thus, for a fixed $\varphi$, $\rho^{h,\tau} = u^{h,\tau} h^d$ solves the minimization problem
    $$
        \rho^{h,\tau} \in \arg\min_{\rho^h \in \calP(\calT^h), ~ u^h \leq 1} \sum_{x\in\calT^h} \bigg( V(x) + \frac{\varphi_x}{\tau} \bigg) \rho^h_x.
    $$
    The minimizer of this linear minimization problem with a linear constraint has to have the form
    $$
        u^{h,\tau}_x = \begin{cases}
            1, & V(x) + \varphi_x / \tau < c \\
            \in [0, 1], & V(x) + \varphi_x / \tau = c \\
            0, & V(x) + \varphi_x / \tau > c,
        \end{cases}
    $$
    where $c\in \R$ is chosen in the way to ensure $\sum_{x\in \calT^h} \rho^{h,\tau}_x = 1$.
    Employing the Kantorovich potential $\varphi$, one can define a pressure-like function $p$ (as in \cite[Lemma~3.3]{maury2010macroscopic}):
    \begin{equation}\label{eq:def-pressure}
        p \coloneq \bigg(c - V - \frac{\varphi}{\tau} \bigg)^+.
    \end{equation}
    
    By a similar arguments as in Lemma~\ref{lemma:lower-bound-by-slope}, one can derive the inequality:
    \begin{align*}
        \frac{1}{2\tau^2} W_2^2(\rho^{h,\tau},\rho^h_0)
        \geq \frac{1}{2} \Big( 1 - \frac{h}{2\tau} \Big)
        \sum_{(x,y)\in\Sigma^h} \frac{(p_x - p_y + V(x) - V(y))^2}{h^2} \Lambda_{p+V}(u_x, u_y) \, h^d - \frac{dh}{4\tau},
    \end{align*}
    where we use the notation $\Lambda_{p+V}(u_x, u_y)$ introduced in \eqref{eq:average-p-V} and $(u, p)$ is the density-flux pair corresponding to $\rho^{h,\tau}$, i.e. $\rho^{h,\tau} = u h^d$ and $p(1 - u) = 0$ on $\calT^h$.
    
    The approach of Corollary~\ref{cor:slope-bound-with-V} is not applicable in this case, because the proof relies on the strict monotonicity of $f'$ that the pressure variable does not satisfy. We provide a different argument exploiting the properties of the pressure. 

    Consider an edge $(x,y) \in \Sigma^h$ and assume (without loss of generality) that $p_x + V(x) > p_y + V(y)$. The contribution of this edge to the sum above is
    $$
        \frac{(p_x - p_y + V(x) - V(y))^2}{h^2} u_x h^d
        = \frac{(p_x - p_y)^2}{h^2} u_x h^d + 2\frac{(p_x - p_y) (V(x) - V(y))}{h^2} u_x h^d + \frac{(V(x) - V(y))^2}{h^2} u_x h^d.
    $$

    Fix $\varepsilon > 0$ such that $\varepsilon \geq (\Lip(V) + 1)h$. We claim that we have
    \begin{equation}\label{eq:ineq-P}
        \frac{(p_x - p_y)^2}{h^2} u_x h^d \geq \frac{(g_\varepsilon(p_x) - g_\varepsilon(p_y))^2}{h^2} h^d.
    \end{equation}
    To show \eqref{eq:ineq-P}, we first use that $g_\varepsilon$ is 1-Lipshitz:
    $$
        (p_x - p_y)^2 u_x  
        \geq (g_\varepsilon(p_x) - g_\varepsilon(p_y))^2 u_x, 
    $$
    and we can remove the coefficient $u_x$ from the right-hand side if we have $p_x>0$ because of the density-pressure relation $p_x (1 - u_x) = 0$. If, instead, we assume $p_x = 0$, then, by the assumption on the edge $(x,y)\in\Sigma^h$, we have $p_y < V(x) - V(y) \leq \Lip(V) h < \varepsilon$ and $g_\varepsilon(p_y) = 0$. Moreover, 
    $p_x=0$ implies $g_\varepsilon(p_x)=0$ so that we have
    $$
        (p_x - p_y)^2 u_x \geq (g_\varepsilon(p_x) - g_\varepsilon(p_y))^2=0
    $$
    and we get \eqref{eq:ineq-P}. Thus, summing up over all the edges yields
    \begin{equation}\label{eq:lower-bound-p}
        \sum_{(x,y)\in\Sigma^h} \frac{(p_x - p_y)^2}{h^2} \Lambda_{p+V}(u_x, u_y) h^d 
        \geq \frac{1}{2} \sum_{(x,y)\in\Sigma^h} \frac{(g_\varepsilon(p_x) - g_\varepsilon(p_y))^2}{h^2} h^d.
    \end{equation}

    Consider the cross-term:
    \begin{align*}
        (p_x - p_y)(V(x) - V(y)) \, u_x 
        = (p_x - p_y)(V(x) - V(y)) + (u_x - 1) (p_x - p_y)(V(x) - V(y)).
    \end{align*}
    If $p_x > 0$, then $u_x = 1$, and the last term equals 0. In the case when $p_x = p_y =0$, the last term vanishes as well. It is left to consider the case when $p_x = 0$ and $p_y > 0$. Since we assume that $p_x + V(x) > p_y + V(y)$, then $V(x) > V(y)$, and it follows that the last term is positive. Thus, we get the inequality
    \begin{align*}
        (p_x - p_y)(V(x) - V(y)) \, u_x &\geq (p_x - p_y)(V(x) - V(y)) \\
        &= (g_\varepsilon(p_x) - g_\varepsilon(p_y)) (V(x) - V(y))
        + (p_x - g_\varepsilon(p_x) - p_y + g_\varepsilon(p_y)) (V(x) - V(y)).
    \end{align*}

    Since the cross-term is symmetric in for edges $(x,y)\in\Sigma^h$, it will sum up as
    \begin{align*}
        \Big| \sum_{(x,y)\in\Sigma^h} (p_x - g_\varepsilon(p_x) - p_y + g_\varepsilon(p_y)) &(V(x) - V(y)) h^{d-2} \Big| \\
        &= 2 \Big| \sum_{x\in\calT^h} (p_x - g_\varepsilon(p_x)) \sum_{\vh \in \vdh} (V(x) - V(x+\vh)) h^{d-2} \Big| \\
        &\leq 2 \varepsilon \sum_{x\in\calT^h} \Big| \sum_{\vh \in \vdh} (V(x) - V(x+\vh)) \Big| h^{d-2}.
    \end{align*}    
    Since $V\in C^1(\R^d)$, then for any $x\in\calT^h$ away from the boundary the discrete divergence of the gradient is bounded as:
    \begin{align*}
         \Big| \sum_{\vh \in \vdh} (V(x+\vh) - V(x)) \Big|
         &= \Big| \sum_{i=1}^d ( V(x+he_i) + V(x+he_i) - 2 V(x) ) \Big| \\
         &\hspace{-1cm}= \Big| \sum_{i=1}^d \Big( \int_0^1 \nabla V (x + \lambda h e_i) \dd \lambda \cdot (h e_i) - \int_0^1 \nabla V (x + (1 - \lambda) h e_i) \dd \lambda \cdot (h e_i) \Big) \Big| \\
         &\hspace{-1cm}\leq \sum_{i=1}^d \int_0^1 \big| \nabla V (x + \lambda h e_i) - \nabla V (x + (1 - \lambda) h e_i) \big| \dd \lambda \, h
         \leq d \omega(h) h,
    \end{align*}
    where $\omega$ is the modulus of continuity of $\nabla V$. Therefore, sum up, we obtain 
    \begin{align*}
        \Big| \sum_{x\in\calT^h} (p_x - g_\varepsilon(p_x)) \sum_{\vh \in \vdh} (V(x) - V(x+\vh)) h^{d-2} \Big|
        \leq d \omega(h) \frac{\varepsilon}{h} \sum_{x\in\calT^h}  h^d = d |\Omega| \omega(h) \frac{\varepsilon}{h}.
    \end{align*}
    In total, the cross-term enjoys the following lower bound:
    \begin{align}\label{eq:lower-bound-pV}
        2 \sum_{(x,y)\in\Sigma^h} \frac{(p_x - p_y) ( V(x) - V(y))}{h^2} 
        &\Lambda_{p+V}(u_x, u_y) \, h^d \\
        &\geq \sum_{(x,y)\in\Sigma^h} \frac{(g_\varepsilon(p_x) - g_\varepsilon(p_y)) (V(x) - V(y))}{h^2} h^d
        + 2d|\Omega| \omega(h) \frac{\varepsilon }{h}. \notag
    \end{align}

    To rewrite the third term depending on the potential, we claim that $\Lambda_{p+V} (u_x, u_y) \geq \Lambda_{V} (u_x, u_y)$, where, analogously to \eqref{eq:average-p-V}, $\Lambda_{V} (u_x, u_y)$ is defined as
    \begin{equation*}
        \Lambda_{V}(u_x, u_y) \coloneq \begin{cases}
            u_x, \qquad V(x) \geq V(y) \\
            u_y, \qquad V(x) < V(y).
        \end{cases}
    \end{equation*}
    To prove the claim, we consider the cases. It is clear that if $p_x=p_y = 0$, then $\Lambda_{V}(u_x, u_y) = \Lambda_{p+V}(u_x, u_y)$. The equality also holds if $p_x > 0$ and $p_y > 0$ because it follows that $u_x = u_y = 1$ and $\Lambda_{V}(u_x, u_y) =1= \Lambda_{p+V}(u_x, u_y)$. The last case is (w.l.o.g.) $p_x > 0$ and $p_y = 0$. In this case, two options are possible: either $p_x + V(x) \geq p_y + V_y$ and $\Lambda_{p+V}(u_x, u_y) = u_x = 1 \geq \Lambda_{V}(u_x, u_y)$, or $p_x + V(x) < p_y + V_y$ implying that $V(x) < V_y$ and $\Lambda_{p+V}(u_x, u_y) = u_y =  \Lambda_{V}(u_x, u_y)$. Therefore,
    \begin{equation}\label{eq:lower-bound-V}
        \sum_{(x,y)\in\Sigma^h} \frac{(V(x) - V(y))^2}{h^2} \Lambda_{p+V}(u_x, u_y) h^d 
        \geq \sum_{(x,y)\in\Sigma^h} \frac{(V(x) - V(y))^2}{h^2} \Lambda_{V}(u_x, u_y) h^d.
    \end{equation}

    Combining \eqref{eq:lower-bound-p}, \eqref{eq:lower-bound-pV}, and \eqref{eq:lower-bound-V}, we obtain for any $\varepsilon \geq (\Lip(V) + 1)h$
    \begin{align*}
        \sum_{(x,y)\in\Sigma^h} \frac{(p_x - p_y + V(x) - V(y))^2}{h^2} \Lambda_{p+V}(u_x, u_y) \, h^d
        \geq \Scm^h(u, g_\varepsilon(p)) + 2d|\Omega| \omega(h) \frac{\varepsilon }{h}. 
    \end{align*}
    Choosing $\varepsilon = (\Lip(V) + 1)h$ provides that the error term $\sim \omega (h)$ converges to 0 as $h\to 0$.
\end{proof}

\begin{theorem}[Compactness]  \label{th:compactness-CM}
    Let $\rho_0\in\calP(\Omega)$. Let $\Tilde{\rho}^{h,\tau} = \Tilde{u}^{h,\tau} h^d$ be the variational interpolation from Definition~\ref{def:var-interpolation} for the minimizers of \eqref{eq:JKO-crowd-disc} and $\Tilde{p}^{h,\tau}$ be the corresponding pressure variable constructed in \eqref{eq:def-pressure}.

    Then there exists the density-pressure pair $(u,p)$ with $u \in L^1((0,T), L^1(\Omega))$ and $p \in L^2((0,T), H^1(\Omega))$ such that $p(1 - u) = 0$ and
    \begin{align*}
        &\Tilde{u}^{h,\tau}_t \rightharpoonup u_t \qquad \text{weakly-* in } L^\infty(\Omega) \text{ for } t\in (0,T), \\
        &\Tilde{p}^{h,\tau} \rightharpoonup p \qquad \text{weakly in } L^2((0,T), L^2(\Omega)), \\
        &\intbar_s^t \Tilde{p}^{h,\tau}_r \dd r \to \intbar_s^t p_r \dd r \qquad \text{strongly in } L^2(\Omega) ~\text{ for any } (s,t) \subset (0,T), \\
        &\dnabla^h g_{\varepsilon(h)} (\Tilde{p}^{h,\tau}) \rightharpoonup \nabla p \qquad \text{weakly in } L^2((0,T), L^2(\Omega)),
    \end{align*}
    where the discrete approximation for the gradient $\dnabla^h g_{\varepsilon(h)} (\Tilde{p}^{h,\tau})$ is defined as in \eqref{eq:disc-grad}.
\end{theorem}
\begin{proof}
    \emph{Step 1. Convergence of the density.} By Lemma~\ref{lem:compactness-curves}, $\Tilde{\rho}^{h,\tau} \to \rho$ uniformly for the $W_2$ distance, where $\rho\in C([0,T];\calP(\Omega))$. Since $\|\Tilde{u}^{h,\tau}\|_{L^\infty} \leq 1$, we have $\Tilde{u}^{h,\tau}_t \rightharpoonup u_t$ weakly-* in $L^\infty$ for $t\in (0,T)$. 
    
    \emph{Step 2. Convergence of the pressure.} From Lemma~\ref{lem:lower-bound-by-slope-CM}, we have
    $$
        \sup_{h,\tau>0, h/\tau < 1} \int_0^T \Scm^h(\Tilde{u}^{h,\tau}_t, g_{\varepsilon(h)}(\Tilde{p}^{h,\tau}_t) ) \dd t \leq C
    $$  
    with $C>0$ independent of $h$ and $\tau$ and $\varepsilon(h) = (\Lip(V) + 1)h $. By a similar argument as in Lemma~\ref{lem:V-doesnt-matter}, one can deduce that the previous bound implies
    $$
        \sup_{h,\tau>0, h/\tau < 1} \int_0^T \sum_{(x,y)\in\Sigma^h} \big(g_{\varepsilon(h)}(p^{h,\tau}_t(y)) - g_{\varepsilon(h)}(p^{h,\tau}_t(x)) \big)^2 \dd t \leq C,
    $$
    with possibly different constant $C>0$ but still independent of $h$ and $\tau$. We define
    $$
        g^{h,\tau} \coloneq \sum_{x\in\calT^h} g_{\varepsilon(h)} (\Tilde{p}^{h,\tau}(x)) \Ind_{Q_h(x)}.
    $$
    
     Following the proof of Lemma~\ref{lem:strong-compactness} with $g_{\varepsilon(h)}$ instead of $\hat{\ell}^h$, one can derive
    \begin{align*}
        \sup_{h>0} \int_0^T \|g^{h,\tau}_t - \overline{g}^{h,\tau} \|^2_{L^2} \dd t \leq C \quad \text{and} \quad
        \int_0^T \int_{\Omega_{|\eta|}} |g^{h,\tau}_t(z - \eta) - g^{h,\tau}_t (z) |^2 \dd z \dd t \leq C |\eta| \max(|\eta|, h),
    \end{align*}
    where $\displaystyle \overline{g}^{h,\tau} \coloneq \intbar_0^T \intbar_\Omega g^{h,\tau}_t (x) \dd x \dd t$. It is clear that $\overline{g}^{h,\tau} \geq 0$. We also claim that $\displaystyle\sup_{h,\tau>0} \overline{g}^{h,\tau} < \infty$. Suppose that there exists a subsequence such that $\overline{g}^{h,\tau} \to \infty$, then there exists a subsequence such that $g^{h,\tau}_t(x) \to \infty$ pointwise a.e. on $(0, T)\times\Omega$. Consequently, $\Tilde{p}^{h,\tau}_t(x) \to \infty$ pointwise a.e. on $(0, T)\times\Omega$ and $\Tilde{u}^{h,\tau}_t(x) \to 1$ pointwise a.e. on $\Omega$ for a.e. $t\in (0,T)$, which is a contradiction with $\int_\Omega \hat{u}^{h,\tau}_t(x) \dd x = 1$ given that $|\Omega| > 1$.

    Since $\overline{g}^{h,\tau}$ is uniformly bounded, we find that $g^{h,\tau}$ is uniformly bounded in $L^2((0,T),L^2(\Omega))$ and there exists $p\in L^2((0,T),L^2(\Omega))$ such that $g^{h,\tau} \rightharpoonup p$ weakly in $L^2((0,T),L^2(\Omega))$.

    Consider $g^{h,\tau}_{s,t} \coloneq \intbar_t^s g^{h,\tau}_r \dd r$ and $p_{s,t} = \intbar_t^s p_r \dd r$ for some fixed $(t,s)\subset (0,T)$. Of course we have $g^{h,\tau}_{s,t} \rightharpoonup p_{s,t}$ weakly in $L^2(\Omega)$.  Since $\| g^{h,\tau}_{s,t}(\cdot+\eta) - g^{h,\tau}_{s,t} \|_{L^2(\Omega_{|\eta|})} \leq \frac{C}{t-s} |\eta| \max(|\eta|, h)$ for any $\eta\in\R^d$, the Riesz-Fréchet-Kolmogorov theorem implies $g^{h,\tau}_{s,t} \to p_{s,t}$ strongly in $L^2(\Omega)$. Futhermore, we have $\Tilde{p}^{h,\tau}_{s,t} \to p_{s,t}$ strongly in $L^2(\Omega)$, because 
    \begin{equation}\label{eq:converg-av-pressure}
        \| \Tilde{p}^{h,\tau}_{s,t} - p_{s,t} \|_{L_2}
        \leq \| \Tilde{p}^{h,\tau}_{s,t} - g^{h,\tau}_{s,t} \|_{L_2} + \| g^{h,\tau}_{s,t} - p_{s,t} \|_{L_2}
        \leq \varepsilon(h) |\Omega| + \| g^{h,\tau}_{s,t} - p_{s,t} \|_{L_2} \to 0.
    \end{equation}
    
    
    \emph{Step 3. Convergence of the density-pressure relation.} The proof of this step is an adaptation of a similar argument used in \cite[Theorem 2.4]{maury2010macroscopic}. Since the density-pressure relation holds for all the discrete pairs with $h>0$, then
    \begin{align*}
        0 &= \intbar_{t}^{s} \int_\Omega \Tilde{p}^{h,\tau}_r(x) (1 - \Tilde{u}^{h,\tau}_r(x) ) \dd x \dd r \\
        &= \intbar_{t}^{s} \int_\Omega (m_\epsilon * \Tilde{p}^{h,\tau}_r) (x) (1 - \Tilde{u}^{h,\tau}_r(x) ) \dd x \dd r
        + \intbar_{t}^{s} \int_\Omega \big( \Tilde{p}^{h,\tau}_r(x) - (m_\epsilon * \Tilde{p}^{h,\tau}_r) (x) \big) (1 - \Tilde{u}^{h,\tau}_r(x) ) \dd x \dd r \\
        &\eqcolon I^{h,\tau}_\epsilon(t,s) + E^{h,\tau}_\epsilon(t,s),
    \end{align*}
    where $m_\epsilon$ is a standard mollifier. The second term is bounded by the H\"older inequality
    \begin{align*}
        |E^{h,\tau}_\epsilon(t,s)|
        &\leq \bigg( \intbar_{t}^{s} \| \Tilde{p}^{h,\tau}_r - (m_\epsilon * \Tilde{p}^{h,\tau}_r) \|_2^2 \dd r \bigg)^{1/2} \bigg( \intbar_{t}^{s} \| 1 - \Tilde{u}^{h,\tau}_r \|_2^2 \dd r \bigg)^{1/2} \\
        &\leq \sqrt{|\Omega|} \bigg( \intbar_{t}^{s} \| \Tilde{p}^{h,\tau}_r - (m_\epsilon * \Tilde{p}^{h,\tau}_r) \|_2^2 \dd r \bigg)^{1/2}.
    \end{align*}
    We apply the change of variables $\eta = x - y$ and Jensen's inequality to get
    \begin{align*}
        \intbar_{t}^{s} \| \Tilde{p}^{h,\tau}_r - (m_\epsilon * \Tilde{p}^{h,\tau}_r) \|_2^2 \dd r 
        &= \intbar_{t}^{s} \int_\Omega \bigg| \int_\Omega m_\epsilon(x - y) (\Tilde{p}^{h,\tau}_r(x) - \Tilde{p}^{h,\tau}_r(y)) \dd y \bigg|^2 \dd x \\
        &\leq  \intbar_{t}^{s} \int_\Omega \int_{\{\eta\in\R^d: x-\eta \in \Omega\}} m_\epsilon(\eta) | \Tilde{p}^{h,\tau}_r(x) - \Tilde{p}^{h,\tau}_r(x - \eta) |^2 \dd \eta  \dd x \dd r \\
        &=  \int_{\R^d}  m_\epsilon(\eta) \intbar_{t}^{s} \int_{\Omega_{|\eta|}} | \Tilde{p}^{h,\tau}_r(x) - \Tilde{p}^{h,\tau}_r(x - \eta)) |^2 \dd x \dd r \dd \eta  \\
        &\leq \int_{\R^d}  m_\epsilon(\eta) |\eta| \max(|\eta|, h) \dd \eta \leq \epsilon \max(\epsilon, h),
    \end{align*}
    where the second-to-last inequality follows from a similar argument as in Lemma~\ref{lem:strong-compactness} applied to $\intbar_t^s \Tilde{p}^{h,\tau}_r \dd r$. Thus, we obtain
    $$
        |E^{h,\tau}_\epsilon(t,s)| \leq \sqrt{|\Omega| \epsilon \max(\epsilon,h)}.
    $$

    For the first term, we have that
    \begin{align*}
        I^{h,\tau}_\epsilon(t,s)
        &=  \intbar_{t}^{s} \int_\Omega  (m_\epsilon * \Tilde{p}^{h,\tau}_r) (x)  (1 - \Tilde{u}^{h,\tau}_t(x) ) \dd x \dd r
        + \intbar_{t}^{s} \int_\Omega (m_\epsilon * \Tilde{p}^{h,\tau}_r) (x) (\Tilde{u}^{h,\tau}_t(x) - \Tilde{u}^{h,\tau}_r(x)) \dd x \dd r \\
        &\eqcolon A^{h,\tau}_\epsilon(t,s) + B^{h,\tau}_\epsilon(t,s)
    \end{align*}

    By \cite[Lemma~3.4]{maury2010macroscopic}, we have the bound
    \begin{align*}
        |B^{h,\tau}_\epsilon(t,s)| &\leq \intbar_{t}^{s} \| \nabla (m_\epsilon * \Tilde{p}^{h,\tau}_r) \|_2 W_2(\rho^{h,\tau}_t, \rho^{h,\tau}_r) \dd r 
        \leq C \sqrt{s-t} \intbar_{t}^{s} \| \nabla (m_\epsilon * \Tilde{p}^{h,\tau}_r) \|_2 \dd r \\ 
        &\leq C\bigg( \int_{t}^{s} \| \nabla (m_\epsilon * \Tilde{p}^{h,\tau}_r) \|_2^2 \dd r \bigg)^{1/2}.
    \end{align*}

   We now aim to pass to the limit $A^{h,\tau}_\epsilon(t,s) + B^{h,\tau}_\epsilon(t,s) + E^{h,\tau}_\epsilon(t,s)$ as $h, \tau \to 0$. By \eqref{eq:converg-av-pressure}, we get that
   \begin{align*}
       \intbar_t^s (m_\epsilon* \Tilde{p}^{h,\tau}_r)\dd r \to \intbar_t^s (m_\epsilon*p_r)\dd r \qquad \text{as } h,\tau \to 0 \text{ strongly in } L^2(\Omega).
   \end{align*}
   This convergence of convolution together with the weak-$*$ convergence $u^{h,\tau}_t \rightharpoonup u_t$ in $L^\infty(\Omega)$ for a.e. $t\in (0,T)$ imply
   \begin{align*}
       \lim_{h,\tau\to 0} A^{h,\tau}_\epsilon(t,s) = \int_\Omega  \intbar_{t}^{s} (m_\epsilon * p_r) (x) \dd r \, (1 - u_t(x) ) \dd x \eqcolon A_\epsilon(t,s).
   \end{align*}

    Since $\int_0^T \|\Tilde{p}^{h,\tau}_t\|^2_{L^2} \dd t$ is bounded, uniformly in $h$ and $\tau$, then $\int_0^T \|\nabla m_\epsilon * \Tilde{p}^{h,\tau}_t\|^2_{L^2} \dd t$ is also bounded uniformly in $h$ and $\tau$ (with a constant depending on $\epsilon$). Therefore, the function $\|\nabla m_\epsilon * \Tilde{p}^{h,\tau}_t\|^2_{L^2} \in L^1((0,T))$ converges weakly as $h,\tau\to 0$ to a measure $\mu\in\calM^+((0,T))$ and
    \begin{align*}
        \lim_{h\to 0, \tau\to 0} B^{h,\tau}_\epsilon(t,s) \leq C_\epsilon \sqrt{\mu([t,s])} \eqcolon B_\epsilon(t,s) .
    \end{align*}
    For $E^{h,\tau}_\epsilon(t,s) $, we simply get $\displaystyle\lim_{h\to 0, \tau\to 0} |E^{h,\tau}_\epsilon(t,s) | \leq \sqrt{|\Omega|} \epsilon \eqcolon E_\epsilon$.
    
    We now pass $s \to t$. For any Lebesgue point $t\in (0,T)$ of $p$, we get
    $$
        \lim_{s\to t} A_\epsilon(t,s) = \int_\Omega  (m_\epsilon * p_t) (x)  (1 - u_t(x) ) \dd x \eqcolon A_\epsilon
    $$
    and 
    $$
        \lim_{s\to t} |B_\epsilon(t,s)| \leq C_\epsilon \lim_{s\to t} \sqrt{\mu([t,s])} = 0 \qquad \text{for a.e. } t\in (0,T).
    $$

    Finally, as $\epsilon\to 0$, we obtain
    $$
        \lim_{\epsilon\to 0} \, A_\epsilon =\int_\Omega p_t(x) (1 - u_t(x)) \dd x
    $$
    and the error term $E_\epsilon$ vanishes. In total, we obtain
    $$
        \int_\Omega p_t(x) (1 - u_t(x)) \dd x = 0 \qquad \text{for almost all } t\in [0,T].
    $$

    \emph{Step 4. Convergence of the pressure gradient.}
    A small modification of the proof of Lemma~\ref{lem:convergence-disc-grad} yields
    \begin{align*}
        \int_0^T \| \dnabla^h g^{h,\tau}_t \|_{L^2}^2 \dd t = \int_0^T \Scm^h (\Tilde{u}^{h,\tau}_t, \Tilde{p}^{h,\tau}_t) \dd t.
    \end{align*}
    Since the right-hand side is uniformly bounded by $\sup_{h>0} \Fcm^h(\rho^h_0) < \infty$, there exists $\zeta \in L^2((0,T), L^2(\Omega))$ such that $\dnabla^h g^{h,\tau} \rightharpoonup \zeta$ weakly in $L^2((0,T), L^2(\Omega))$. Following similar lines as the proof of Lemma~\ref{lem:convergence-disc-grad}, we obtain for an arbitrary $\psi\in C_c^\infty((0,T)\times\Omega;\R^d)$ that
    \begin{align*}
        \bigg| \int_0^T \int_\Omega \big( \nabla_x \cdot \psi(x,t) \, g^{h,\tau}_t(x) + \psi(x,t) \cdot (\dnabla^h g^{h,\tau}_t) (x) \big) \dd x \dd t \bigg|
        \leq C h \bigg( \int_0^T \Scm^h (\Tilde{u}^{h,\tau}_t, \Tilde{p}^{h,\tau}_t) \dd t \bigg)^{1/2}.
    \end{align*}
    This gives an approximate integration-by-parts formula
    $$
        \int_0^T \int_\Omega \nabla_x \cdot \psi(x,t) \, g^{h,\tau}_t(x) \dd x \dd t = \int_0^T \int_\Omega \psi(x,t) \cdot (\dnabla^h g^{h,\tau}_t) (x) \dd x \dd t + O(h)
    $$
    and passing to the limit $h,\tau \to 0$ we obtain
    $$
        \int_0^T \int_\Omega \nabla_x \cdot \psi(x,t) \, p_t(x) \dd x \dd t = \int_0^T \int_\Omega \psi(x,t) \cdot \zeta_t (x) \dd x \dd t.
    $$
    Since $\psi$ is arbitrary, we have $\zeta = \nabla p$.
\end{proof}

\begin{lemma}\label{lem:liminf-Fisher-CM} Under the assumptions of Theorem~\ref{th:compactness-CM}, we get the following liminf inequlity for the Fisher information:
    $$
        \liminf_{h,\tau\to 0} \int_0^T \Scm^h \big( \Tilde{u}^{h,\tau}_t, g_{\varepsilon(h)} (\Tilde{p}^{h,\tau}_t) \big) \dd t \geq \int_0^T \Scm (u_t, p_t) \dd t.
    $$
\end{lemma}
\begin{proof} We consider the three part of $\Scm^h$ one by one. For the first part, following the lines of Lemma~\ref{lem:convergence-disc-grad}, one gets that
    \begin{align*}
        \liminf_{h,\tau \to 0} \int_0^T \frac{1}{4} \sum_{(x,y)\in\Sigma^h} \big( g_\varepsilon(\Tilde{p}^{h,\tau}_x(t)) - g_\varepsilon(\Tilde{p}^{h,\tau}_y(t)) \big)^2 h^{d-2} \dd t 
        &= \liminf_{h,\tau \to 0} \frac{1}{2} \int_0^T \|\dnabla^h g_\varepsilon(\Tilde{p}^{h,\tau}_t) \|_{L^2}^2 \dd t \\
        &\geq \frac{1}{2} \int_0^T \| \nabla p_t \|_{L^2}^2 \dd t.
    \end{align*}

    To handle the cross-term, we use that for all $x\in\calT^h$ and $\vh\in\vdh$
    $$
        V(x + \vh) - V(x) = (\nabla V)(x) \cdot \vh + o(h)
        = \intbar_{Q_h(x + \vh/2)} (\nabla V)(z) \dd z \cdot \vh + o(h).
    $$
    Thus, we have
    \begin{align*}
        \frac{1}{2} &\sum_{(x,y)\in\Sigma^h} \big( g_\varepsilon(\Tilde{p}^{h,\tau}_x) - g_\varepsilon(\Tilde{p}^{h,\tau}_y) \big) (V(x) - V(y)) h^{d-2} \\
        &= \frac{1}{2} \sum_{x\in\calT^h} h^d \sum_{\vh\in\vdh} \intbar_{Q_h(x + \vh/2)} \big( g_\varepsilon(\Tilde{p}^{h,\tau}_{x+\vh}) - g_\varepsilon(\Tilde{p}^{h,\tau}_x) \big) \frac{\vh}{h^2} \cdot (\nabla V)(z) \dd z + o(1)|_{h\to 0} \\
        &= \int_\Omega \dnabla^h \big( g_\varepsilon(\Tilde{p}^{h,\tau}) \big) (z) \cdot (\nabla V)(z) \dd z + o(1)|_{h\to 0}.
    \end{align*}
    Since $\dnabla^h g_\varepsilon(\Tilde{p}^{h,\tau}) \rightharpoonup \nabla p$ weakly in $L^2((0,T), L^2(\Omega))$, we obtain
    $$
        \lim_{h,\tau \to 0} \int_0^T \frac{1}{2} \sum_{(x,y)\in\Sigma^h} \big( g_\varepsilon(\Tilde{p}^{h,\tau}_x)(t) - g_\varepsilon(\Tilde{p}^{h,\tau}_y)(t) \big) (V(x) - V(y)) h^{d-2} \dd t
        = \int_0^T \int_\Omega \nabla p_t(z) \cdot \nabla V(z) \dd z \dd t.
    $$

    We now turn to the last term in the Fisher information. The same idea was used in \cite[Theorem~6.2]{hraivoronska2024variational}, and we briefly repeat it here for completeness. A simple rewriting gives
    \begin{align*}
        \frac{1}{4} \sum_{(x,y)\in\Sigma^h}  (V(x) - V(y))^2 \Lambda_V(u^h_x, u^h_y) h^{d-2}
        &= \frac{1}{2} \sum_{x\in\calT^h} \sum_{\vh\in\vdh} u^h_x \big((V(x) - V(x - \vh))^+\big)^2 h^{d-2}.
    \end{align*}
    Since the positive part is a Lipschitz continuous function and $V\in C^1$, we have
    $$
        (V(x) - V(x - \vh))^+ =  (\nabla V (x) \cdot \vh)^+ + o(h),
    $$
    therefore,
    \begin{align*}
        \frac{1}{2} \sum_{x\in\calT^h} \sum_{\vh\in\vdh} u^h_x &\big((V(x) - V(x - \vh))^+\big)^2 h^{d-2}
        = \frac{1}{2} \sum_{x\in\calT^h} \sum_{\vh\in\vdh} u^h_x \big((\nabla V (x))^+ \cdot \vh \big)^2 h^{d-2} + o(1)|_{h\to 0} \\
        &= \frac{1}{2} \sum_{x\in\calT^h} u^h_x \langle \nabla V(x), \sum_{\vh\in\vdh} \vh \otimes \vh \Tilde{\Ind}_{\nabla V(x) \cdot \vh}  \nabla V(x) \rangle h^{d-2} + o(1)|_{h\to 0},
    \end{align*}
    where we use the notation $\Tilde{\Ind}_{\nabla V(x) \cdot \vh} = \Ind \{\nabla V(x) \cdot \vh > 0\} +\frac12 \Ind \{\nabla V(x) \cdot \vh = 0\} $. The tensor can be rewritten as
    $$
        h^{d-2} \sum_{\vh\in\vdh} \vh \otimes \vh \Tilde{\Ind}_{\nabla V(x) \cdot \vh} = h^d \sum_{i=1}^d e_i \otimes e_i = h^d \Id.
    $$
    Therefore,
    \begin{align*}
        \frac{1}{2} \sum_{x\in\calT^h} \sum_{\vh\in\vdh} u^h_x \big((V(x) - V(x - \vh))^+\big)^2 h^{d-2}
        &= \frac{1}{2} \sum_{x\in\calT^h} u^h_x |\nabla V (x)|^2  h^d + o(1)|_{h\to 0} \\
        &= \frac{1}{2} \sum_{x\in\calT^h} \int_{Q_h(x)} |\nabla V (z)|^2 \hat{u}^h(z) \dd z + o(1)|_{h\to 0},
    \end{align*}
    with $\hat{u}^h$ being piecewise constant reconstruction on $\{Q_h(x)\}_{x\in\calT^h}$.
    \begin{align*}
        \liminf_{h,\tau\to 0} \int_0^T \frac{1}{4} \sum_{(x,y)\in\Sigma^h}  (V(x) - V(y))^2 \Lambda_V(\Tilde{u}^{h,\tau}_x(t), \Tilde{u}^{h,\tau}_y(t)) h^{d-2} \dd t
        &\geq \liminf_{h,\tau\to 0} \frac{1}{2} \int_0^T \int_\Omega |\nabla V|^2 \Tilde{u}^{h,\tau}_t \dd z \dd t \\
        &= \int_0^T \int_\Omega |\nabla V(z)|^2 u_t(z) \dd z \dd t.
    \end{align*}
\end{proof}

We conclude with the proof of Theorem~\ref{th:main-result-cm-sec}.
\begin{proof}
    Similarly to Lemma~\ref{lem:var-ineq}, we use one-step variational inequality for \eqref{eq:JKO-crowd-disc} as in \eqref{eq:disc-step-var-ineq}, the lower bound with the Fisher information from Lemma~\ref{lem:lower-bound-by-slope-CM}, geodesic interpolation $(\check{\rho}^{h,\tau}, \check{v}^{h,\tau})$, and time rescaling 
    to arrive at
    \begin{align*}
        0 \geq \Fcm^h(\rho^{h,\tau}_N) - \Fcm^h(\rho^h_0) + \frac{1}{2} \int_0^T \|\check{v}^{h,\tau}_t\|_{L^2(\check{\rho}^{h,\tau}_t)} \dd t 
        + \int_0^T \calS^h_{\text{CM}}(\Tilde{\rho}^{h,\tau}_t, ~ &g_{\varepsilon(h)}(\Tilde{p}^{h,\tau}_t) )\lambda_\epsilon^{h,\tau} (\dd t) \\
        &+ \frac{dT}{4} \frac{h}{\tau} \log \epsilon + o(1)_{h\to 0}.
    \end{align*}

\end{proof}

\bibliographystyle{abbrv}
\bibliography{ref}
\end{document}